\documentclass[11pt, a4paper, reqno]{amsart}

\usepackage{graphicx}
\usepackage{amsmath}
\usepackage{amssymb}
\usepackage{amsthm}
\usepackage[utf8]{inputenc}
\usepackage[T1]{fontenc}
\usepackage{amsfonts}
\usepackage{mathtools}
\usepackage{mathrsfs}
\usepackage{enumitem}
\usepackage{xfrac}
\usepackage{bbm}
\usepackage{xcolor}
\usepackage{lmodern}
\usepackage{xparse}
\usepackage[
backend=biber,
style=alphabetic,
sorting=nyt
]{biblatex}
\usepackage{geometry}
\usepackage{blindtext}
\usepackage{stmaryrd} 
\usepackage{url}
\usepackage{bbold}
\usepackage{tikz}\usetikzlibrary{arrows.meta, cd, decorations.markings, 3d, shapes.geometric}
\usepackage{pgfplots}\pgfplotsset{compat=1.18}
\usepackage{standalone}
\usepackage[hidelinks]{hyperref}
\usepackage[title]{appendix}
\usepackage{cleveref}
\usepackage{comment}

\title[]{Compactness of Moduli Spaces of Gradient Flow Lines in the Uniform Topology}
\author[Tom Stalljohann]{Tom Stalljohann$^\ast$}
\thanks{$\ast$ Universit\"at Heidelberg, \url{tstalljohann@mathi.uni-heidelberg.de}}
\date{2026}

\newtheorem{theorem}{Theorem}[section]
\newtheorem*{theorem*}{Theorem}
\newtheorem{MainThm}{Theorem}

\newtheorem{lemma}[theorem]{Lemma}
\newtheorem*{lemma*}{Lemma}
\newtheorem{corollary}[theorem]{Corollary}
\newtheorem*{corollary*}{Corollary}
\newtheorem{MainCor}[MainThm]{Corollary}
\newtheorem{proposition}[theorem]{Proposition}
\newtheorem*{proposition*}{Proposition}

\newtheorem*{claim}{Claim}

\theoremstyle{definition}
\newtheorem{definition}[theorem]{Definition}
\newtheorem*{definition*}{Definition}
\newtheorem*{convention}{Convention}

\theoremstyle{remark}
\newtheorem{remark}[theorem]{Remark}

\newtheorem{example}[theorem]{Example}

\newenvironment{proofClaim}{ \proof[Proof of the claim]}{\endproof}

\newcommand{\Forall}[0]{\forall\,}

\newcommand{\R}[0]{\mathbb{R}}

\newcommand{\N}[0]{\mathbb{N}}

\newcommand{\bbS}[0]{\mathbb{S}}

\newcommand{\eps}[0]{\varepsilon}
\newcommand{\comma}[0]{\, , \, }
\newcommand{\Cinfty}[0]{C^{\infty}}
\newcommand{\Cinftyloc}[0]{\Cinfty_{\mathrm{loc}}}

\newcommand{\ie}[0]{i.e.\ }

\newcommand{\cf}[0]{cf.\ }

\newcommand{\set}[2]{\left\{\,#1 \,\big| \, #2 \right\}}
\newcommand{\Bigset}[2]{\left\{\,#1 \,\Big| \, #2 \right\}}

\newcommand{\id}[0]{\mathrm{id}}

\newcommand{\Dcal}[0]{\mathcal{D}}
\newcommand{\Czeroloc}[0]{C^0_{\mathrm{loc}}}
\newcommand{\specgap}[0]{\mathfrak{S}^{+}_{f,G}(Z)}
\newcommand{\Crit}[0]{\mathrm{Crit}}
\newcommand{\moduli}[0]{\mathcal{M}}
\newcommand{\moduliData}[0]{\mathcal{M}_{(E_0,I)}}
\newcommand{\Rbar}[0]{\overline{\R}}

\newcommand{\modulihat}[0]{\widehat{\mathcal{M}}}
\newcommand{\BcalMorse}[0]{\mathcal{B}_{\delta}(Z_{-},Z_{+})}
\newcommand{\gaux}[0]{\Bar{g}}

\newcommand{\actionH}[0]{\mathcal{A}_{H}}
\newcommand{\Nloopspace}[0]{\mathcal{N}}
\newcommand{\Ucal}[0]{\mathcal{U}}
\newcommand{\Vcal}[0]{\Vcal{V}}

\newcommand{\loopspace}[0]{\mathscr{L}}

\newcommand{\shift}[1]{\mathsf{sh}^{#1}}
\newcommand{\Asf}[0]{\mathsf{A}}
\newcommand{\Asfxzerobar}[0]{\Asf_{\Bar{x}_0}}
\newcommand{\Psf}[0]{\mathsf{P}}
\newcommand{\Qsf}[0]{\mathsf{Q}}
\newcommand{\xzerobar}[0]{\Bar{x}_{0}}
\newcommand{\Ucalxzerobar}[0]{\Ucal(\xzerobar)}
\newcommand{\Lhat}[0]{\widehat{L}}
\newcommand{\Hhat}[0]{\widehat{H}}
\newcommand{\Asfhat}[0]{\widehat{\Asf}}
\newcommand{\Zhat}[0]{\widehat{Z}}
\newcommand{\Psfhat}[0]{\widehat{\Psf}}
\newcommand{\Qsfhat}[0]{\widehat{\Qsf}}

\numberwithin{equation}{section}

\begin{document}

\begin{abstract}
We prove a compactness result for gradient flow lines in a general set-up which comprises both the situation of Morse gradient flow lines as well as Floer cylinders converging to a critical submanifold respectively. For the compactness result we have to impose two conditions. Both are readily verified in the Morse case but establishing the second condition in the Floer case poses a technical challenge and relies on an exponential decay estimate for Floer cylinders, with coefficient function continuously depending on the initial loop. This is a result of independent interest. 
\end{abstract}

\maketitle

\section{Introduction}
\label{sec: Introduction}

\subsection{Main Compactness Result}
\label{subsec: Main Compactness Result}

\noindent Let $M$ be a metrizable topological space which we keep fixed throughout.
Suppose we are moreover given
\begin{enumerate}[label=(\arabic*)]
    \item continuous functions $f , G \in C^0(M,\R)$ with $G \geq 0 \,$,
    \item a compact subset $Z \subseteq G^{-1}(0)$ on which $f$ is constant with value $\zeta \in \R \,$.
\end{enumerate}

\begin{definition}[Spectral Gap]
\label{def: specgap}
The \textit{spectral gap $\specgap$ (for positive ends)} is
    \begin{align*}
    \specgap := \inf\Bigset{\zeta - f(x) }{\begin{array}{c} x \in G^{-1}(0) \backslash Z \\
    \text{with } f(x) \leq \zeta
    \end{array}} \in [0,\infty] \,\, .
    \end{align*}
    (We follow the convention that the infimum of the empty set is $+\infty\,$.)
\end{definition}

\noindent By a closed interval we mean an interval that is a closed subset of $\R$. For convenience, we do not consider single points to be intervals. That is, the closed intervals are of the form $[a,b] \,$, with $a < b \,$, or $(-\infty,a]$ or $[a,\infty)\,$, for $a \in \R \,$, or the entire real line $\R$.

\begin{definition}
\label{def: energy}
For $\gamma \in C^0(I,M) \,$, defined on a closed interval $I \subseteq \R\,$, its \textit{energy} $E(\gamma)$ is defined by
    \[
    E(\gamma) := \int_I G \circ \gamma(s)  \, ds \,\,\in [0,\infty] \,\, .
    \]
\end{definition}

\begin{definition}[Class of Gradient Flow Lines]
\label{def: class of gradient flow lines}

A \textit{class of $(f,G)$-gradient flow lines} (or just \textit{gradient flow line class}) is a family $\Gamma= \{\Gamma(I)\}_I$ of sets, indexed over the closed intervals $I \subseteq \R \,$, so that
\begin{enumerate}[label=(\arabic*)]
    \item\label{it: gradient flow line class - contained in Ccal(I)} $\Gamma(I) \subseteq C^0(I,M)$ for every $I \,$,
    \item\label{it: gradient flow line class - restriction} for every pair $J \subseteq I$ of closed intervals, the restriction $\gamma \mapsto \gamma|_J$ gives a well-defined map $\Gamma(I) \rightarrow \Gamma(J) \,$,
    \item\label{it: gradient flow line class - translations} $\Gamma$ is translation invariant, that is
    \[
    \set{\gamma(\,\cdot + s_0)}{\gamma \in \Gamma(I)} = \Gamma(I-s_0) \qquad \Forall I  \comma s_0 \in \R \,\, , 
    \]
\end{enumerate}
and so that moreover for every closed interval $I$ and every $\gamma \in \Gamma(I)$ it holds
\begin{enumerate}[label=(\arabic*), resume]
    \item\label{it: gradient flow line class - gradient flow line eq} $f \circ \gamma \in C^1(I,\R)$ and $(f \circ \gamma)'(s) = G \circ \gamma(s)$ for all $s \in I \,$.
\end{enumerate}
An element $\gamma \in \Gamma(I)$ is called \textit{$\Gamma$-gradient flow line on $I$} (or just \textit{gradient flow line} if $\Gamma$ is understood).
\end{definition}

\begin{remark}
\label{rmk: gradient flow line}
Given a $\Gamma$-gradient flow line $\gamma : I \rightarrow M \,$.
\begin{enumerate}[label=(\roman*)]
    \item\label{it: gradient flow line - decreasing} Since $G$ is nonnegative, the function $f \circ \gamma$ is increasing.
    \item\label{it: gradient flow line - constant} Let $J \subseteq I$ be a closed subinterval. Then $f \circ \gamma$ is constant on $J$ if and only if $\gamma(J) \subseteq G^{-1}(0) \,$.
    \item\label{it: gradient flow line - energy} Suppose $I = [a,b]$ is compact.
    The energy of $\gamma$ is 
    \[
    E(\gamma) = \int_a^b (f \circ \gamma)'(s) \, ds = f\circ \gamma(b)- f \circ \gamma(a) \geq 0 \,\, .
    \]
    If $I = [a,\infty)$ is unbounded, then $E(\gamma) = \lim_{s \rightarrow \infty} f \circ \gamma(s) - f \circ \gamma(a) \,$, where $\lim_{s\rightarrow\infty} f \circ \gamma(s)$ may be $+\infty \,$.
\end{enumerate}
\end{remark}

\begin{example}
\label{ex: gradient flow line}
\begin{enumerate}[label=(\roman*)]
    \item\label{it: example gradient flow line - entire class}
    For every closed interval $I \,$, let $\Gamma(I)$ be the set consisting of those $\gamma \in C^0(I,M)$ for which \ref{it: gradient flow line class - gradient flow line eq} in Definition \ref{def: class of gradient flow lines} holds. Then $\Gamma $ is a gradient flow line class.
   \item\label{it: example gradient flow line - restricting to subset} Let $\Gamma$ be a gradient flow line class and $S \subseteq M$ be a subset, endowed with the subspace topology. For every closed interval $I$ define
   \[
   \Gamma_S(I) := \set{\gamma \in \Gamma(I)}{\gamma(I) \subseteq S} \,\, .
   \]
   Then $\Gamma_S$ is both a class of $(f,G)$- and of $(f|_{S}, G|_S)$-gradient flow lines.
   \item Let $(M,g)$ be a Riemannian manifold and $f \in \Cinfty(M)$ be a smooth function. Let  $G(x) := |\nabla f(x)|^2$ be the square of the norm of the gradient, so that $G^{-1}(0) = \Crit(f) \,$. In particular, $f$ is constant on $Z \subseteq \Crit(f)$ if $Z$ is connected. Now let $\Gamma(I)$ be the set of gradient flow lines of $f$, that is
    \[
    \Gamma(I) := \set{\gamma \in \Cinfty(I,M)}{\gamma' = \nabla f\circ \gamma} \,\, .
    \]
    Then $\Gamma$ is a class of $(f,G)$-gradient flow lines since
    \[
    (f \circ \gamma)'(s) = |\nabla f \circ \gamma(s) |^2 = |\gamma'(s)|^2 \qquad \Forall s \in I \comma \gamma \in \Gamma(I) \,\, .
    \]
    The energy of $\gamma \in \Gamma(I)$ is the usual energy $E(\gamma) = \int_I |\gamma'(s)|^2 \, ds \,$.
    \item Let $(W,\omega)$ be a symplectic manifold that is symplectically aspherical. The space of contractible loops $M := \Cinfty_{\mathrm{contr}}(\bbS^1,W)$ with the $\Cinfty$-topology is metrizable. Given a Hamiltonian $H : W \times \bbS^1 \rightarrow \R \,$, consider the Hamiltonian action functional $f := \mathcal{A}_H : M \rightarrow \R$ defined via
    \[
    \mathcal{A}_H(x) := \int_{\mathbbm{D}^2} \overline{x}^*\omega - \int_0^1 H_t(x(t)) \, dt \,\, ,
    \]
where $\overline{x} : \mathbbm{D}^2 \rightarrow W$ is an arbitrary choice of filling disk, \ie $\overline{x}|_{\bbS^1} = x \,$. A fixed one-periodic family $J=(J_t)_{t \in \bbS^1}$ of $\omega$-compatible almost complex structures on $W$ induces an $L^2$-metric on $M$, which at $x \in M$ is given by
\[
\langle \xi_1, \xi_2 \rangle_{x} :=  \int_0^1 \omega_{x}(\xi_1, J_t(x) \, \xi_2) \, dt \qquad \Forall \xi_1,\xi_2 \in \Cinfty(\bbS^1,x^*TW) \,\, .
\]
The gradient of $\mathcal{A}_H$ at $x$ with respect to this metric is
\[
\nabla \mathcal{A}_H (x) = - J_t (x) \, \big(\dot x  - X_{H_t} \circ x \big) \in \Cinfty(\bbS^1,x^*TW) \,\, ,
\]
where we use the sign convention $\omega(\,\cdot\comma X_{H_t}) = d H_t$ for the Hamiltonian vector field $X_{H_t} \,$.
Let $I \subseteq \R$ be a closed interval and identify a map $I \rightarrow M$ with a cylinder $I \times \bbS^1 \rightarrow W \,$. Then the gradient flow lines of $\mathcal{A}_H$ on $I$ are precisely the solutions $u \in \Cinfty(I \times \bbS, W)$ of the Floer equation
\begin{equation}
\label{eq: Floer}
\partial_s u + J_t(u) \, (\partial_t u - X_{H_t}(u)) = 0 
\end{equation}
for which $u(s,\,\cdot\,)$ is contractible for some (and hence every) $s \in I \,$.
Now define $G(x) := |\nabla \mathcal{A}_H(x)|_{L^2}^2 \,$. Then
\begin{align*}
\Gamma(I) := \Bigset{\gamma : I \rightarrow M}{ \begin{array}{c}
   u_\gamma : I \times \bbS^1 \rightarrow W \text{ satisfies Eq.~\eqref{eq: Floer},}   \\ 
  \text{where }  u_\gamma(s,t) := \gamma(s)(t) 
\end{array}} 
\end{align*}
defines a class of $(f,G)$-gradient flow lines. The energy of $\gamma \in \Gamma(I)$ is the usual energy $E(u_\gamma) = \int_{I \times \bbS^1} |\partial_s u_\gamma|^2_{J_t} \, ds \, dt$ considered in Floer theory.
\end{enumerate}
\end{example}

\noindent Back to the general setting, for $E_0 \geq 0$ and an unbounded interval $I = [a,\infty)$ we define the moduli space 
\begin{align}
\label{eq: moduli space def}
\moduli_{(E_0,I)} := \moduli_{(E_0,I)}(\Gamma,Z) := \Bigset{\gamma \in \Gamma(I)}{\begin{array}{c}
    E(\gamma) \leq E_0 \text{ and}  \\
   \text{the limit } \gamma(\infty) \in Z \text{ exists} 
\end{array}} 
\end{align}
of $\Gamma$-gradient flow lines on $I$ of energy bounded by $E_0$ and converging to a point in $Z$. In Section \ref{subsec: Assumption A1 and A2} we introduce two conditions \ref{it: ass A - Czeroloc convergence} and \ref{it: ass A - shortening gradient flow lines}. Our main theorem then reads as follows.

\begin{MainThm}
\label{mthm: Compactness moduli spaces}
Suppose $0 \leq E_0 < \specgap$ and $I= [a,\infty)$ satisfy \ref{it: ass A - Czeroloc convergence} and \ref{it: ass A - shortening gradient flow lines}. Then for every sequence $(\gamma_n)_{n \in \N} \subseteq \moduliData$ there exists a subsequence which converges in $C^0(I,M)$ to some $\gamma \in \moduliData \,$.
\end{MainThm}

\noindent The uniform convergence in $C^0(I,M)$ is independent of the used metric on $M$ (see Definition \ref{def: uniform convergence}). Observe that $\moduliData$ embeds naturally into $C^0([a,\infty],M)$ and so it makes sense to speak of compactness of $\moduliData$ with respect to the compact-open topology on $C^0([a,\infty],M)$.
Since $M$ is metrizable, the compact-open topology on $C^0([a,\infty],M)$ is metrizable. So the next corollary is just a reformulation of Theorem \ref{mthm: Compactness moduli spaces}.

\begin{MainCor}
\label{mcor: Compactness moduli spaces}
Suppose $0 \leq E_0 < \specgap $ and $I= [a,\infty)$ satisfy \ref{it: ass A - Czeroloc convergence} and \ref{it: ass A - shortening gradient flow lines}. Then $\moduliData$ is compact with respect to the compact-open topology on $C^0([a,\infty],M) \,$.
\end{MainCor}

\subsection{Motivation and Outlook}
\label{subsec: Motivation and Outlook}

\noindent The above compactness theorem is the result of the author's attempts to adapt a homotopy stretching argument for Morse-Bott gradient flow lines by Albers-Hein \cite{Albers_Hein} to the setting in which the Rabinowitz action functional takes the place of the Morse-Bott function. Such a homotopy stretching argument relies fundamentally on compactness of a suitably defined moduli space, which in the Morse-Bott case is established by Albers-Hein in \cite[Thm.~3.7]{Albers_Hein}.

Trying to adapt a Morse theoretic proof for compact moduli spaces to Floer theory naturally led to the general framework of gradient flow line classes, which encode both classical gradient flow lines as well as Floer cylinders. Of course, a compactness result as stated in Theorem \ref{mthm: Compactness moduli spaces} in such a general set-up needs somewhat strong assumptions to begin with. These are encapsulated in assumptions \ref{it: ass A - Czeroloc convergence} and \ref{it: ass A - shortening gradient flow lines}. To be able to apply Theorem \ref{mthm: Compactness moduli spaces} both in the (finite-dimensional) Morse and the (infinite-dimensional) Floer setting, we have to verify those two assumptions in both settings respectively. 

Broadly speaking, assumption \ref{it: ass A - Czeroloc convergence} is ensured by the existence of a  $\Cinftyloc$-convergent subsequence of a sequence of Morse trajectories respectively Floer cylinders. In the Morse case this is by a bootstrapping argument and in the Floer case by Gromov compactness (which is standard, although nontrivial). The actual crux to apply Theorem \ref{mthm: Compactness moduli spaces} is the second condition \ref{it: ass A - shortening gradient flow lines}. We establish it for Morse-Bott functions, by quoting an exponential decay result for gradient flow lines of Morse-Bott functions. Importantly, we need an exponential decay estimate in which the coefficient in front of the exponential depends continuously on the initial point of the gradient flow line. This unsurprisingly holds true in the Morse-Bott situation because we are working with the flow of a vector field. 

Here is the crucial point: To establish \ref{it: ass A - shortening gradient flow lines} in the Floer case, under a standard Morse-Bott assumption on the Hamiltonian (see assumption \ref{it: MB}), one again wants an exponential decay estimate for Floer cylinders with coefficient function in front of the exponential map depending continuously on the initial point of the cylinder in the loopspace. Since the $L^2$-gradient of the action functional is not an actual vector field, such a result is far less obvious than in the Morse case. It still holds true though and this will be our second main result, stated in Theorem \ref{mthm: Floer - exponential decay}.

For the sake of brevity, in this work we restrict the verification of the assumptions \ref{it: ass A - Czeroloc convergence} and \ref{it: ass A - shortening gradient flow lines} and the application of Theorem \ref{mthm: Compactness moduli spaces} to the classical Hamiltonian action functional. However, one can analogously proceed with the Rabinowitz action functional instead. As mentioned above, in an upcoming work, we will be concerned with a homotopy stretching argument for the latter functional and, to this end, will make use of Theorem \ref{mthm: Compactness moduli spaces}.

\subsection{Organization of the Paper}
\label{subsec: Organization Paper}

\noindent In Subsection \ref{subsec: Uniform convergence and Coarse Metrics} we review the uniform convergence appearing in Theorem \ref{mthm: Compactness moduli spaces} and define the notion of a \textit{coarse metric}. The latter will reappear in Subsection \ref{subsec: Assumption A1 and A2}, in which we introduce the crucial assumptions \ref{it: ass A - Czeroloc convergence} and \ref{it: ass A - shortening gradient flow lines} for Theorem \ref{mthm: Compactness moduli spaces}. Because \ref{it: ass A - shortening gradient flow lines} is a quite strong assumption, we state a sufficient condition for it in Subsection \ref{subsec: Criterion for ass A - shortening gradient flow lines}. The proof of Theorem \ref{mthm: Compactness moduli spaces} will then be given in Subsection \ref{subsec: Proof of Compactness Result for moduli spaces}. To conclude this abstract general section, we present a slight generalization (Theorem \ref{thm: Compactness moduli spaces - variant}) of Theorem \ref{mthm: Compactness moduli spaces} in Subsection \ref{subsec: Variants of Compactness Thm}, which often is needed when working in the Floer setting.

Section \ref{sec: Morse - Illustration} deals with the finite-dimensional Morse-Bott case on a Riemannian manifold. Both conditions \ref{it: ass A - Czeroloc convergence} and \ref{it: ass A - shortening gradient flow lines} can be checked in a standard way, as we explain in Subsection \ref{subsec: Morse-Bott functions (finite dim)}, so that Theorem \ref{mthm: Compactness moduli spaces} immediately reproves a well-known compactness result about gradient flow lines converging to a critical submanifold of the Morse-Bott function (Corollary \ref{cor: Morse setting - compactness moduli space}). Often the compactness of a moduli space of interest is meant with respect to the topology of an ambient Banach manifold of loops. In Subsection \ref{subsec: Morse setting - ambient Banach mfd} we illustrate how one easily obtains such compactness from the $C^0$-compactness in Theorem \ref{mthm: Compactness moduli spaces} and exponential decay (which is needed to ensure condition \ref{it: ass A - shortening gradient flow lines} anyway). The section is concluded in Subsection \ref{subsec: Optimality of Upper Bound for Energy} with an example showing that the upper bound for the energy in Theorem \ref{mthm: Compactness moduli spaces} is optimal.

Section \ref{sec: Floer - Illustration} is concerned with illustrating Theorem \ref{mthm: Compactness moduli spaces} in the Floer case. Subsection \ref{subsec: Floer - MB set-up for ass A shortening flow lines} introduces the general setting, the relevant Morse-Bott assumption \ref{it: MB} and states Theorem \ref{mthm: Floer - exponential decay}. In Subsection \ref{subsec: Floer - symplectic homology setting}, we then apply Theorem \ref{mthm: Floer - exponential decay} to verify that conditions \ref{it: ass A - Czeroloc convergence} and \ref{it: ass A - shortening gradient flow lines} in the natural set-up for symplectic homology are satisfied. This will allow us to conclude the compactness result in Theorem \ref{thm: Floer - compactness result symplectic homology}.

The proof of Theorem \ref{mthm: Floer - exponential decay} is extremely technical and for this reason deferred to Section \ref{sec: Exponential Decay - proof}. It comprises a substantial part of the paper.

\subsection{Acknowledgments}
\label{subsec: Acknowledgments}

Foremost, I am indebted to my supervisor, P.~Albers, who originally got me in touch with the intriguing homotopy stretching argument which motivated the present compactness result. Second, I would like to express my deep gratitude to R.~Siefring for pointing out and kindly explaining to me in great detail and outstanding clarity the diagonalization argument for exponential decay of higher order derivatives of Floer cylinders. I would also like to thank both L.~Dahinden, A.~Fauck and J.~Wi\'sniewska for helpful discussions.

\section{The Compactness Result for Moduli Spaces}
\label{sec: Compactness Result for Moduli Spaces}

\subsection{Uniform convergence and Coarse Metrics}
\label{subsec: Uniform convergence and Coarse Metrics}
 Let $M$ be a metrizable topological space.
 Given a closed interval $I \subseteq \R \,$, we denote by $\Czeroloc(I,M)$ the topological space of continuous functions $C^0(I,M)$ endowed with the compact-open topology.

\begin{definition}[Uniform convergence]
\label{def: uniform convergence}
Given an unbounded interval $I = [a,\infty)$ and continuous functions $\gamma_n, \gamma : I \rightarrow M \comma n \in \N \,$, for which the limits $\gamma_n(\infty) := \lim_{s \rightarrow \infty} \gamma_n(s)$ and $\gamma(\infty)$ in $M$ exist. Let $\bar{\gamma}_n , \bar{\gamma} : [a,\infty] \rightarrow M$ be the unique continuous extensions of $\gamma_n$ respectively $\gamma \,$.
We define
\begin{align*}
    \gamma_n \overset{n}{\rightarrow} \gamma \text{ in } C^0(I,M) \quad : \Longleftrightarrow \quad \bar{\gamma}_n \overset{n}{\rightarrow} \bar{\gamma} \text{ in } C^0([a,\infty],M) \,\, .
\end{align*}
Here $C^0([a,\infty],M)$ is endowed with the compact-open topology.
\end{definition}

\noindent The above definitions are independent of any metric on $M$. However, the presence of a metric makes matters more explicit, as the next remark shows.

\begin{remark}
\label{rmk: C^0 uniform convergence}
Let $d_M$ be a metric inducing the topology on $M$ and $I = [a,\infty)$ be an unbounded interval.
\begin{enumerate}[label=(\roman*)]
    \item $\gamma_n \overset{n}{\rightarrow} \gamma$ in $\Czeroloc(I,M)$ if and only if $\gamma_n$ converges uniformly to $\gamma$ (with respect to $d_M$) on every compact subinterval of $I$.
    \item Suppose the limits $\gamma_n(\infty)$ and $\gamma(\infty)$ exist. Then $\gamma_n \overset{n}{\rightarrow} \gamma$ in $C^0(I,M)$ if and only if $\gamma_n$ converges uniformly to $\gamma$ (with respect to $d_M$) on $I$.
\end{enumerate}
\end{remark}

\noindent Recall that a topology $\mathcal{O}_1$ is called \textit{coarser} than a topology $\mathcal{O}_2$ if $\mathcal{O}_1 \subseteq \mathcal{O}_2 \,$.

\begin{definition}[Coarse Metric]
\label{def: coarse metric}
A \textit{coarse metric} $d : M \times M \rightarrow [0,\infty)$ on $M$ is a metric whose induced topology is coarser than the given topology on $M$. 
\end{definition}

\begin{convention}
\label{convention: coarse metric} When working with a coarse metric, all topological notions on $M$ (like openness, continuity, convergence etc.) without further qualification refer to the given (original) topology on $M$.
\end{convention}

\begin{remark}
\label{rmk: coarse metric}
\begin{enumerate}[label=(\roman*)]
\item Let $S \subseteq M$ be a subset endowed with the subspace topology. If $d$ is a coarse metric on $M$, then $d|_{S \times S}$ is a coarse metric on $S$.
\item A metric $d$ on $M$ is a coarse metric on $M$ if and only if for every $x \in M$ and every sequence $(x_n)_n \subseteq M$
\begin{align*}
    x_n \overset{n}{\rightarrow} x \text{ in the given topology on } M \quad \Longrightarrow \quad d(x_n,x) \overset{n}{\rightarrow} 0 \,\, .
\end{align*}
The ``if''-direction uses that $M$ is first-countable.
\end{enumerate}
\end{remark}

\subsection{Conditions \ref{it: ass A - Czeroloc convergence} and \ref{it: ass A - shortening gradient flow lines}}
\label{subsec: Assumption A1 and A2}

Suppose we are in the situation of Subsection \ref{subsec: Main Compactness Result}. That is we have fixed a metrizable topological space $M$ together with continuous functions $f : M \rightarrow \R$ and $G : M \rightarrow [0,\infty)$ and moreover a compact subset $Z \subseteq G^{-1}(0) $ on which $f|_Z \equiv \zeta \,$. Moreover, we fix a class $\Gamma$ of $(f,G)$-gradient flow lines. In the following, by a gradient flow line we always mean a $\Gamma$-gradient flow line. For $E_0 \geq 0$ and an unbounded interval $I = [a,\infty) \,$, recall from \eqref{eq: moduli space def} the definition of the moduli space
\begin{align*}
\moduli_{(E_0,I)} := \moduli_{(E_0,I)}(\Gamma,Z) := \Bigset{\gamma \in \Gamma(I)}{\begin{array}{c}
    E(\gamma) \leq E_0 \text{ and}  \\
   \text{the limit } \gamma(\infty) \in Z \text{ exists} 
\end{array}} 
\end{align*}
For $E_0 \geq 0$ and $I=[a,\infty)$ consider the following two assumptions:

{
\let\realItem\item 
\makeatletter
\NewDocumentCommand\myItem{ o }{%
   \IfNoValueTF{#1}%
      {\realItem}
      {\realItem[#1]\def\@currentlabel{#1}}
}
\makeatother

\setlist[enumerate]{
    before=\let\item\myItem,       
    label=\textnormal{(\arabic*)}, 
    widest=(2')                    
}

\begin{enumerate}
    \item[(A1$_{(E_0,I)}$)]\label{it: ass A - Czeroloc convergence} For every sequence $(\gamma_n)_{n \in \N} \subseteq \Gamma(I)$ of gradient flow lines with energy $E(\gamma_n) \leq E_0$ for every $n \in \N \,$, there exists a subsequence $(\gamma_{n_k})_{k \in \N}$ converging in $\Czeroloc(I,M)$ to a gradient flow line $\gamma \in \Gamma(I) \,$.
    \item[(A2$_{(E_0,I)}$)]\label{it: ass A - shortening gradient flow lines} For every $z \in Z$ there exist an open neighborhood $U \subseteq M$ of $z$ and a coarse metric $d : U \times U \rightarrow [0,\infty)$ on $U$ with the following significance:\\
    Let $(s_n)_{n \in \N} \subseteq I$ be a sequence tending to infinity. 
    Let $(\gamma_n)_{n \in \N} \subseteq \moduliData$ be a sequence in the moduli space with $\gamma_n(s) \in U$ for every $s_n \leq s \leq \infty $ and $n \in \N$. Then
    \begin{align*}
       (\gamma_n(s_n))_n \text{ converges to some point in } Z \cap U \,\,\, \Longrightarrow \,\,\, d(\gamma_n(s_n), \gamma_n(\infty)) \overset{n}{\rightarrow} 0 \,\, . 
    \end{align*}
\end{enumerate}
}

\noindent We remind the reader of our stipulated convention that, in the presence of a coarse metric, all topological notions without further qualification are with respect to the original topology. So, in \ref{it: ass A - shortening gradient flow lines}, $U$ being open and the limit $\lim_{n \rightarrow \infty} \gamma_n(s_n)$ are both with respect to the original topology on $M$.
In light of the above and the definition of uniform convergence (Definition \ref{def: uniform convergence}), we now repeat our main theorem.

{ \theoremstyle{theorem}
\newtheorem*{MainThmA}{Theorem \ref{mthm: Compactness moduli spaces}}
\begin{MainThmA}
Suppose $0 \leq E_0 < \specgap$ and $I= [a,\infty)$ satisfy \ref{it: ass A - Czeroloc convergence} and \ref{it: ass A - shortening gradient flow lines}. Then for every sequence $(\gamma_n)_{n \in \N} \subseteq \moduliData$ there exists a subsequence which converges in $C^0(I,M)$ to some $\gamma \in \moduliData \,$.
\end{MainThmA}
}

\noindent By assumption \ref{it: ass A - Czeroloc convergence}, every sequence of gradient flow lines in $\moduliData$ converges up to a subsequence in $\Czeroloc(I,M)$ to some gradient flow line in $\Gamma(I)$. So the non-trivial part of the theorem is the assertion that the limit again lies in $\moduliData$ and that the convergence is uniform on $I$.
\\

\subsection{Sufficient Condition for \ref{it: ass A - shortening gradient flow lines}}
\label{subsec: Criterion for ass A - shortening gradient flow lines}

\noindent The next, essentially trivial, lemma gives a sufficient condition for \ref{it: ass A - shortening gradient flow lines}.

\begin{lemma}
\label{lem: shortening gradient flow lines criterion}
Let $I = [a,\infty)$ and $E_0 \geq 0$ be fixed.
Suppose for every $z \in Z$ there exist an open neighborhood $U \subseteq M$ of $z$, a coarse metric $d$ on $U$ and a continuous function $\Xi : U \rightarrow \R $ vanishing on $Z \cap U$ with the following significance:
For every interval $J=[b,\infty) \subseteq I$ and $\gamma \in \moduli_{(E_0,J)}$ with $\gamma(s) \in U$ for every $b \leq s \leq \infty \,$, it holds
\[
d(\gamma(b) , \gamma(\infty)) \leq \Xi(\gamma(b)) \,\, .
\]
Then condition \ref{it: ass A - shortening gradient flow lines} holds.
\end{lemma}

\noindent Following our convention, $\Xi$ is meant to be continuous with respect to the original topology on $M$.

\begin{proof}
    For given $z \in Z$, choose a neighborhood $U \subseteq M$ of $z$, a coarse metric $d$ and a function $\Xi$ as in the assumptions. Given sequences $(\gamma_n)_n \subseteq \moduliData$ and $(s_n)_n$ as in \ref{it: ass A - shortening gradient flow lines} and so that $\gamma_n(s_n)$ converges to some point $z_1 \in Z \cap U \,$. Then $\Xi(\gamma_n(s_n))$ converges to $\Xi(z_1) = 0 \,$. Moreover $\gamma_n|_{[s_n,\infty)} $ is an element of $\moduli_{(E_0,\,[s_n,\infty))} \,$. Hence
    \[
    d(\gamma_n(s_n),\gamma_n(\infty)) \leq \Xi(\gamma_n(s_n)) \overset{n}{\rightarrow} 0 \,\, .
    \]
\end{proof}

\subsection{Proof of Theorem \ref{mthm: Compactness moduli spaces}}
\label{subsec: Proof of Compactness Result for moduli spaces}

This section is devoted to the proof of Theorem \ref{mthm: Compactness moduli spaces}. We fix, once and for all, $E_0 \geq 0$ with $E_0 < \specgap$ and $I=[a,\infty)$ for which assumptions \ref{it: ass A - Czeroloc convergence} and \ref{it: ass A - shortening gradient flow lines} hold. In particular the spectral gap is positive
\[
0 < \specgap \leq \infty \,\, .
\]

\begin{lemma}
\label{lem: action value estimate}
For every gradient flow line $\gamma \in \moduliData$ it holds 
\[
\zeta - E_0 \leq f \circ \gamma(s) \leq \zeta \qquad \Forall s \geq a \,\, .
\]
\end{lemma}

\begin{proof}
    Recall that $f \circ \gamma$ is increasing and that $\lim_{s \rightarrow \infty } f \circ \gamma(s) = \zeta$ since $\gamma(\infty) \in Z$. The bound on the energy gives
    \[
    f \circ \gamma(\infty) - f \circ \gamma(a) = E(\gamma) \leq E_0
    \]
    and this finishes the proof.
\end{proof}

\begin{lemma}
\label{lem: gamma(s_n) converges to pt in Z}
Let $\gamma \in \Gamma(I)$ be a gradient flow line with $E(\gamma) \leq E_0$ and 
\[
\zeta - E_0 \leq f \circ \gamma(s) \leq \zeta \qquad \Forall s \geq a \,\, .
\]
\begin{enumerate}[label=(\alph*)]
    \item\label{it: gamma( . + s_n) converges to pt in Z} Given a sequence $(s_n)_{n \in \N} \subseteq [0,\infty)$ tending to infinity. Then a subsequence of $\gamma(\,\cdot\, + s_n)$ converges in $\Czeroloc(I,M)$ to a gradient flow line $\bar{\gamma}  \in \Gamma(I)$ whose image is contained in $Z$.
    \item\label{it: gamma(s_n) converges to pt in Z} There exists a sequence $(s_n)_{n \in \N} \subseteq I$ tending to infinity so that $\gamma(s_n)$ converges to a point in $Z$.
\end{enumerate}

\end{lemma}

\begin{proof}
Part \ref{it: gamma(s_n) converges to pt in Z} is obviously a direct consequence of part \ref{it: gamma( . + s_n) converges to pt in Z}.
\\
The proof of part \ref{it: gamma( . + s_n) converges to pt in Z} follows the lines of \cite[Thm.~3.7]{Albers_Hein}. Let $(s_n)_n$ be sequence of positive numbers tending to infinity. By properties \ref{it: gradient flow line class - restriction} and \ref{it: gradient flow line class - translations} in the definition of a gradient flow line class, the curves $\gamma(\,\cdot\,+s_n)|_I $ are gradient flow lines. Their energy is bounded above by $E_0 \,$.
By assumption \ref{it: ass A - Czeroloc convergence}, up to a subsequence (which we suppress in the notation), the sequence $\gamma(\,\cdot\,+s_n)|_I$ converges to a gradient flow line $\bar{\gamma} \in \Gamma(I)$ in $\Czeroloc(I,M)$.

\begin{claim}
$\bar{\gamma} (I) \subseteq G^{-1}(0) \,$.
\end{claim}

\begin{proofClaim}
 We adopt an argument from \cite[Lemma 2.1]{Morse_Homology_Non-Compact_Manifolds}. By Remark \ref{rmk: gradient flow line} \ref{it: gradient flow line - constant}, it suffices to show that $f \circ \bar\gamma$ is constant on $I$. Assume by contradiction that this is not the case. Then, since $f \circ \bar\gamma$ is increasing, there exists $b > a$ with
 \[
 \eps := f \circ \bar\gamma(b) - f \circ \bar\gamma(a) > 0 \,\, .
 \]
 By pointwise convergence of $\gamma(\,\cdot\, +s_n)$ to $\bar\gamma \,$, there exists $n_1 >0$ so that
 \[
 f \circ \gamma(b + s_n) - f \circ \gamma(a + s_n) \geq \frac{\eps}{2} \qquad \Forall n \geq n_1 \,\, .
 \]
 Now define recursively a sequence $(n_k)_{k \in \N}$ of natural numbers starting at above $n_1$ and with
 \[
 n_{k+1} := \min\set{n > n_k}{s_{n} - s_{n_k} > b - a} \,\, .
 \]
 This is well-defined since $(s_n)_n$ tends to infinity. By definition $a+ s_{n_{k+1}}> b+ s_{n_k}  \,$.
 Choose a natural number $k_0 > 2 E(\gamma) / \eps \,$. We estimate
 \begin{align*}
     E(\gamma) = \int_a^\infty G \circ \gamma (s) \, ds &\geq \sum_{k=1}^{k_0} \int_{a+s_{n_k}}^{b+s_{n_k}} G \circ \gamma(s) \, ds  = \sum_{k=1}^{k_0} \int_{a+s_{n_k}}^{b+s_{n_k}} (f \circ \gamma)' (s) \, ds \\
     &= \sum_{k=1}^{k_0} f \circ \gamma(b+s_{n_k}) - f \circ \gamma(a+s_{n_k}) \\
     &\geq \sum_{k=1}^{k_0} \frac{\eps}{2} >  E(\gamma) \,\, .
 \end{align*}
This contradiction finishes the proof of the claim.
\end{proofClaim}

\noindent Now let $s \geq a$ be arbitrary. By assumption and convergence of $\gamma(s+s_n)$ to $\bar\gamma(s) \,$, we have
\[
\zeta - E_0 \leq  f \circ \bar\gamma(s)  \leq \zeta \,\, .
\]
Using the definition of the spectral gap $\specgap$ and that $E_0 < \specgap \,$, we conclude that $\bar\gamma(s) \in G^{-1}(0)$ lies in $Z$. This finishes the proof of part \ref{it: gamma( . + s_n) converges to pt in Z}.
\end{proof}

\noindent As is common, by a neighborhood of a subset $S \subseteq M$, we mean a subset of $M$ whose interior contains $S$.

\begin{lemma}
\label{lem: delta constant existence}
For every open neighborhood $U \subseteq M$ of $Z$ there exists $\delta = \delta(U) > 0$ so that for every $\gamma \in \moduliData$ it holds
\[
\Forall s \geq a : \quad  \zeta - f \circ \gamma(s) \leq \delta \,\,\, \Longrightarrow \,\,\, \gamma(s) \in U \,\, . 
\]
\end{lemma}

\begin{proof}
 Fix an arbitrary open neighborhood $U$ of $Z$. Suppose by contradiction that no $\delta$ as claimed exists. Then we can choose sequences $(\delta_n)_n \subseteq (0,\infty)$ and $(\gamma_n)_n \subseteq \moduliData$ as well as $(s_n)_n \subseteq I = [a,\infty)$ so that
 \begin{align*}
   \begin{cases}
    \delta_n \rightarrow 0  \,\,\text{ as } n \rightarrow \infty \, , \\
    0 \leq \zeta - f \circ \gamma_n(s_n) \leq \delta_n \,\, \text{ for all } n \, , \\
    \gamma_n(s_n) \notin U \,\, \text{ for all } n \, .
   \end{cases}  
 \end{align*}
 It follows that $f \circ \gamma_n(s_n)$ tends to $\zeta \,$. Using assumption \ref{it: ass A - Czeroloc convergence}, up to a subsequence (which we suppress in the notation), the sequence of gradient flow lines $\gamma_n(\,\cdot\, + s_n -a)|_I$ converges in $\Czeroloc(I,M)$ to a gradient flow line $\gamma \in \Gamma(I)\,$. Then
 \begin{equation}
 \label{eq: delta constant existence - f(gamma(a)) = zeta}
 f \circ \gamma(a) = \lim_{n \rightarrow \infty} f \circ \gamma_n(s_n) = \zeta \,\, .
 \end{equation}
Due to Lemma \ref{lem: action value estimate}, the functions $f \circ \gamma_n$ are bounded above by $\zeta \,$, hence also $f \circ \gamma \leq \zeta$ by pointwise convergence. Combining this with \eqref{eq: delta constant existence - f(gamma(a)) = zeta} and the fact that $f \circ \gamma : I = [a,\infty) \rightarrow \R$ is increasing since $\gamma$ is a gradient flow line, we infer that $f \circ \gamma$ must be constant. By Remark \ref{rmk: gradient flow line} \ref{it: gradient flow line - constant}, thus $\gamma(a) $ lies in $G^{-1}(0)$. Again by \eqref{eq: delta constant existence - f(gamma(a)) = zeta} and since the spectral gap is positive, $\gamma(a)$ must in fact lie in $Z \,$. Since $\gamma_n(s_n)$ tends to $\gamma(a) \in Z \subseteq U\,$, almost all $\gamma_n(s_n)$ must lie in $U$, contradicting our choices above.
\end{proof}

\begin{lemma}
\label{lem: uniform time for nbhd of Z}
For every open neighborhood $U \subseteq M$ of $Z$ there exists $s_0 \geq a$ with
\[
\Forall \gamma \in \moduliData : \quad \gamma([s_0,\infty)) \subseteq U \,\, .
\]
\end{lemma}

\begin{proof}
Let $U$ be arbitrary but fixed and choose $\delta = \delta(U) > 0$ as in Lemma \ref{lem: delta constant existence}. Suppose by contradiction that no $s_0$ as claimed exists. Then we can choose sequences $(s_n)_{n} \subseteq I$ tending to infinity and $(\gamma_n)_n \subseteq \moduliData$ with $\gamma_n(s_n) \notin U $ for all $n \in \N$. By assumption \ref{it: ass A - Czeroloc convergence}, up to a subsequence (suppressed in the notation), the sequence $\gamma_n$ converges in $\Czeroloc(I,M)$ to a gradient flow line $\gamma \in \Gamma(I) \,$. Due to Lemma \ref{lem: action value estimate} and pointwise convergence $\gamma_n \rightarrow \gamma$, we have
\[
\zeta - E_0 \leq f \circ \gamma(s) \leq \zeta \qquad \Forall s \geq a 
\]
and, by Fatou's lemma and since the $\gamma_n \in \moduliData$ have energy bounded by $E_0 \,$, moreover
\[
E(\gamma) \leq \liminf_{n \rightarrow \infty} E(\gamma_n) \leq E_0 \,\, .
\]
By Lemma \ref{lem: gamma(s_n) converges to pt in Z} \ref{it: gamma(s_n) converges to pt in Z}, there exists a sequence $(\bar{s}_k)_k \subseteq I$ tending to infinity so that $\gamma(\bar{s}_k)$ converges to a point in $Z$. We can thus choose a sufficiently large $k_0 \in \N$ with
\[
\zeta - \frac{\delta}{2} \leq f \circ \gamma(\bar{s}_{k_0}) \leq \zeta \,\, .
\]
By pointwise convergence of $\gamma_n$ to $\gamma \,$, there exists $n_0 \in \N$ so that
\[
\zeta - \delta \leq f \circ \gamma_n(\bar{s}_{k_0}) \leq \zeta \qquad \Forall n \geq n_0 \,\, .
\]
Now $f \circ \gamma_n$ is increasing and hence
\[
\zeta - \delta \leq f \circ \gamma_n (s) \leq \zeta \qquad \Forall s \geq \bar{s}_{k_0} \comma n \geq n_0 \,\, .
\]
By our choice of $\delta$ this implies that $\gamma_n([\bar{s}_{k_0} , \infty)) \subseteq U$ for every $n \geq n_0 \,$. This contradicts the fact that $(s_n)_n$ tends to infinity and satisfies $\gamma_n(s_n) \notin U$ for every $n$.
\end{proof}

\begin{lemma}
\label{lem: uniform time for pt in Z}
Given $z \in Z $ and a sequence $(\gamma_n)_n \subseteq \moduliData$ for which $(\gamma_n(\infty))_n$ converges to $z$. Then for every neighborhood $V \subseteq M$ of $z$ there exist $\bar{s} \geq a$ and $n_0 \in \N$ with
\[
\Forall n \geq n_0 : \quad \gamma_n( [\bar{s}, \infty)) \subseteq V \,\, .
\]
\end{lemma}

\begin{proof}
   Let $z \in Z$ be fixed and let $(\gamma_n)_n \subseteq \moduliData$ be a sequence for which the limits $\gamma_n(\infty) \in Z$ converge to $z$ as $n \rightarrow \infty \,$. Choose an open neighborhood $U \subseteq M$ of $z$ and a coarse metric $d$ on $U$ as in \ref{it: ass A - shortening gradient flow lines}. It obviously suffices to show the statement of the lemma for a neighborhood basis of $z$. Let $V \subseteq M$ be an arbitrary open neighborhood of $z$ whose closure $\overline{V}$ is contained in $U \,$. (Such $V$ constitute a neighborhood basis of $z$ since $M$ is metrizable). Because $\gamma_n(\infty)$ converge to $z$, there exists $n_0 \in \N$ with $\gamma_n(\infty) \in V$ for $n \geq n_0 \,$. For $n \geq n_0$ define
   \[
   s_n := \inf\set{s \geq a}{ \gamma_n([s, \infty) ) \subseteq V}   \,\, .
   \]
  Notice that $s_n$ is finite. The lemma follows if we manage to show that 
  \[
  \sup_{n \geq n_0} s_n < \infty \,\, .
  \]
  We assume by contradiction that this is not the case. Then, up to extracting a subsequence (which we suppress in the notation), we can assume that $(s_n)_n$ tends to infinity and that $s_n > a$ for every $n$. By definition of $s_n$ we have 
  \[
  \gamma_n([s_n,\infty)) \subseteq \overline{V} \subseteq U
  \]
  and moreover $\gamma_n(s_n) \in \partial V$ since $s_n > a \,$.
\begin{claim}
   A subsequence of $\gamma_n(s_n)$ converges to a point in $Z$.
\end{claim}
\begin{proofClaim}
    By \ref{it: ass A - Czeroloc convergence}, a subsequence (suppressed in the notation) of $\gamma_n(\,\cdot\,+s_n - a)|_I$ converges in $\Czeroloc(I,M)$. In particular $\gamma_n(s_n)$ converges to some point $z_* \in M \,$. Assume by contradiction that $z_* \notin Z \,$. As metrizable space, $M$ is a regular Hausdorff space, so we can choose an open neighborhood $U_Z \subseteq M$ of the compactum $Z$ whose closure does not contain $z_* \,$. By Lemma \ref{lem: uniform time for nbhd of Z}, there exists $\bar{s} \geq a$ with $\gamma_n([\bar{s},\infty)) \subseteq U_Z$ for every $n$. As $(s_n)_n$ tends to infinity, we see that $\gamma_n(s_n) \in U_Z$ for almost all $n$ and hence that $z_* = \lim_n \gamma_n(s_n)$ lies in the closure of $U_Z \,$. Contradiction. 
\end{proofClaim}
\noindent By the claim, up to a further subsequence (suppressed in the notation), we can assume that $\gamma_n(s_n)$ converges to some point $z_* \in Z$. As limit of the sequence $(\gamma_n(s_n))_n$ contained in $\partial V$, also $z_*$ lies in $\partial V$. In particular $z_*$ is distinct from $z \in V \,$. On the other hand, we can estimate 
\begin{align*}
    d(z_*,z) &\leq d(z_*, \gamma_n(s_n)) + d(\gamma_n(s_n) , \gamma_n(\infty)) + d(\gamma_n(\infty), z) \,\, .
\end{align*}
The first and third summand tend to zero as $n \rightarrow \infty$ since $\gamma_n(s_n)$ converges to $z_*$ and $ \gamma_n(\infty)$ converges to $z$ and since $d$ is a coarse metric (see Remark \ref{rmk: coarse metric}). The second summand $d(\gamma_n(s_n),\gamma_n(\infty))$ tends to zero because the coarse metric $d$ has the property described in \ref{it: ass A - shortening gradient flow lines}, using additionally that $\gamma_n([s_n,\infty]) \subseteq \overline{V} \subseteq U$ and again that $\gamma_n(s_n)$ tends to $z_* \in Z \cap U \,$. Thus $d(z_*,z) = 0 \,$, contradicting our above observation that $z_*$ is distinct from $z$. This contradiction finishes the proof.
\end{proof}

\noindent The proof idea of the above Lemma \ref{lem: uniform time for pt in Z} is sketched in Figure \ref{fig: uniform time lemma proof}. The compactness theorem is now an immediate consequence of the lemma.
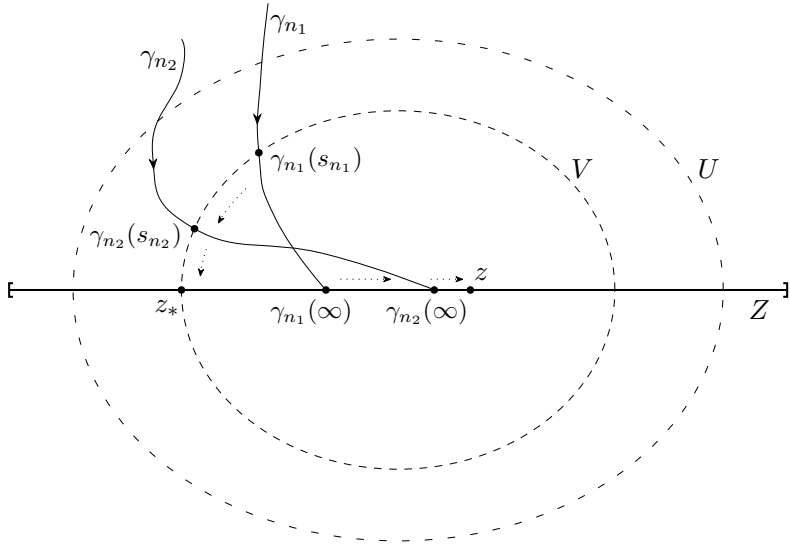
\begin{figure}
    \centering
     \begin{tikzpicture}[scale=1, x=8cm , y=7cm]

\draw[dashed, name=Circle] (0,0) circle [x radius = 3cm, y radius= 2.5cm] ; 
\draw[loosely dashed] (0,0) circle [x radius = 4.5cm, y radius= 3.5cm] ; 

\draw[thick, {Bracket}-{Bracket}] (-5.4cm,0) -- (5.4cm,0) node[below, xshift= -0.4cm] {$Z$} ;
\node[] at (2.55cm,1.7cm) {$V$};
\node[] at (4.3cm,1.7cm) {$U$};

\fill (1cm,0) circle (1.5pt) node[above, xshift=0.15cm] {$z$};
\fill (-3cm,0) circle (1.5pt) node[below, xshift=-0.2cm] {$z_*$} ;
\fill (-1cm,0) circle (1.5pt) node (gamma_one_infty) {};
\node[below, xshift=-0.2cm , node font=\small] at (gamma_one_infty) {$\gamma_{n_1}(\infty)$} ;
\fill (0.5cm,0) circle (1.5pt) node (gamma_two_infty) {};
\node[below, xshift=-0.1cm, node font=\small] at (gamma_two_infty) {$\gamma_{n_2}(\infty)$} ;

\path (3cm,0) arc (0:130:3cm and 2.5cm) node (gamma_one) {};
\fill (gamma_one) circle (1.5pt) node[xshift=0.8cm, yshift=-0.1cm, node font=\small] {$\gamma_{n_1}(s_{n_1})$} ; 
\path (3cm,0) arc (0:160:3cm and 2.5cm) node (gamma_two) {};
\fill (gamma_two) circle (1.5pt) node[node font=\small, left, xshift=-0.05cm,yshift=-0.1cm] {$\gamma_{n_2}(s_{n_2})$}; 

\begin{scope}[decoration={
    markings,
    mark=at position 0.4 with {\arrow{Stealth[]}}}
    ] 
    \draw[postaction={decorate}] plot [smooth, tension = 0.9] coordinates { (-1.8cm, 4cm) (-1.9cm,3cm) (gamma_one) (-1.7cm,1cm) (gamma_one_infty)} ;
\end{scope}
\begin{scope}[decoration={
    markings,
    mark=at position 0.3 with {\arrow{Stealth[]}}}
    ] 
    \draw[postaction={decorate}] plot [smooth, tension = 0.9] coordinates { (-3cm, 3.5cm) (-3cm, 3cm) (-3.4cm, 1.9cm)  (gamma_two) (-1cm,0.5cm) (gamma_two_infty)} ;
\end{scope}
\node at (-1.5cm,3.7cm) {$\gamma_{n_1}$} ;
\node at (-3.3cm,3.2cm) {$\gamma_{n_2}$} ;

\path (2.75cm,0) arc (0:140:2.75cm and 2.2cm) node (ddots_pos_one) {} ;
\path (2.75cm,0) arc (0:165:2.75cm and 2.2cm) node (ddots_pos_two) {} ;
\draw[dotted,-{Stealth[round,scale=0.8]}] (ddots_pos_one) arc (140:155:2.75cm and 2.2cm);
\draw[dotted,-{Stealth[round,scale=0.8]}] (ddots_pos_two) arc (165:175:2.75cm and 2.2cm);

\draw[dotted,-{Stealth[round,scale=0.8]}] (-0.8cm,0.15cm)  --  (-0.1cm,0.15cm)   ;
\draw[dotted,-{Stealth[round,scale=0.8]}] (0.45cm,0.15cm)  --  (0.9cm,0.15cm)   ;

\end{tikzpicture}
    \caption{Illustrating the contradiction in the proof of Lemma \ref{lem: uniform time for pt in Z}. Two trajectories $\gamma_{n_1}$ and $\gamma_{n_2}$ for $n_1 \ll n_2$ are depicted. The entry points $\gamma_{n}(s_{n})$ into $V$ tend to $z_*$ while the end points $\gamma_n(\infty)$ tend to $z$. This contradicts assumption \ref{it: ass A - shortening gradient flow lines}, namely that the distance from $\gamma_n(s_n)$ to $\gamma_n(\infty)$ tends to zero.} 
    \label{fig: uniform time lemma proof}
\end{figure}

\begin{proof}[Proof of Theorem \ref{mthm: Compactness moduli spaces}]
    Given an arbitrary sequence $(\gamma_n)_n \subseteq \moduliData \,$. By \ref{it: ass A - Czeroloc convergence}, a subsequence (suppressed in the notation) of $(\gamma_n)_n$ converges to a gradient flow line $\gamma \in \Gamma(I)$ in $\Czeroloc(I,M)$. By Fatou's lemma
    \[
    E(\gamma) \leq \liminf_{n \rightarrow \infty} E(\gamma_n) \leq E_0 \,\, .
    \]
    Since $Z$ is compact, up to a further subsequence, we can assume that the $\gamma_n(\infty) \in Z$ converge to some $z \in Z$. From Lemma \ref{lem: uniform time for pt in Z}, applied to this $z$ and $(\gamma_n)_n \, $, and pointwise convergence $\gamma_n \rightarrow \gamma \,$, we infer that $\gamma(s) \rightarrow z$ as $s \rightarrow \infty$. Hence $\gamma $ is an element of $\moduliData \,$.
    The uniform convergence $\gamma_n \overset{n}{\rightarrow} \gamma$ in $C^0(I,M)$ follows from $\Czeroloc(I,M)$-convergence and Lemma \ref{lem: uniform time for pt in Z} again.
\end{proof}

\subsection{Variants of the Compactness Theorem}
\label{subsec: Variants of Compactness Thm}

In practical applications one sometimes need slight variations of Theorem \ref{mthm: Compactness moduli spaces}. We list them here for completeness. Their proofs are completely analogous and left to the reader.

\begin{remark}
\label{rmk: compactness thm - negative gradient flow lines and negative ends}
There are analogous versions of Theorem \ref{mthm: Compactness moduli spaces} (with the assumptions therein adapted mutatis mutandis) for
\begin{enumerate}[label=(\roman*)]
    \item gradient flow lines on unbounded intervals of the form $(-\infty,a]$. For this one needs to replace the spectral gap $\specgap$ for positive ends by the spectral gap for negative ends
    \begin{align*}
    \mathfrak{S}^{-}_{f,G}(Z) := \inf\Bigset{f(x)-\zeta}{\begin{array}{c} x \in G^{-1}(0) \backslash Z \\
    \text{with } \zeta \leq f(x)
    \end{array}} \in [0,\infty] \,\, .
    \end{align*}
    \item negative gradient flow lines.
\end{enumerate}
\end{remark}

\noindent In Floer theory, one often cannot show the precise condition \ref{it: ass A - Czeroloc convergence} as the usual rescaling argument to show uniformly bounded gradients (by producing $J$-holomorphic bubbles if gradients explode) needs an arbitrarily small margin away from the boundary of the interval. For this reason, we introduce a modification of condition \ref{it: ass A - Czeroloc convergence}: For intervals $I_1 = [a_1,\infty) \subseteq I = [a,\infty) \,$, consider instead

{
\let\realItem\item 
\makeatletter
\NewDocumentCommand\myItem{ o }{%
   \IfNoValueTF{#1}%
      {\realItem}
      {\realItem[#1]\def\@currentlabel{#1}}
}
\makeatother

\setlist[enumerate]{
    before=\let\item\myItem,       
    label=\textnormal{(\arabic*)}, 
    widest=(2')                    
}

\begin{enumerate}
    \item[(A1'$_{(E_0,I,I_1)}$)]\label{it: ass A variant - Czeroloc convergence} 
    For every sequence $(\gamma_n)_n \subseteq \Gamma(I)$ with energy uniformly bounded by $E_0$ there exists a subsequence converging in $\Czeroloc(I_1,M)$ to some $\gamma\in \Gamma(I_1) \,$.
\end{enumerate}
}

\noindent Replacing assumption \ref{it: ass A - Czeroloc convergence} with \ref{it: ass A variant - Czeroloc convergence}, we obtain the following generalization of Theorem \ref{mthm: Compactness moduli spaces}.

\begin{theorem}
\label{thm: Compactness moduli spaces - variant}
Suppose $0 \leq E_0 < \specgap $ and $I_1=[a_1,\infty) \subseteq I = [a,\infty)$ satisfy \ref{it: ass A variant - Czeroloc convergence} and \ref{it: ass A - shortening gradient flow lines}. Then for every sequence $(\gamma_n)_{n \in \N} \subseteq \moduliData$ there exists a subsequence which converges in $C^0(I_1,M)$ to some $\gamma \in \moduli_{(E_0,I_1)} \,$.
\end{theorem}

\begin{proof}
    The proof is completely analogous to the one of Theorem \ref{mthm: Compactness moduli spaces}. We just point out that one must prove a version of Lemma \ref{lem: gamma(s_n) converges to pt in Z} which asserts that for every sequence $(s_n)_n \subseteq [0,\infty)$ tending to infinity and every $\gamma \in \Gamma(I_1)$ with energy bounded by $E_0$ and $f \circ \gamma $ taking values in $[\zeta-E_0,\zeta] \,$, a subsequence of $\gamma(\,\cdot\,+s_n)$ converges in $\Czeroloc(I_1,M)$ to some $\bar\gamma \in \Gamma(I_1)$ with image contained in $Z$. To this end, apply \ref{it: ass A variant - Czeroloc convergence} to the sequence $\gamma(\,\cdot\,+s_n)$ of gradient flow lines defined on $[a_1-s_n,\infty) \supseteq I$ (for $n$ sufficiently large).
    This version of Lemma \ref{lem: gamma(s_n) converges to pt in Z} suffices for the proof of Lemma \ref{lem: uniform time for nbhd of Z} to go through.
\end{proof}

\section{Illustrating the Compactness Result in the Morse Setting}
\label{sec: Morse - Illustration}

\noindent In this section, we illustrate Theorem \ref{mthm: Compactness moduli spaces} by reproving various compactness results in Morse-Bott theory. In Subsection \ref{subsec: Morse-Bott functions (finite dim)} we show how to verify conditions \ref{it: ass A - Czeroloc convergence} and \ref{it: ass A - shortening gradient flow lines} under a Morse-Bott assumption, so that Corollary \ref{cor: Morse setting - compactness moduli space} will be an immediate consequence of Theorem \ref{mthm: Compactness moduli spaces}. In the analytic approach to Morse theory, as carried out in \cite{M_Schwarz}, one is primarily interested in compactness of moduli spaces with respect to the subspace topology of an ambient Banach manifold of maps. In Subsection \ref{subsec: Morse setting - ambient Banach mfd}, we will illustrate how to easily derive the latter kind of compactness from the compactness in the topology of uniform convergence and an exponential decay estimate. Finally, Subsection \ref{subsec: Optimality of Upper Bound for Energy} contains an example, showing that the upper bound for the energy in Theorem \ref{mthm: Compactness moduli spaces} is optimal.

\subsection{Gradient Flow Lines of Morse-Bott Functions}
\label{subsec: Morse-Bott functions (finite dim)}

Let $(M,g)$ be a (finite dimensional) Riemannian manifold and $f \in \Cinfty(M)$ be a smooth function which is bounded below and proper. (These conditions are automatically fulfilled if $M$ is a closed manifold.) In particular $f$ has compact sublevel sets. Let $Z \subseteq \Crit(f)$ be a closed (\ie compact and without boundary), connected submanifold of $M$ and suppose that $f$ is Morse-Bott along $Z$, that is
\[
T_z Z = \ker \mathrm{Hess}_z(f) \subseteq T_z M \qquad \Forall z \in Z  \,\, .
\]
Since $Z$ is a connected critical component, $f$ is constant on $Z$, say with value $\zeta \in \R \,$. As discussed in Example \ref{ex: gradient flow line}, assigning to each closed interval $I \subseteq \R$ the set
\begin{align}
\label{eq: Morse setting - gradient flow line class}
\set{\gamma \in \Cinfty(I,M)}{\gamma' = \nabla f \circ \gamma  \text{ and } f \circ \gamma \leq \zeta } \,\, ,
\end{align}
defines a gradient flow line class with respect to $f$ and $G := |\nabla f|^2 \,$. The spectral gap $\specgap$ for positive ends (see Definition \ref{def: specgap}) is just the ordinary spectral gap for gradient flow lines ending at $Z$, namely
\begin{align*}
    \mathfrak{S}^+_f(Z) := \inf \set{\zeta - f(x)}{x \in \Crit(f)\backslash Z \text{ with } f(x) \leq \zeta} = \specgap \,\, .
\end{align*}
We also have a spectral gap for gradient flow lines beginning at $Z$, namely
\begin{align*}
    \mathfrak{S}^-_f(Z) := \inf \set{f(x) - \zeta}{x \in \Crit(f)\backslash Z \text{ with } \zeta \leq f(x) } =  \mathfrak{S}^-_{f,G}(Z) 
\end{align*}
and the ordinary spectral gap
\begin{align*}
    \mathfrak{S}_f(Z) := \inf \set{|\zeta - f(x)|}{x \in \Crit(f)\backslash Z}  = \min\{ \mathfrak{S}_f^+(Z), \mathfrak{S}_f^-(Z)\} \,\, .
\end{align*}
Next we ensure that conditions \ref{it: ass A - Czeroloc convergence} and \ref{it: ass A - shortening gradient flow lines} hold for all energy levels $E_0$ and closed intervals $I = [a,\infty)$ with respect to the gradient flow line class given in \eqref{eq: Morse setting - gradient flow line class}.

\begin{lemma}
\label{lem: Morse setting - ass A Czeroloc convergence}
For every $E_0 >0$ and every interval $I=[a,\infty) \,$, condition \ref{it: ass A - Czeroloc convergence} holds. The convergence of the subsequence can even be arranged to be in $\Cinftyloc(I,M) \,$.
\end{lemma}

\begin{proof}
The gradient $|\nabla f|$ is bounded on the compact sublevel set $\{f \leq \zeta\} \subseteq M$. Given a sequence $(\gamma_n)_n$ of gradient flow lines on $I$ with image contained in $\{f \leq \zeta\} \,$. Then a subsequence converges to some $\gamma$ in the $\Czeroloc(I,M)$-topology by the Arzel\`a-Ascoli theorem. Deriving the equation $\gamma_n' = \nabla f\circ \gamma_n$ iteratively, we see inductively that the higher order derivatives of the $\gamma_n$ (in finitely many coordinate charts covering $\gamma(I)$) are uniformly bounded. Again by Arzel\`a-Ascoli, we conclude that a subsequence of $(\gamma_n)_n$ converges in $\Cinftyloc(I,M)$ to $\gamma$. It follows that $\gamma$ itself is a gradient flow line with $f \circ \gamma \leq \zeta \,$.
\end{proof}

\noindent To establish \ref{it: ass A - shortening gradient flow lines}, we need the following exponential decay lemma which crucially relies on the assumption that $f$ is Morse-Bott along $Z$. We denote by $d_g : M \times M \rightarrow [0,\infty)$ the Riemannian distance function induced by the metric $g$.

\begin{lemma}[Lemma 3.6 in \cite{Albers_Hein}]
\label{lem: Morse setting - exponential decay}
There exist a tubular neighborhood $U_Z \subseteq M$ of $Z$ and constants $A,B > 0$ with the following significance. For every interval $I = [s_0,\infty)$ and every $\gamma \in \Cinfty(I,U_Z)$ with $\gamma' = \nabla f \circ \gamma$ and existing limit $\gamma(\infty) = \lim_{s \rightarrow \infty} \gamma(s) \in Z$, it holds
\[
d_g(\gamma(s) , \gamma(\infty)) \leq A \, \sqrt{\zeta - f \circ \gamma(s_0)} \,  e^{-B (s-s_0)} \qquad \Forall s \geq s_0 \,\, .
\]
\end{lemma}

\noindent The lemma is formulated slightly differently in \cite{Albers_Hein}, notably for negative gradient flow lines and under the additional assumption that $\zeta=0 \,$.

\begin{remark}
\label{rmk: Morse setting - exponential decay derivative}
Choosing $U_Z$ smaller if necessary, we may assume that it is precompact. Hence $|\nabla f|$ is Lipschitz continuous on $U_Z$ with some Lipschitz constant $C > 0 \,$. Given a gradient flow line $\gamma \in \Cinfty(I,U_Z)$ with limit $\gamma(\infty) \in Z$, the estimate
\[
|\gamma'(s)| = \big| \, |\nabla f \circ \gamma(s)| - |\nabla f \circ \gamma(\infty)| \, \big| \leq C \, d_g(\gamma(s),\gamma(\infty)) 
\]
shows that also $|\gamma'|$ exhibits exponential decay with same exponential coefficient $B \,$.
\end{remark}

\begin{lemma}
\label{lem: Morse setting - ass A shortening gradient flow}
For every $E_0 > 0$ and every interval $I=[a,\infty)$ condition \ref{it: ass A - shortening gradient flow lines} holds.
\end{lemma}

\begin{proof}
    We want to invoke Lemma \ref{lem: shortening gradient flow lines criterion}. Choose the tubular neighborhood $U_Z$ and the constants $A,B > 0$ as in Lemma \ref{lem: Morse setting - exponential decay}.
    Consider the coarse metric $d_g$ on $M$ (the topology induced by $d_g$ is not just coarser but in fact equal to the given topology on $M$) which restricts to a coarse metric on $U_Z \,$. As continuous function $\Xi : U_Z \rightarrow \R$ take 
    \[
    \Xi(x) := A \, \sqrt{|\zeta - f(x)|} \geq 0 \,\, .
    \]
    Then $\Xi$ vanishes on $Z = Z \cap U_Z $ and 
    \[
    d_g(\gamma(s_0), \gamma(\infty)) \leq A \, \sqrt{\zeta - f \circ \gamma(s_0)} \, e^{-B (s_0-s_0)} =  \Xi(\gamma(s_0)) 
    \]
    for every gradient flow line $\gamma : [s_0,\infty) \rightarrow M$ with image contained in $U_Z$ and limit in $Z$. Hence the above choices meet the requirements in Lemma \ref{lem: shortening gradient flow lines criterion} for every $z \in Z \,$.
\end{proof}

\noindent Having verified conditions \ref{it: ass A - Czeroloc convergence} and \ref{it: ass A - shortening gradient flow lines}, we can apply Theorem \ref{mthm: Compactness moduli spaces} to obtain the next corollary.

\begin{corollary}
\label{cor: Morse setting - compactness moduli space}
Given $E_0 < \mathfrak{S}_f^+(Z)$ and an interval $I=[a,\infty) \,$.
For every sequence $(\gamma_n)_n \subseteq \Cinfty(I,M)$ of gradient flow lines with energy uniformly bounded by $E_0$ and limits $\gamma_n(\infty) \in Z \,$, there exists a subsequence converging in $\Cinfty(I,M)$ to a gradient flow line $\gamma \in \Cinfty(I,M)$ of energy $E(\gamma) \leq E_0$ and with limit $\gamma(\infty) \in Z $. 
\end{corollary}

\begin{proof}
    By Theorem \ref{mthm: Compactness moduli spaces}, a subsequence $(\gamma_{n_k})_k$ converges in $C^0(I,M)$ to a gradient flow line $\gamma \in \Cinfty(I,M)$ of energy bounded by $E_0$ and with existing limit $\gamma(\infty) \in Z \,$. Now, analogous to the proof of Lemma \ref{lem: Morse setting - ass A Czeroloc convergence}, the gradient flow equation $\gamma_{n_k}' = \nabla f \circ \gamma_{n_k}$ and $C^0(I,M)$-convergence implies $\Cinfty(I,M)$-convergence $\gamma_{n_k} \rightarrow \gamma $ as $ k \rightarrow \infty \,$.
\end{proof}

\subsection{Compactness of Moduli Spaces in an Ambient Banach Manifold}
\label{subsec: Morse setting - ambient Banach mfd}

Continue in the setting of the previous Subsection \ref{subsec: Morse-Bott functions (finite dim)}. In the analytic approach to Morse theory, one is often not primarily interested in compactness of the moduli space with respect to uniform convergence but the ambient topology of the Banach manifold of maps containing the moduli space. In this subsection we show by an example (see Theorem \ref{thm: modulihat(Z_-,Z_+,zeta_0) compact in Bcal(Z_-,Z_+)} below) how one can still easily deduce compactness with respect to the latter topology from compactness with respect to the former topology. For the Morse case this is similarly stated in \cite[Lemma 2.39]{M_Schwarz} and the Morse-Bott case is a straightforward adaption.

\subsubsection{Moduli Space of Unparametrized Trajectories}
\label{subsubsec: Morse setting ambient Banach mfd - moduli space unparametrized trajectories}

\noindent Again let $f \in \Cinfty(M)$ be a proper, bounded below function on an $m$-dimensional Riemannian manifold $(M,g)$. Let $Z_- , Z_+ \subseteq \Crit(f)$ be disjoint connected components of $\Crit(f)$ that are closed submanifolds of $M$ and suppose $f$ is Morse-Bott along $Z_-$ and $Z_+ \,$. On $Z_\pm$ the function $f$ is constant, say with value $\zeta_\pm \equiv f|_{Z_\pm} \,$. We suppose moreover that there exists a gradient flow line $\gamma : \R \rightarrow M$ with existing limits $\gamma(\pm \infty) \in Z_\pm \,$. From this it follows that $\zeta_- < \zeta_+ \,$. We assume $\zeta_0 \in (\zeta_-,\zeta_+)$ is a value for which
\begin{align}
\label{eq: Morse setting - zeta_0}
    \zeta_+ - \zeta_0 < \mathfrak{S}_f^+(Z_+) \quad \text{ and } \zeta_0 - \zeta_- < \mathfrak{S}_f^-(Z_-) \,\, .
\end{align}
If $\mathfrak{S}_f^+(Z_+) = \mathfrak{S}_f^-(Z_-) = \zeta_+ - \zeta_-$, which holds for instance if $Z_-$ and $Z_+$ are the only components of $\Crit(f)$, then every value $\zeta_0 \in (\zeta_-,\zeta_+)$ has the above property.
Now consider the \textit{moduli space of unparametrized trajectories from $Z_-$ to $Z_+$}
\begin{align*}
    \modulihat (Z_-,Z_+, \zeta_0) := \Bigset{\gamma \in \Cinfty(\R,M)}{\begin{array}{c}
    \gamma' = \nabla f \circ \gamma \comma \,\, f \circ \gamma(0) = \zeta_0 \, , \\
   \text{limits } \gamma(\pm \infty) \in Z_\pm \text{ exist}
    \end{array}} \,\, .
\end{align*}
Elements in $\modulihat(Z_-,Z_+,\zeta_0)$ have energy $\zeta_+ - \zeta_0 < \mathfrak{S}_f^+(Z_+)$ on $[0,\infty)$ and energy $\zeta_0 - \zeta_- < \mathfrak{S}_f^-(Z_-)$ on $(-\infty,0]$. Therefore, Corollary \ref{cor: Morse setting - compactness moduli space} and an analogous version for the interval $(-\infty,0] $ yield the following result.

\begin{corollary}
\label{cor: Morse setting ambient Banach mfd - C_0 compactness}
Suppose \eqref{eq: Morse setting - zeta_0} holds. Then the moduli space $\modulihat(Z_-,Z_+,\zeta_0)$ is compact with respect to the topology of $\Cinfty(\R,M)$-convergence.
\end{corollary}

\noindent The above corollary is a special case of a well-known compactness result for moduli spaces of unparametrized trajectories up to breaking, cf. \cite[Prop.~2.35]{M_Schwarz} for the Morse case. Under \eqref{eq: Morse setting - zeta_0}, breaking cannot happen.

\subsubsection{The Ambient Banach manifold}
\label{subsubsec: Morse setting ambient Banach mfd - def Banach mfd}

Since we are in a Morse-Bott situation, we have to work with weighted Sobolev spaces, the definition of which we recall now.

\begin{definition}[Weighted Sobolev Space]
\label{def: weighted Sobolev space - Morse setting}
Given $\delta > 0 \comma p > 1 \comma k \in \N_{0} $ and an unbounded closed interval $I \subseteq \R \,$. The \textit{$\delta$-weighted Sobolev space} $W^{k,p}_\delta(I,\R^n)$ is defined as
\begin{align*}
    W^{k,p}_\delta(I,\R^m) := &\Bigset{u \in W^{k,p}(I,\R^m)}{ (s \mapsto e^{\delta |s|} \, u(s) ) \in W^{k,p}(I,\R^n)} \\
    = &\Bigset{u \in W^{k,p}(I,\R^m)}{e^{\delta |s|} \,\partial_s^i u \in L^p(I,\R^m) \text{ for all } 0 \leq i \leq k}
\end{align*}
and is a Banach space with the norm $\lVert u \rVert_{W^{k,p}_\delta(I)} := \lVert e^{\delta |s|} \, u \rVert_{W^{k,p}(I)} \,$.\\
Instead of $W^{k,2}_\delta$ we also write $H^k_\delta \,$.
\end{definition}

\noindent A norm equivalent to $\lVert \,\cdot\, \rVert_{W^{k,p}_\delta}$ is given by $\lVert u \rVert := \sum_{i=0}^k \lVert e^{\delta  |s|} \, \partial_s^i u \rVert_{L^p(I)} \,$. We interchange these norms whenever it feels more convenient.
\\

\noindent Given $\gamma \in H^1_{\mathrm{loc}}(\R,M) \subseteq C^0(\R,M)$ with existing limits $\gamma(\pm\infty) = \lim_{s \rightarrow \pm \infty} \gamma(s)$ and coordinate charts $\varphi_\pm : U_\pm \rightarrow \varphi_\pm(U_\pm) \subseteq \R^m$ around $\gamma(\pm \infty)\,$, the condition
\begin{align}
 &\varphi_+ \circ \gamma - \varphi_+(\gamma(\infty)) \in H^1_\delta([s_0,\infty),\R^m) \quad \hspace{10mm} \text{and }\notag\\
  &\varphi_- \circ \gamma - \varphi_-(\gamma(-\infty)) \in H^1_\delta((-\infty,-s_0],\R^m) \quad \text{ for some } s_0 > 0 \text{ sufficiently large.}   \label{eq: Morse setting - def Bcal limit condition}
\end{align}
is independent of the charts $(U_\pm,\varphi_\pm)$ around $\gamma(\pm \infty) \,$. We now define the Banach manifold $\BcalMorse$ by
\begin{align}
\label{eq: Morse setting - def BcalMorse}
    \BcalMorse := \Bigset{\gamma \in H^1_{\mathrm{loc}}(\R,M)}{\begin{array}{c}
       \gamma(\pm \infty) \in Z_\pm \text{ exist and \eqref{eq: Morse setting - def Bcal limit condition} holds for any}  \\
       \text{choice of charts } (U_\pm,\varphi_\pm) \text{ around } \gamma(\pm \infty)   
    \end{array}} \,\, .
\end{align}

\subsubsection{Charts for $\BcalMorse$}
\label{subsubsec: Morse setting ambient Banach mfd - charts}
Next we construct charts for $\BcalMorse \,$. To this end, we first define a distinguished class of elements where our charts will be centered. Following Schwarz \cite{M_Schwarz}, we give $\Rbar := \R \cup \{\pm \infty\}$ the structure of a smooth manifold with boundary by declaring the homeomorphism
\[
\Rbar \rightarrow [-1,1] \comma \quad s \mapsto \frac{s}{\sqrt{1+s^2}}
\]
to be a diffeomorphism. We construct charts around every $\gamma_0 \in \BcalMorse \cap \Cinfty(\Rbar,M)$ which altogether constitute an atlas for $\BcalMorse$.\\ 
Let us temporarily write $d_\pm := \dim(Z_\pm) $ and fix, once and for all, a monotone cut-off function $\beta \in \Cinfty(\R,[0,1])$ with
\begin{align*}
\begin{cases}
    \beta(s) = 0 & \text{, if } s \leq \eps \\
    \beta(s) = 1 & \text{, if } s \geq 1 - \eps \\
\end{cases}
\end{align*}
for some small $\eps > 0 \,$.
For given $\gamma_0 \in \BcalMorse \cap \Cinfty(\Rbar,M)$ fix slice charts $(U_\pm,\varphi_\pm)$ for $Z_\pm$ centered at $z_\pm := \gamma_0(\pm \infty) \in Z$, that is $\varphi_\pm(z_\pm)$ is the origin in $\R^m$ and
\[
V_\pm := \varphi_\pm(U_\pm \cap Z_\pm) = \varphi_\pm (U_\pm) \cap (\R^{d_\pm} \times \{0\}) \subseteq \R^m
\]
is the intersection of the chart image with $\R^{d_\pm} \times \{0\} \subseteq \R^m \,$. We choose $U_-$ and $U_+$ so small that the charts $\varphi_\pm$ can be extended over the disjoint compact closures of $U_\pm \,$. We fix an auxiliary Riemannian metric $\gaux$ on $M$ which, pulled back via the charts $\varphi_\pm^{-1} \,$, is the Euclidean metric on $\varphi_\pm(U_\pm) \subseteq \R^m \,$. Moreover, we fix an open, fiberwise convex neighborhood $\Dcal \subseteq TM$ of the zero section for which
 \[
 \exp^{\gaux} : \Dcal \rightarrow M \times M
 \]
 is a diffeomorphism onto an open neighborhood of the diagonal and a smaller open, fiberwise convex neighborhood $\mathcal{D}' \subseteq TM$ of the zero section with closure contained in $\mathcal{D} \,$.
Since $\varphi_\pm$ is an isometry from $(U_\pm, \gaux)$ to Euclidean space, we have
\begin{align}
\label{eq: Morse setting - exp gaux vs Euclidean exp}
    \varphi_\pm \circ \exp^{\gaux}_x (w) = \varphi_\pm(x) +(d \varphi_\pm)_x \cdot w \qquad \text{ for } w \in \mathcal{D}_x \subseteq T_x M \comma x \in U_\pm \,\, ,
\end{align}
 the term on the right-hand side being the Euclidean exponential map at $\varphi_\pm(x) \,$. We define
 \[
 H_\delta^1(\R,\gamma_0^*\Dcal') := \Bigset{\xi \in \Gamma^0(\R,\gamma_0^*TM)}{\begin{array}{c}
    \xi(\R) \subseteq \Dcal' \text{ and } \Psi \cdot \xi \in H^1_\delta(\R,\R^m) \text{ for any}    \\
    \text{smooth trivialization } \Psi : \gamma_0^*TM \overset{\sim}{\rightarrow} \Rbar \times \R^m
 \end{array}}  \,\, ,
 \]
 the space of continuous sections $\xi : \R \rightarrow \gamma_0^*TM$ of $\gamma_0^*TM$ which have image in $\Dcal'$ and lie in $H^1_\delta$ with respect to a trivialization.
 We note that $\gamma_0^*TM \rightarrow \Rbar$ admits a smooth global trivialization since $\Rbar$ is contractible and that every smooth map $A \in \Cinfty(\Rbar, \R^{m \times m})$ induces a well-defined map $H^1_\delta(\R,\R^m) \rightarrow H^1_\delta(\R,\R^m) \comma \xi \mapsto A \cdot \xi \,$, so that the above condition is independent of the trivialization $\Psi \,$.\\
 
 \noindent We fix $s_0 > 1$ sufficiently large so that both $\gamma_0([s_0,+\infty)) \subseteq U_+$ and $\gamma_0((-\infty,-s_0]) \subseteq U_- \,$. For every $\xi \in H_\delta^1(\R,\gamma_0^*\Dcal')$ and vectors $v_\pm \in V_\pm $ we now define the section $\widetilde{\xi}_{(\xi,v_-,v_+)}$ of $\gamma_0^*TM$ over $\R$ by
 \begin{align*}
     \widetilde{\xi}_{(\xi,v_-,v_+)}(s) = \begin{cases}
        \xi(s) & \text{, if } s \in [-s_0,s_0] \\
        \xi(s) + \beta(s-s_0) \, (d\varphi_+)_{\gamma_0(s)}^{-1} \cdot v_+ & \text{, if } s \geq s_0 \\
        \xi(s) + \beta(-s-s_0) \, (d\varphi_-)_{\gamma_0(s)}^{-1} \cdot v_- & \text{, if } s \leq - s_0 \,\, .
     \end{cases}
 \end{align*}
 We lastly fix sufficiently small relatively open neighborhoods $V_\pm' \subseteq V_\pm$ of the origin (of dimension $d_\pm$) so that $\widetilde{\xi}_{(\xi,v_-,v_+)}$ has image in $\Dcal$ for every $v_\pm \in V_\pm'$ and $\xi \in H^1_\delta(\R,\gamma_0^* \Dcal') \,$.
 \\
 \noindent Using \eqref{eq: Morse setting - exp gaux vs Euclidean exp} and that $\varphi_\pm$ are slice charts for $Z_\pm$ centered at $z_\pm \,$, we immediately see that
 \begin{equation}
 \label{eq: Morse setting - exp xitilde limits}
 \exp^{\gaux}_{\gamma_0(\pm \infty)} ( \widetilde{\xi}_{(\xi,v_-,v_+)}(\pm \infty))  = \varphi_\pm^{-1}(v_\pm) \in Z_\pm 
 \end{equation}
 for $\xi \in H^1_\delta(\R,\gamma_0^*\Dcal')$ and $v_\pm \in V_\pm' \,$.
 Again by \eqref{eq: Morse setting - exp gaux vs Euclidean exp}, the curve 
 \[
 s \mapsto \gamma(s) := \exp^{\gaux}_{\gamma_0(s)}(\widetilde{\xi}_{(\xi,v_-,v_+)} (s))
 \]
 satisfies \eqref{eq: Morse setting - def Bcal limit condition} for the charts $(U_\pm, \varphi_\pm)$ and therefore is an element of $\BcalMorse \,$.
\\

\noindent We are finally ready to define the charts. The map 
 \begin{align*}
    \Phi : \,\, H^1_\delta(\R,\gamma_0^*\Dcal') \times V_-' \times V_+' \rightarrow \BcalMorse \comma \quad (\xi,v_-,v_+) \mapsto \exp^{\gaux}_{\gamma_0}( \widetilde{\xi}_{(\xi, v_-,v_+)} )
 \end{align*}
 has $C^0$-open image in $\BcalMorse \,$. Consider the collection $\{\Phi\}$ of all such charts, indexed over the $\gamma_0$, slice charts $(U_\pm,\varphi_\pm)$ centered at $\gamma_0(\pm\infty)$, admissible auxiliary metrics $\gaux$ on $M$, zero section neighborhoods $\Dcal' \subseteq \Dcal \subseteq TM \,$, admissible $s_0 >1$ and $V_\pm' \subseteq \varphi_\pm(U_\pm \cap Z_\pm)$ as above. We give $\BcalMorse$ the final topology with respect to this collection.
 The collection $\{\Phi^{-1}\}$ then constitutes a smooth atlas for $\BcalMorse \,$. Figure \ref{fig: Morse setting - chart construction} illustrates the asymptotics of a trajectory in the domain of a chart $\Phi^{-1} \,$.
\\
 \begin{figure}
    \centering
     \begin{tikzpicture}[scale=1, x=8cm , y=7cm]

\draw (0,0,0) -- (1.05,0,0) -- (1.15,0,-4) -- (0.1,0,-4) -- cycle ;
\node at (1.25,0.1,0.5) {$Z_+$} ;

\begin{scope}[
plane x={(1,0,0)}, plane y = {(0.1,0,-4)}, canvas is plane
]
\draw[loosely dashed] (0.5,0.5) circle [x radius=0.35, y radius = 0.3] ;

\node (origin) at (0.5,0.5) {};
\node (gamma_infty) at (0.65,0.5) {};
\node[below,yshift=0.05cm, node font=\small] (v_+) at (0.59,0.5) {$v_+$};
\draw[arrows={-{Stealth[round,slant=0.7]}}] (origin) -- (gamma_infty) ;
\end{scope}

\fill (origin) circle (1.3pt) node[left, xshift=0.1ex, yshift=-0.5ex, node font=\small] {$z_+$} ;
\fill (gamma_infty) circle (1.3pt) node[right, xshift=-0.2ex, node font=\small] {$\gamma(\infty)$} ;

\begin{scope}[
plane origin={(0,0.5,0)}, plane x={(1,0.5,0)}, plane y = {(0.1,0.5,-4)}, canvas is plane
]
\draw[loosely dashed] (0.5,0.5) circle [x radius=0.35, y radius = 0.3] ;  
\end{scope}

\begin{scope}[
plane origin={(0,-0.2,0)}, plane x={(1,-0.2,0)}, plane y = {(0.1,-0.2,-4)}, canvas is plane
]
\draw[loosely dashed] (0.5,0.5) circle [x radius=0.35, y radius = 0.3] ; 
\end{scope}

\draw[loosely dashed] (0.24,0.52,-0.9) -- (0.24,0.08,-0.9)  (0.24,-0.05,-0.9) -- (0.24,-0.12,-0.9);
\draw[loosely dashed] (0.96,0.55,-0.9) -- (0.96,0.08,-0.9) (0.96,-0.05,-0.9) -- (0.96,-0.12,-0.9);
\node at (0.19,0.43,-1) {$U_+$} ;

\begin{scope}[decoration={
    markings,
    mark=at position 0.23 with {\arrow{Stealth[]}}}
    ] 
    \draw[postaction={decorate}] plot[smooth, tension=0.8] coordinates {  (0.62,0.6,-2) (0.61,0.5,-2) (0.6,0.4,-1.9) (0.58,0.3,-1.9) (0.6,0.2,-1.8) (origin)};
\end{scope}
\node[node font=\small] (gamma_0) at (0.585,0.52,-2) {$\gamma_0$} ;

\begin{scope}[decoration={
    markings,
    mark=at position 0.22 with {\arrow{Stealth[]}}}
    ]
     \node (xi_start) at (0.74,0.4,-2.1) {};
    \draw[postaction={decorate}] plot[smooth, tension=0.8] coordinates { (0.77,0.6,-2) (0.75,0.5,-2)  (0.74,0.4,-2.1)  (0.73,0.3,-2)  (0.7,0.2,-2) (0.69,0.1,-2) (gamma_infty)} ;
\end{scope}
\node[node font=\small] (gamma) at (0.775,0.52,-2) {$\gamma$} ;

\begin{scope}[decoration={
    markings,
    mark=at position 0.5 with {\arrow{Stealth[]}}}
    ]
    \draw[densely dashed, postaction=decorate] plot[smooth, tension=0.8] coordinates { (0.58,0.3,-2)  (0.55,0.2,-2) (0.54,0.1,-2) (origin)} ;
\end{scope}
\node[node font=\small] (exp xi) at (0.47,0.23,-2) {$\exp^{\Bar{g}}_{\gamma_0}(\xi)$} ;

\draw[densely dashed, rounded corners] (xi_start) -- (0.73,0.37,-2) [rounded corners=4pt] -- (0.7,0.36,-2) .. controls (0.67,0.35,-2) and (0.63,0.34,-2) .. (0.6,0.33,-2) [rounded corners=5pt] -- (0.58,0.3,-2) ;

\end{tikzpicture}
    \caption{Positive asymptotics for a trajectory $\gamma = \exp^{\gaux}_{\gamma_0}(\widetilde{\xi}_{(\xi,v_-,v_+)})$ in the chart domain of $\Phi^{-1} \,$. The auxiliary metric $\gaux$ is Euclidean on $(U_+,\varphi_+) \,$, so $\gamma(s) =  \exp^{\gaux}_{\gamma_0(s)}(\xi(s)) + \beta(s-s_0) \, v_+$ in the chart $\varphi_+ $ due to \eqref{eq: Morse setting - exp gaux vs Euclidean exp}.
    } 
    \label{fig: Morse setting - chart construction}
\end{figure}
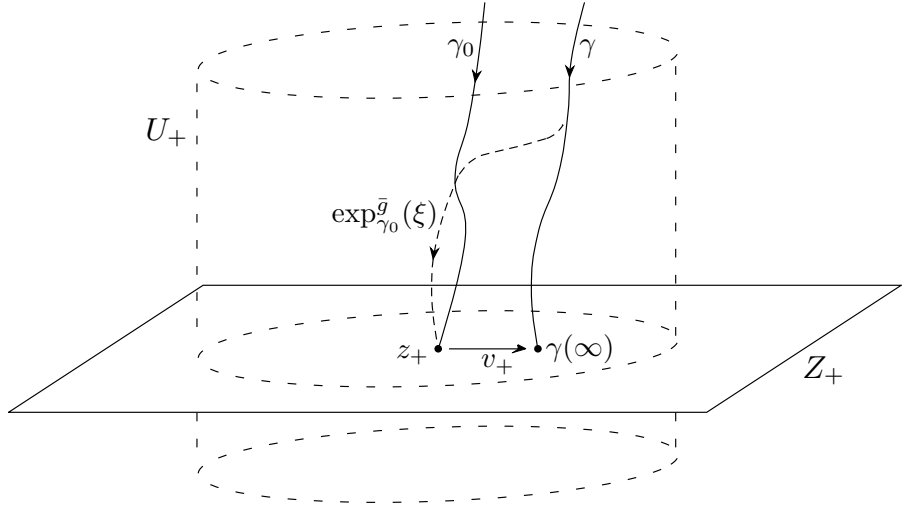
 
 \noindent The above charts $\Phi^{-1}$ are modeled on open subsets of second-countable Banach spaces. It is not difficult to show that the above atlas admits a countable subatlas, so that $\BcalMorse$ is a second-countable Banach manifold. In particular, by a result in \cite{Palais_ParacompactBanachMfds}, $\BcalMorse$ is metrizable.

 \subsubsection{Compactness of $\modulihat(Z_-,Z_+,\zeta_0)$ in $\BcalMorse$}
 \label{subsubsec: Morse setting ambient Banach mfd - compactness in Bcal(Z_-,Z_+)}

 We continue with the setting of Subsection \ref{subsubsec: Morse setting ambient Banach mfd - moduli space unparametrized trajectories}. In particular, we assume $f$ is Morse-Bott along both $Z_-$ and $Z_+$ and have fixed $\zeta_0 \in (\zeta_-,\zeta_+)$ satisfying \eqref{eq: Morse setting - zeta_0}. Let $B > 0$ be a constant as in Lemma \ref{lem: Morse setting - exponential decay}, simultaneously for $Z_-$ and $Z_+ \,$. 

 \begin{theorem}
\label{thm: modulihat(Z_-,Z_+,zeta_0) compact in Bcal(Z_-,Z_+)}
For every $\delta \in (0,B)$, the moduli space $\modulihat(Z_-,Z_+,\zeta_0)$ is a compact subset of $\BcalMorse \,$.
 \end{theorem}
 
\begin{proof}
   Fix an arbitrary $\delta \in (0,B)$. Having chosen the constant $B$ as in Lemma \ref{lem: Morse setting - exponential decay} and using also Remark \ref{rmk: Morse setting - exponential decay derivative}, we conclude that $\modulihat(Z_-,Z_+,\zeta_0)$ is a subset of $\BcalMorse \,$. Because $\BcalMorse$ is metrizable, to show compactness of $\modulihat(Z_-,Z_+,\zeta_0)$ it suffices to show sequential compactness. Given a sequence $(\gamma_n)_n \subseteq \modulihat(Z_-,Z_+,\zeta_0)$, we know from Corollary \ref{cor: Morse setting ambient Banach mfd - C_0 compactness} that a subsequence (which we suppress in the notation) converges both in $\Cinftyloc(\R,M)$ and in $C^0(\R,M)$ to an element $\gamma \in \modulihat(Z_-,Z_+,\zeta_0) \,$. It remains to show that the convergence $\gamma_n \overset{n}{\rightarrow} \gamma$ also is with respect to the given topology on $\BcalMorse \,$. To this end observe that $\gamma$ is an element of $\Cinfty(\Rbar,M)$ since (in local coordinates) all higher order derivatives of $\gamma$ exhibit exponential decay.\footnote{
   This can be seen as in Remark \ref{rmk: Morse setting - exponential decay derivative}.
   } Hence we can verify convergence of $\gamma_n$ to $\gamma$ in a chart $\Phi^{-1}$ for $\BcalMorse$ centered at $\gamma \,$, as constructed in the previous subsection. Choosing the charts $(U_\pm,\varphi_\pm)$ smaller if necessary, we may assume that they are contained in tubular neighborhoods around $Z_\pm$ as in the exponential decay lemma (Lemma \ref{lem: Morse setting - exponential decay}).
   Let $(\xi_n,v^-_n,v^+_n)$ be defined by
   \[
  \gamma_n = \Phi(\xi_n,v^-_n,v^+_n) = \exp^{\gaux}_{\gamma}( \widetilde{\xi}_{(\xi_n, v^-_n,v^+_n)} ) \,\, .
   \]
   We have to show that 
   \[
   (\xi_n,v^-_n,v^+_n) \overset{n}{\rightarrow} (0,0,0) \text{ in } H^1_\delta(\R,\gamma^*TM) \times (\R^{d_-} \times \{0\}) \times ( \R^{d_+} \times \{0\}) \,\, .
   \]
   From \eqref{eq: Morse setting - exp xitilde limits} and the $C^0(\R,M)$-convergence of $\gamma_n$ to $\gamma$, we immediately conclude that $v_n^-$ and $v_n^+$ tend to the origin respectively. Again by the $C^0(\R,M)$-convergence, we can choose $\Bar{s} > s_0 +1$ so that $\gamma_n([\Bar{s},\infty]) \subseteq U_+$ and $\gamma_n([-\infty,-\Bar{s}] \subseteq U_-$ for every $n$ sufficiently large. The differential $(d\varphi_+)_{\gamma}$ of the chart $\varphi_+$ gives a local trivialization of $\gamma^*TM$ over $[\Bar{s},\infty]$ and we consider the trivialized section $\xi_n^{\varphi_+}  := (d \varphi_+)_{\gamma} \cdot \xi_n \in H^1_\delta([\Bar{s},\infty),\R^m) \,$. The $\Cinftyloc$-convergence $\gamma_n \rightarrow \gamma$ and the already proven convergence of $v_n^+$ to the origin imply that
   \[
   \xi_n^{\varphi_+} = (d \varphi_+)_{\gamma} \cdot \widetilde{\xi}_{(\xi_n,v_n^-,v_n^+)} - v_n^+ \overset{\text{\eqref{eq: Morse setting - exp gaux vs Euclidean exp}}}{=} \varphi_+ \circ \gamma_n - \varphi_+ \circ \gamma - v_n^+
   \]
   converges to the zero function in $\Cinftyloc \,$. We claim that
   \begin{align}
    \label{eq: Morse setting - xi --> 0 goal}
    \begin{cases}
    \int_{\Bar{s}}^\infty e^{2 \delta s} \, | \xi_n^{\varphi_+}(s) |^2 \, ds \, \overset{n}{\rightarrow} 0 & \text{and} \\
    \int_{\Bar{s}}^\infty  e^{2 \delta s} \, | (\xi_n^{\varphi_+})'(s) |^2 \, ds \, \overset{n}{\rightarrow} 0 \,\, .
    \end{cases}
   \end{align}
The integrable function $s \mapsto C \, e^{- 2 (B- \delta) s} \,$, for some suitable constant $C >0$ independent of $n$, is a dominating function for both $e^{2 \delta s} \, | \xi_n^{\varphi_+}(s) |^2$ and $e^{2 \delta s} \, | (\xi_n^{\varphi_+})'(s) |^2$ by exponential decay of gradient flow lines (Lemma \ref{lem: Morse setting - exponential decay} and Remark \ref{rmk: Morse setting - exponential decay derivative}). Hence \eqref{eq: Morse setting - xi --> 0 goal} follows from pointwise convergence of $|\xi_n^{\varphi_+}|$ and $|(\xi_n^{\varphi_+})'|$ to the zero function and the dominated convergence theorem. Combining \eqref{eq: Morse setting - xi --> 0 goal} and an analogous version on the interval $(-\infty, - \Bar{s}]$ with $\Cinftyloc(\R,M)$-convergence of $\gamma_n$ to $\gamma$, we conclude that $\xi_n$ tends to the zero section in $H^1_\delta(\R,\gamma^*TM)$. This finishes the proof of the theorem.
\end{proof}

\subsection{Optimality of Upper Bound for the Energy}
\label{subsec: Optimality of Upper Bound for Energy}

The spectral gap as upper bound for the energy in Theorem \ref{mthm: Compactness moduli spaces} is optimal. We will illustrate this briefly in the setting of the previous Subsection \ref{subsec: Morse setting - ambient Banach mfd}. Recall that $f|_{Z_\pm} = \zeta_\pm$ and that we assumed that there is at least one gradient flow line from $Z_-$ to $Z_+ \,$. Suppose $\Crit(f)$ consists precisely of the components $Z_-$ and $Z_+$, so that
\[
E_0 := \mathfrak{S}_f^-(Z_-) = \mathfrak{S}_f^+(Z_+) = \zeta_+ - \zeta_- > 0 \,\, .
\]
Examples as above exist in abundance. (For instance the sphere with the height function.) Consider the \textit{moduli space of parametrized trajectories from $Z_-$ to $Z_+$}
\begin{align*}
    \moduli (Z_-,Z_+) := \Bigset{\gamma \in \Cinfty(\R,M)}{
    \gamma' = \nabla f \circ \gamma  \text{ and }
    \gamma(\pm \infty) \in Z_\pm \text{ exist} } \,\, .
\end{align*}
Elements in $\moduli(Z_-,Z_+)$ have energy $E_0 = \zeta_+ - \zeta_- \,$, hence they restrict to elements of the moduli space $\moduli_{(E_0, \, [0,\infty))} (\Gamma,Z_+) $ introduced in \eqref{eq: moduli space def}. Here $\Gamma$ is the gradient flow line class introduced in \eqref{eq: Morse setting - gradient flow line class}.
Fix a gradient flow line $\gamma \in \moduli(Z_-,Z_+)$ and a sequence $(s_n)_n$ tending to $+\infty \,$. Then $\gamma(\,\cdot - s_n)$ is a sequence in $\moduli(Z_-,Z_+)$ which converges in $\Cinftyloc(\R,M)$ to the constant flow line sitting at $\gamma(-\infty) \,$. In particular, neither $\moduli(Z_-,Z_+)$ nor $\moduli_{(E_0, \,[0,\infty))}(\Gamma,Z_+)$ are compact for the topology of $\Czeroloc$-convergence. A fortiori, $\moduli_{(E_0, \, [0,\infty))}$ is not compact for the topology of uniform convergence.

\begin{remark}
\label{rmk: non-compactness moduli space energy < specgap - counterexample}
The above example also shows that one cannot expect a compactness statement as in Theorem \ref{mthm: Compactness moduli spaces} for the moduli space of gradient flow lines $\gamma$ converging to points in $Z$ as $s \rightarrow \infty$ and with energy $E(\gamma) < \specgap \,$. That is, for compactness one needs a uniform energy bound strictly smaller than $\specgap$.
\end{remark}

\section{Illustrating the Compactness Result in the Floer Setting}
\label{sec: Floer - Illustration}

\noindent The advantage of the rather general definition of gradient flow lines in Definition \ref{def: class of gradient flow lines} will become apparent in the Floer case below. Usually, even if a proof from Morse theory carries over to Floer theory, one has to formally repeat the analogous arguments in the Floer case. This has various reasons, one prominent being that the $L^2$-gradient of the action functional is not a vector field on the underlying Banach manifold $H^1(\bbS^1,W)$, so that Floer cylinders are not actually gradient flow lines. The advantage of the rather general definition of gradient flow line classes we have given (Definition \ref{def: class of gradient flow lines}) is that they encompass Floer cylinders as trajectories into the loopspace $H^1(\bbS^1,W)$ or even better into the metrizable space of smooth loops $\Cinfty(\bbS^1,W)$. So we can indeed apply our compactness result in the Floer realm as well, once conditions \ref{it: ass A - Czeroloc convergence} and \ref{it: ass A - shortening gradient flow lines} have been verified. Condition \ref{it: ass A - Czeroloc convergence} essentially follows directly from Gromov compactness. Condition \ref{it: ass A - shortening gradient flow lines}, however, is much harder to ensure and requires our second main theorem (Theorem \ref{mthm: Floer - exponential decay}), the proof of which will be deferred to Section \ref{sec: Exponential Decay - proof}.

\subsection{Morse-Bott set up for condition \ref{it: ass A - shortening gradient flow lines}}
\label{subsec: Floer - MB set-up for ass A shortening flow lines}

In this section let $(W,\omega)$ be a symplectic manifold of dimension $2n$ that is symplectically aspherical. The latter means
\[
\int_{\bbS^2} f^* \omega = 0 \qquad \Forall f \in \Cinfty(\bbS^2,W) \,\, .
\]
An important special case are exact symplectic manifolds $(W,\omega = d \lambda) \,$. Let $H : W \rightarrow \R $ be an autonomous Hamiltonian. Assuming symplectic asphericity, we can define its Hamiltonian action functional $\actionH : H^1_{\mathrm{contr}}(\bbS^1,W) \rightarrow \R$ on the space of contractible loops in the Sobolev space via
\begin{align*}
\actionH(x) := \int_{\mathbbm{D}^2} \Bar{x}^*\omega - \int_0^1 H(x(t)) \, dt \,\, , 
\end{align*}
where $\Bar{x} \in \Cinfty(\mathbbm{D}^2,W)$ is any choice of filling disk for $x$, that is $\Bar{x}|_{\bbS^1} = x \,$. In the case of an exact symplectic manifold with primitive $\lambda$, we may extend the domain of $\actionH$ to $H^1(\bbS^1,W)$ by setting 
\begin{equation}
\label{eq: action functional exact case}
\actionH(x) := \int_{\bbS^1} x^*\lambda - \int_0^1 H (x) \, dt  \qquad \Forall x \in H^1(\bbS^1,W ) \,\, .
\end{equation}
Let $J = (J_t)_{t \in \bbS^1}$ be a smooth $1$-periodic family of $\omega$-compatible almost complex structures on $W$. This gives rise to a smooth family $g_t := \omega(\,\cdot\, , J_t \,\cdot\,)$ of Riemannian metrics on $W$ and to an induced $L^2$-metric, which at $x \in H^1(\bbS^1,W)$ is defined by
\begin{align}
\label{eq: g_J metric loopspace}
    (g_J)_x(X_1,X_2) := \int_0^1 (g_{t})_{x}(X_1, X_2) \, dt \qquad \Forall X_1,X_2 \in L^2(\bbS^1,x^*TW) \,\, .
\end{align}
We emphasize that the $L^2$-metric gives a bundle metric, not a Riemannian metric. The gradient of $\actionH$ with respect to this metric at the loop $x$ is the $L^2$-vector field along $x$ given by
\[
\nabla^{g_J} \actionH(x) = - J_t (x) \, (\dot x - X_{H}(x)) \,\, ,
\]
where the Hamiltonian vector field $X_H$ is defined by $dH = \omega(\,\cdot\, , X_H) \,$. 
It is important to realize that, while $\actionH$ might only be defined on the component $H^1_{\mathrm{contr}}(\bbS^1,V)$ of contractible loops, the definition of the gradient makes sense on the entire Sobolev space $H^1(\bbS^1,W) \,$. For this reason, we can speak of the gradient $\nabla^{g_J} \actionH (x)$ at an arbitrary loop $x$, not just at a contractible one. The zeroes of the gradient are the one-periodic solutions of $\dot x = X_H(x)$ and are in one-to-one correspondence with fixed points of the time-$1$ flow $\phi_H^1$ of the Hamiltonian vector field. If $x$ is additionally contractible or we are working with an exact symplectic manifold, then $x$ is also a critical point of $\actionH$.
\\
In the following, let us abridge 
\[
\loopspace := \Cinfty(\bbS^1,W) \,\, .
\]
The $\actionH$-gradient flow lines on an interval $I \subseteq \R$ are maps $u : I \rightarrow \loopspace$ that, considered as cylinders $u : I \times \bbS^1 \rightarrow W$, are smooth solutions of the Floer equation
\begin{equation}
\label{eq: Floer eq - Floer illustration}
\partial_s u + J_t(u)\, (\partial_t u - X_H(u)) = 0 \,\, .
\end{equation}
(We do not require $u$ to take on values in the space of contractible loops.) 
To distinguish between $u$ as being defined on an interval versus on a cylinder, we also write
\[
u_s := u(s,\,\cdot\,) \in \loopspace \,\, .
\]
We now make the following Morse-Bott type assumption.

{
\let\realItem\item 
\makeatletter
\NewDocumentCommand\myItem{ o }{%
   \IfNoValueTF{#1}%
      {\realItem}
      {\realItem[#1]\def\@currentlabel{#1}}
}
\makeatother

\setlist[enumerate]{
    before=\let\item\myItem,       
    label=\textnormal{(\arabic*)}, 
    widest=(2')                    
}

\begin{enumerate}
    \item[(MB$_{(H,N)}$)]\label{it: MB} Let $N \subseteq W$ be a closed, connected submanifold so that $N$ is fibered by periodic Hamiltonian flow lines of $H$ of minimal period $1$. Suppose moreover
    \[
    T_p N = \ker (d_p \phi_H^1 - \id) \quad \Forall p \in N \,\, .
    \]
\end{enumerate}
}

\begin{remark}
\label{rmk: MB minimal period}
The minimal period assumption can likely be dropped by lifting the charts in Lemma \ref{lem: tubular nbhd around limit loop} to a covering.
\end{remark}

\noindent One can naturally identify $N$ with the set of Hamiltonian flow lines starting at points of $N$. Define
\begin{align}
\label{eq: Floer - Nloopspace}
    \Nloopspace := \set{x \in \loopspace}{x(t) = \phi_H^t(x(0)) \comma x(0) \in N} \,\, ,
\end{align}
so that $\Nloopspace$ is contained in the zero set of $\nabla^{g_J} \actionH \,$. The natural continuous injective map $N \rightarrow \loopspace \comma \,\, p \mapsto (\phi_H^t(p))_t \,$, is a topological embedding onto its image $\mathcal{N}$. (Here and always, $\loopspace$ carries the $\Cinfty$-topology.) Thus $\Nloopspace$ is a compact and connected topological manifold. If $\Nloopspace$ consists of contractible loops or we are in the exact case, then $\Nloopspace \subseteq \Crit(\actionH)$ and $\actionH$ is constant on its connected critical set $\Nloopspace \,$.
\\

\noindent We can now state our second main result.

\begin{MainThm}
\label{mthm: Floer - exponential decay}
Suppose \ref{it: MB} holds. Then for every $x_0 \in \Nloopspace$ there exists a $\Cinfty$-open neighborhood $\Ucal \subseteq \loopspace$ of $x_0 \,$, a $\Cinfty$-continuous function $\Xi : \Ucal \rightarrow [0,\infty) $ vanishing on $\Nloopspace \cap \Ucal$ and a constant $B >0$ with the following significance.\\
For every $s_0 \in \R$ and every $\actionH$-gradient flow line $u : [s_0,\infty) \rightarrow \loopspace$ with image contained in $\Ucal$ and existing limit (in the $\Cinfty$-topology) $u_\infty = \lim_{s \rightarrow \infty} u_s \in \Nloopspace \,$, it holds
\begin{align*}
     |\partial_s u(s,t)| \leq \Xi(u_{s_0}) \, e^{-B(s-s_0)}  \qquad \Forall s \geq s_0 \comma t \in \bbS^1 \,\, .
\end{align*}
\end{MainThm}

\noindent The conclusion of the theorem is independent of the Riemannian metric used on $W$ to measure the norm of $\partial_s u (s,t) \,$, which is why we did not specify it. The proof of Theorem \ref{mthm: Floer - exponential decay} will be given in Section \ref{sec: Exponential Decay - proof}. As an immediate consequence, we see that condition \ref{it: ass A - shortening gradient flow lines} is satisfied. Let us state this as a separate result.

\begin{corollary}
\label{cor: Floer - ass A shortening gradient flow lines}
Suppose \ref{it: MB} holds.
\begin{enumerate}[label=(\alph*)]
    \item\label{it: Floer - ass A shortening general case} In the general case, assume additionally that the periodic Hamiltonian flow lines in $N$ are contractible loops. Endow $M := \Cinfty_{\mathrm{contr}}(\bbS^1,W)$ with the $\Cinfty$-topology. Set $Z:= \Nloopspace$ and $f := \actionH : M \rightarrow \R$ and define $G : M \rightarrow [0,\infty)$ by
     $G(x) := |\nabla^{g_J} \actionH (x )|_{g_J}^2 \,$. Fix the gradient flow line class, which assigns to a closed interval $I \subseteq \R$ the set
    \begin{align*}
       \Bigset{u : I \rightarrow M }{u : I \times \bbS^1 \rightarrow W \textit{ satisfies \eqref{eq: Floer eq - Floer illustration}}} \,\, .
    \end{align*} 
    \item\label{it: Floer - ass A shortening exact case} In the exact case $\omega = d \lambda \,$, we can even take $M := \loopspace = \Cinfty(\bbS^1,W)$. Again, let $Z := \Nloopspace$ and let $f = \actionH : M \rightarrow \R$ be defined by \eqref{eq: action functional exact case}. Define $G : M \rightarrow [0,\infty)$ and the gradient flow line class by the same formulae as above.
\end{enumerate}
Then condition \ref{it: ass A - shortening gradient flow lines} holds for every $E_0 \geq 0$ and every interval $I = [a,\infty) \,$. 
\end{corollary}

\begin{proof}
    It suffices to verify the assumptions of Lemma \ref{lem: shortening gradient flow lines criterion}. Fix a Riemannian metric on $W$ with induced distance function $d_W : W \times W \rightarrow [0,\infty) \,$.
    Given $x_0 \in \Nloopspace \,$, choose an open neighborhood $\Ucal \subseteq M $ of $x_0 \,$, a continuous function $\Xi : \Ucal \rightarrow [0,\infty)$ vanishing on $\Nloopspace \cap \Ucal$ and a constant $B > 0$ with the property described in Theorem \ref{mthm: Floer - exponential decay}. As coarse metric take the sup-metric $d_{C^0}$ on $M$, defined by
    \[
    d_{C^0}(x,y) := \max_{t \in \bbS^1} \, d_W(x(t), y(t)) \quad \text{ for } x,y \in M \,\, .
    \]
    Clearly $d_{C^0}$ induces the $C^0$-topology on $M$, which is coarser than the given $\Cinfty$-topology. Now let $u : [s_0,\infty) \rightarrow M$ be any gradient flow line with image contained in $\Ucal$ and existing limit $u_{\infty} \in \Nloopspace \cap \Ucal $ (in the $\Cinfty$-topology). We can now estimate
    \begin{align*}
        d_{C^0}(u_{s_0}, u_{\infty}) \leq \sup_{t \in \bbS^1} \int_{s_0}^{\infty} |\partial_s u(s,t) | \, ds \leq \Xi(u_{s_0}) \int_{s_0}^{\infty} e^{-B(s-s_0)} \, ds = \frac{\Xi(u_{s_0})}{B}\,\, .
    \end{align*}
    Thus the function $\frac{1}{B} \, \Xi : \Ucal \rightarrow [0,\infty) $ has the required property stated in Lemma \ref{lem: shortening gradient flow lines criterion}.
\end{proof}

\noindent The exponential coefficient $B > 0$ in Theorem \ref{mthm: Floer - exponential decay} can be chosen uniformly for every loop $x_0 \in \Nloopspace \,$, which is relevant if one wants to embed moduli spaces of Floer cylinders into a Banach manifold of Sobolev maps of weighted exponential decay (choosing the weight smaller than a uniform exponential decay coefficient $B$), similarly to Theorem \ref{thm: modulihat(Z_-,Z_+,zeta_0) compact in Bcal(Z_-,Z_+)}. This can be deduced from abstract nonsense only. To this end, we first state a version of the well-known tube lemma from topology.

\begin{lemma}
\label{lem: tube lemma}
Let $X$ be a compact topological space. Let $\mathbf{P}$ be a predicate taking arguments in $(0,\infty) \times X$ and satisfying

{
\let\realItem\item 
\makeatletter
\NewDocumentCommand\myItem{ o }{%
   \IfNoValueTF{#1}%
      {\realItem}
      {\realItem[#1]\def\@currentlabel{#1}}
}
\makeatother

\setlist[enumerate]{
    before=\let\item\myItem,       
    label=\textnormal{(\arabic*)}, 
    widest=(2')                    
}

\begin{enumerate}
    \item[($\mathbf{P}_1$)]\label{it: P1 - tube lemma} For all $0 < r < R$ and $x \in X$: \quad $\mathbf{P}(R,x) \,\, \Longrightarrow \,\, \mathbf{P}(r,x) \,\, .$
    \item[($\mathbf{P}_2$)]\label{it: P2 - tube lemma} For every $x \in X$ there exists a neighborhood $U \subseteq X$ of $x$ and $r > 0$ so that $\mathbf{P}(r,y) $ holds for every $y \in U \,$.
\end{enumerate}
}
Then there exists $r_0 >0$ so that $\mathbf{P}(r,x)$ holds for all $0 < r \leq r_0$ and $x \in X \,$.
\end{lemma}

\begin{proof}
    For each $x \in X$ choose a neighborhood $U(x) \subseteq X$ of $x$ and $r(x) > 0$ as in \ref{it: P2 - tube lemma}. By compactness, finitely many $U(x_1),\ldots , U(x_n)$ cover $X$. Then $r_0 := \min_{1 \leq i \leq n} r(x_i) > 0 $ has the desired property by \ref{it: P1 - tube lemma}.
\end{proof}

\begin{corollary}
\label{cor: Floer - exponential decay uniform constant}
The exponential decay constant $B > 0$ in Theorem \ref{mthm: Floer - exponential decay} can be chosen uniformly for every loop $x_0 \in \Nloopspace \,$.
\end{corollary}

\begin{proof}
Apply Lemma \ref{lem: tube lemma} to the compactum $\Nloopspace$ and the obvious choice of predicate. To verify \ref{it: P2 - tube lemma}, observe that the choices $(\Ucal,\Xi, B) \,$, working for a particular loop $x_0 \in \Nloopspace$ in Theorem \ref{mthm: Floer - exponential decay}, also work for every loop $x_1 \in \Nloopspace \cap \Ucal \,$.
\end{proof}

\subsection{Theorem \ref{mthm: Compactness moduli spaces} in the context of Symplectic Homology}
\label{subsec: Floer - symplectic homology setting}

We will give an application of Theorem \ref{mthm: Compactness moduli spaces} to Floer cylinders for Hamiltonians appearing naturally in the setting of symplectic homology. We refer to Seidel \cite{Seidel_sympletic_cohomology} for an excellent survey on symplectic homology and Bourgeois-Oancea \cite{Oancea_Bourgeois_MorseBott} for the analytical details.\\

\noindent Let $(W,\lambda)$ be a compact Liouville domain. This means that $\lambda \in \Omega^1(W)$ is the primitive of a symplectic form $\omega := d \lambda$ on $W$ and that the global Liouville vector field $X$ on $W$, defined by $\iota_X \omega = \lambda \,$, is outward pointing on the boundary $\Sigma := \partial W \,$.
The restriction $\alpha := \lambda|_\Sigma$ gives a contact form on $\Sigma$ whose Reeb vector field we denote by $R_\alpha \,$. We endow the completion $\widehat{W} := W \, \cup_\Sigma \,  [1,\infty) \times \Sigma$ with the one-form
\begin{align*}
    \widehat{\lambda} := \begin{cases}
        \lambda & \text{ on } W \\
       r \alpha &  \text{ on } [1,\infty) \times \Sigma \,\, ,
    \end{cases}
\end{align*}
where $r$ designates the coordinate on $[1,\infty) \,$, and the symplectic form $\widehat{\omega} := d \widehat{\lambda} \,$.

\begin{example}
\label{ex: cotangent bundle}
An important example of a compact Liouville domain is the unit disk cotangent bundle over a closed Riemannian manifold, equipped with the canonical one-form. Its completion is the whole cotangent bundle. 
\end{example}

\noindent For every $T \in \R$ let $N^T \subseteq \Sigma$ denote the fixed point set of the time-$T$ Reeb flow $\phi_{R_\alpha}^T \,$. For $T >0$ consider the following Morse-Bott type assumption.

{
\let\realItem\item 
\makeatletter
\NewDocumentCommand\myItem{ o }{%
   \IfNoValueTF{#1}%
      {\realItem}
      {\realItem[#1]\def\@currentlabel{#1}}
}
\makeatother

\setlist[enumerate]{
    before=\let\item\myItem,       
    label=\textnormal{(\arabic*)}, 
    widest=(2')                    
}

\begin{enumerate}
    \item[(ReebMB$_{T}$)]\label{it: Reeb MB} Suppose $N^{T} \subseteq \Sigma$ is a closed submanifold, that every point in $N^{T}$ has minimal Reeb period $T$ and that
    \[
    T_p N^{T} = \ker (d_p \phi_{R_\alpha}^{T} - \id_{T_p \Sigma}) \qquad \Forall p \in N^{T} \,\, .
    \]
\end{enumerate}
}

\noindent To understand the Reeb orbits on $\Sigma$ in a certain spectrum of periods, one now studies Hamiltonians $H : \widehat{W} \rightarrow \R$ with
\begin{enumerate}[label=(\arabic*)]
    \item On $[1,\infty) \times \Sigma$, the Hamiltonian $H$ is of the form $H(r,p) = h(r)$ for some smooth $h : [1,\infty) \rightarrow \R$ with $h' > 0 \comma h'' \geq 0$ and $h'' > 0$ on $[1,\Bar{r})$ for some $\Bar{r} > 1 \,$.
    \item $H|_W$ is a $C^2$-small Morse function with $h(1)- h'(1) \leq H \leq 0$ on $W$.
\end{enumerate}
This is the usual set-up described in \cite[\S3]{Oancea_Bourgeois_MorseBott} and \cite[\S3(c)]{Seidel_sympletic_cohomology}, except the extra condition $h(1) -h'(1) \leq \min_W H \,$, which can be easily achieved. Since $H|_W$ is $C^2$-small, the only $1$-periodic Hamiltonian orbits in $W$ are the critical points of $H$. On the cylindrical end, the Hamiltonian vector field is given by $X_H(r,p) = h'(r) \, R_\alpha(p) \,$, so that the Hamiltonian flow is
\begin{align*}
\phi_H^t(r,p) = (r, \phi_{R_\alpha}^{h'(r)t}(p)) \,\, .
\end{align*}
For $r \geq 1$ let $T:= h'(r)\,$. Then we have a bijection between $T$-periodic Reeb orbits $y$ and $1$-periodic Hamiltonian orbits $(r,y(T\,\cdot\,)) \,$. The action value of a critical point $(r,y(T\,\cdot\,))$ of the action functional $\actionH$, defined by \eqref{eq: action functional exact case}, is
\[
\actionH(r,y(T\,\cdot\,)) = r h'(r) - h(r) \,\, .
\]
The action value therefore only depends on $r$ and moreover is increasing in $r$ because
\begin{align*}
\frac{d}{dr} \big( r h'(r) - h(r) \big)= r h''(r) 
\end{align*}
is nonnegative and even strictly positive for $r < \Bar{r} \,$.
Since $h'' > 0$ on $[1,\Bar{r}) \,$, for every $T \in [h'(1) , h'(\Bar{r}) )$ there exists a uniquely determined $r \geq 1$ with $h'(r) = T \,$.

\begin{lemma}
\label{lem: Reeb MB --> MB}
For $r \in [1, \Bar{r})$ set $T:= h'(r) \,$. If \ref{it: Reeb MB} holds, then every connected component $N$ of $\{r\} \times N^{T}$ satisfies condition \ref{it: MB}.
\end{lemma}

\begin{proof}
    The argument is analogous to \cite[Lemma 3.3]{Oancea_Bourgeois_MorseBott}. The differential of $\phi_H^1$ at $(r,p) \in \{r\} \times N^{T} $ is
\begin{align*}
d_{(r,p)} \phi_H^1 = \left(\begin{matrix}
    \mathrm{id}_{\R} & 0  \\
    h''(r) \, R_\alpha ( \phi_{R_\alpha}^{h'(r)}(p)) & d_p \phi_{R_\alpha}^{h'(r)}
\end{matrix} \right) =  \left(\begin{matrix}
    \mathrm{id}_{\R} & 0  \\
    h''(r) \, R_\alpha (p) & d_p \phi_{R_\alpha}^{T}
\end{matrix} \right) \,\, .
\end{align*}
Since $h''(r) > 0 \,$, the vector $\partial_r$ is not an eigenvector for eigenvalue $1$. Keeping in mind that $R_\alpha(p)$ however is and that $\xi_p= \ker(\alpha_p)$ and $\mathrm{span}\{\partial_r, \, R_\alpha(p)\}$ are invariant under $d_{(r,p)}\phi_H^1 \,$, we conclude that the $1$-eigenspace of $d_{(r,p)} \phi_H^1$ is contained in $T_p \Sigma \,$. By assumption \ref{it: Reeb MB}, it thus equals $ T_p N^{T} = T_{(r,p)} N \,$.
\end{proof}

\begin{lemma}
\label{lem: Reeb MB --> specgap positive}
For $T > 0$ suppose \ref{it: Reeb MB} holds. Then $T$ is isolated in the spectrum
\begin{align*}
    \mathrm{Spec}(\Sigma,\alpha) := \set{\tau > 0}{\text{there exists a } \tau\text{-periodic Reeb orbit}} \,\, .
\end{align*}
\end{lemma}

\begin{proof}
    This is well-known. See \cite[Thm.~23]{FauckThesis} for a proof.
\end{proof}

\noindent Next we relate the above setting to the setting of our compactness theorem. We endow $M :=  \Cinfty(\bbS^1,\widehat{W}) $ with the $\Cinfty$-topology. Consider a fixed $T_0 \in (h'(1),h'(\Bar{r}))$ for which (ReebMB$_{T_0}$) holds. Let $r_0 \in (1,\Bar{r})$ be determined by $h'(r_0) = T_0 \,$. Analogous to \eqref{eq: Floer - Nloopspace}, we now define
\begin{align*}
    Z := &\Bigset{x \in \Cinfty(\bbS^1,\widehat{W})}{x(t) = \phi_H^t(x(0)) \comma x(0) \in \{r_0\} \times N^{T_0}} \\
    = &\Bigset{(r_0,y(T_0 \,\cdot\,))}{y \text{ is a } T_0 \text{-periodic Reeb orbit}} \,\, .
\end{align*}
Then $Z$ is homeomorphic to $N^{T_0}$ and in particular compact.
Let $f := \actionH : M \rightarrow \R$ be the action functional, defined in \eqref{eq: action functional exact case}, for primitive $\widehat{\lambda} \,$.
We have seen above that $f= \actionH$ has constant action value $\zeta := r_0 h'(r_0) - h(r_0)$ on $Z \,$. 
Fix a one-periodic family $J = (J_t)_{t \in \bbS^1}$ of $\widehat{\omega}$-compatible almost complex structures on $\widehat{W} \,$, which is of SFT-type\footnote{
``SFT'' stands for ``Symplectic Field Theory''
} 
on the cylindrical end $[\Bar{r},\infty) \times \Sigma \,$, that is
\begin{enumerate}[label=(\arabic*)]
    \item $J|_{[\Bar{r},\infty) \times \Sigma}$ is $t$-independent and invariant under the map $(r,p) \mapsto (c \,r,p)$ for every $c \geq 1 \,$,
    \item $J|_{[\Bar{r},\infty) \times \Sigma}$ leaves $\xi = \ker(\alpha)$ invariant and $J|_\xi $ is $d \alpha|_\xi$-compatible,
    \item $J (r \partial_r) = R_\alpha \,$.
\end{enumerate}
As before, let $G : M \rightarrow [0,\infty)$ be defined by $G := |\nabla^{g_J} \actionH|_{g_J}^2 \,$. Consider the class $\Gamma$ of $(f,G)$-gradient flow lines, given by
\begin{align*}
   \R \supseteq I \longmapsto \Gamma(I) := \Bigset{u : I \rightarrow M}{\begin{array}{c}
         u : I \times \bbS^1 \rightarrow \widehat{W} \text{ solves \eqref{eq: Floer eq - Floer illustration} and} \\
         u(I \times \bbS^1) \subseteq \widehat{W}({\Bar{r}}) := W \, \cup_\Sigma \, [1,\Bar{r}] \times \Sigma
    \end{array}} \,\, .
\end{align*}

\begin{lemma}
\label{lem: Floer - ass A verification}
Under assumption (ReebMB$_{T_0}$) and with the above choices it holds:
\begin{enumerate}[label=(\alph*)]
    \item\label{it: Floer - specgap positive} The spectral gap $\mathfrak{S}^+_{f,G}(Z)$ is strictly positive. In fact even
    \begin{align*}
       0 < \mathfrak{S}_{f,G}(Z) := \inf\Bigset{|\zeta - f(x)| }{\begin{array}{c} x \in G^{-1}(0) \backslash Z
    \end{array}} \leq \mathfrak{S}^+_{f,G}(Z) \,\, .
    \end{align*}
\end{enumerate}
Let $E_0 \geq 0$ and the interval $I = [s_0,\infty)$ be arbitrary. 
\begin{enumerate}[label=(\alph*),resume]
    \item\label{it: Floer - ass A Czeroloc} Condition \ref{it: ass A variant - Czeroloc convergence} holds for every interval $I_1=[s_1,\infty)$ with $s_0 < s_1 \,$.
    \item\label{it: Floer - ass A shortening gradient flow lines} Condition \ref{it: ass A - shortening gradient flow lines} holds.
\end{enumerate}
\end{lemma}

\begin{proof}
    We prove \ref{it: Floer - specgap positive}. For $p \in \Crit(H|_W) \,$, the action value of the loop sitting constantly at $p$ is
    \[
    -H(p)  \leq  h'(1) - h(1) \leq r h'(r) - h(r) \qquad \Forall r \geq 1 \,\, .
    \]
    As mentioned above, the action value at a $1$-periodic Hamiltonian orbit in $\{r\} \times \Sigma$ is $r h'(r) - h(r) \,$, which is increasing in $r$ and even strictly for $1 \leq r < \Bar{r} \,$. 
    Now \ref{it: Floer - specgap positive} follows since $T_0$ is isolated in the Reeb spectrum (Lemma \ref{lem: Reeb MB --> specgap positive}), so that $r_0$ is isolated in the set of $r \geq 1$ for which there exists a $1$-periodic Hamiltonian orbit in $\{r\} \times \Sigma \,$.
    \\
    Part \ref{it: Floer - ass A Czeroloc} follows from a standard Gromov compactness argument. In more detail, given a sequence of Floer cylinders $u_n : [s_0,\infty) \times \bbS^1 \rightarrow \widehat{W}$ with energy bounded uniformly by $E_0$ and image contained in the compactum $ \widehat{W}(\Bar{r}) \,$. Then, for arbitrary $\eps > 0 \,$, the derivatives of the $u_n$ are uniformly bounded on $[s_0+\eps,\infty) \times \bbS^1 \,$. (For otherwise one could construct a nonconstant $J$-holomorphic sphere bubbling off, contradicting the fact that $\widehat{\omega} = d \widehat{\lambda}$ is exact.) By elliptic regularity, a subsequence of $(u_n)_n$ converges in $\Cinftyloc([s_1,\infty) \times \bbS^1, \widehat{W})$ to a Floer cylinder, whose image must again be contained in $\widehat{W}(\Bar{r})  \,$. This shows \ref{it: Floer - ass A Czeroloc}.
    \\
    Lastly, part \ref{it: Floer - ass A shortening gradient flow lines} follows from Corollary \ref{cor: Floer - ass A shortening gradient flow lines} \ref{it: Floer - ass A shortening exact case}, using that \ref{it: MB} holds for every connected component $N$ of $\{r_0\} \times N^{T_0}$ due to Lemma \ref{lem: Reeb MB --> MB}.
\end{proof}

\begin{theorem}
\label{thm: Floer - compactness result symplectic homology}
For $T_0 \in (h'(1),h'(\Bar{r}))$ assume (ReebMB$_{T_0}$) and let $r_0 \in (1,\Bar{r})$ be determined by $h'(r_0) = T_0 \,$.
Given $0 \leq E_0 < \mathfrak{S}^+_{f,G}(Z)$ and $s_0 < s_1$ as well as a sequence $u_n : [s_0,\infty) \times \bbS^1 \rightarrow \widehat{W}$ of Floer cylinders with
\begin{enumerate}[label=(\arabic*)]
    \item\label{it: Floer compactness - energy bound} energy $E(u_n) = \int_{s_0}^\infty \int_0^1 |\partial_s u_n|_{J_t}^2 \, dt \, ds$ uniformly bounded by $E_0 \,$,
    \item\label{it: Floer compactness - initial loop} $u_n(s_0 , \,\cdot\,) \subseteq \widehat{W}(\Bar{r}) = W \, \cup_{\Sigma} \, [1,\Bar{r}] \times \Sigma$ and
    \item\label{it: Floer compactness - C^0-convergence} $(u_n)(s,\,\cdot\,) \rightarrow (r_0,y_n(T_0\,\cdot\,))$ uniformly as $s \rightarrow \infty \,$, where $y_n$ is a $T_0$-periodic Reeb orbit.
\end{enumerate}
Then a subsequence of $(u_n)_n$ converges in $\Cinfty([s_1,\infty) \times \bbS^1,\widehat{W})$ to a Floer cylinder $u$ with image contained in $\widehat{W}(\Bar{r})$ and for which $u(s,\,\cdot\,)$ converges to some $(r_0,y(T_0\,\cdot\,))$ as $s \rightarrow \infty$ uniformly with all derivatives, where $y$ is a $T_0$-periodic Reeb orbit.
\end{theorem}

\begin{proof}
    Lemma \ref{lem: Floer - ass A verification} shows that assumptions \ref{it: ass A variant - Czeroloc convergence} and \ref{it: ass A - shortening gradient flow lines} in Theorem \ref{thm: Compactness moduli spaces - variant} (the generalization of Theorem \ref{mthm: Compactness moduli spaces}) are satisfied with $I:=[s_0,\infty)$ and $I_1:=[s_1,\infty)\,$.
    \begin{claim}
    The $u_n$ are elements of the moduli space $\moduli_{(E_0,\,I)}(\Gamma,Z) \,$, defined in \eqref{eq: moduli space def}.
    \end{claim}
    \begin{proofClaim}
    Because $J$ is of SFT-type on $[\Bar{r},\infty) \times \Sigma $, for any Floer cylinder $u$ the function $\rho := r \circ u$ satisfies
    \[
    \Delta \rho + \rho \, h''(\rho) \, \partial_s \rho = |\partial_s u|^2_J \geq 0  \quad \text{ on } u^{-1}( (\Bar{r},\infty) \times \Sigma ) \,\, ,
    \]
    see the computation in \cite[\S(3c)]{Seidel_sympletic_cohomology}. Thus $\rho$ obeys a maximum principle and attains no interior local maximum on each nonconstant component of $u^{-1}( (\Bar{r},\infty) \times \Sigma )$. 
    Using this and that the $u_n$ have asymptotics in $\widehat{W}(\Bar{r})$ by \ref{it: Floer compactness - initial loop} and \ref{it: Floer compactness - C^0-convergence}, we conclude that the $u_n$ are confined to $\widehat{W}(\Bar{r}) $ and hence $u_n \in \Gamma(I) \,$. Moreover, the convergence $u_n(s,\,\cdot\,) \rightarrow (r_0, \, y_n(T_0 \,\cdot\,)) \comma s \rightarrow \infty \,$, in \ref{it: Floer compactness - C^0-convergence} even holds in the $\Cinfty(\bbS^1,\widehat{W})$-topology. Indeed, for every sequence $(\Bar{s}_k)_k$ tending to infinity, a subsequence of $(u_n(\,\cdot\,+\Bar{s}_k, \,\cdot\,))_k $ converges in $\Cinftyloc([s_0,\infty) \times \bbS^1,\widehat{W})$ by Gromov compactness and its limit must be $(s,t) \mapsto (r_0, \, y_n(T_0 t))$ by \ref{it: Floer compactness - C^0-convergence}.
    This shows that the $u_n$ are elements of $\moduli_{(E_0, \, I)}(\Gamma,Z) \,$.
    \end{proofClaim}
   \noindent We can therefore apply Theorem \ref{thm: Compactness moduli spaces - variant} and extract a subsequence (suppressed in the notation) of $(u_n)_n$ converging in $C^0(I_1,M)$ to some $u \in \moduli_{(E_0, \,I_1 )}(\Gamma,Z)\,$. The convergence in 
    \[
    C^0(I_1,M) = C^0([s_1,\infty), \, \Cinfty(\bbS^1,\widehat{W}))
    \]
    and the Floer equation \eqref{eq: Floer eq - Floer illustration} then imply that the $u_n$ converge to $u$ in $\Cinfty([s_1,\infty) \times \bbS^1,\widehat{W}) \,$.
\end{proof}

\section{Proof of Theorem \ref{mthm: Floer - exponential decay}}
\label{sec: Exponential Decay - proof}

\noindent The proof of Theorem \ref{mthm: Floer - exponential decay} below is a parametrized version of a classic exponential decay estimate in the Morse-Bott setting. We follow Bourgeois-Oancea \cite[App.~A]{Oancea_Bourgeois_MorseBott} regarding the classic decay estimate. In addition, we adopt some arguments presented in Fauck's thesis \cite[Ch.~2.2]{FauckThesis}, although it is concerned with the Rabinowitz action functional instead of the classical action functional.

\subsection{Preliminaries}
\label{subsec: Exponential Decay - Preliminaries}

In the following, we always work in the setting of Subsection \ref{subsec: Floer - MB set-up for ass A shortening flow lines}. For a given closed, connected submanifold $N \subseteq W$ we assume \ref{it: MB} and define $\Nloopspace \subseteq \loopspace = \Cinfty(\bbS^1,W)$ by \eqref{eq: Floer - Nloopspace}. We abbreviate
\begin{align*}
    2n = \dim(W) \quad \text{and} \quad d := \dim(N) \,\, .
\end{align*}
We fix, once and for all, a loop $x_0 \in \Nloopspace$ for which we need to find a $\Cinfty$-open neighborhood $\Ucal = \Ucal(x_0) \subseteq \loopspace$, a $\Cinfty$-continuous function $\Xi : \Ucal \rightarrow [0,\infty)$ vanishing on $\Nloopspace \cap \Ucal$ and a constant $B > 0$ with the significance described in Theorem \ref{mthm: Floer - exponential decay}. 

\begin{lemma}
\label{lem: tubular nbhd around limit loop}
There exists a tubular neighborhood $U \subseteq W$ of $x_0$ and coordinates
\[
(\vartheta, z) = (\vartheta, z', z'') : \,U \rightarrow \bbS^1 \times \R^{d-1} \times \R^{2n-d}
\]
so that $U \cap N = \{z''=0\}$ and $\partial_{\vartheta} = X_H$ on $U \cap N$ as well as $(\vartheta , z) \circ x_0 (t) = (t,0) \,$.
\end{lemma}

\noindent 
Observe that, for a tubular neighborhood $U$ as above, the set $U \cap N$ is invariant under the Hamiltonian flow $\phi_H^t \,$. This lemma is the only time we use that the Hamiltonian orbits on $N$ have \textit{minimal} period $1$.

\begin{proof}
We follow the argument in \cite[Lemma 29]{FauckThesis}.
The $\bbS^1$-action of the Hamiltonian flow on $N$ is free by assumption \ref{it: MB}, so the slice theorem (see \cite[Thm.~23.5]{Ana_Cannas_Silva}) gives $\bbS^1$-equivariant coordinates $(\vartheta,z')$ on a neighborhood $U_N \subseteq N$ of $x_0 \,$. Now the normal bundle of $N$ in $W$ is trivial over $U_N \,$,\footnote{
$U_N \cong \bbS^1 \times \R^{d-1}$ is homotopy equivalent to a circle and the normal bundle $TW|_{U_N} / TU_N$ is orientable since $TW|_{U_N} \cong T U_N \oplus (TW|_{U_N} / TU_N)$ and $TU_N$ are respectively. 
} 
so a tubular neighborhood of $U_N$ in $W$ yields globally defined coordinates $z'' \,$.
\end{proof}

\noindent We now fix a tubular neighborhood $U$ and coordinates $(\vartheta,z) =(\vartheta,z',z'')$ as in the above lemma. In these coordinates
\begin{equation*}
\Bigset{x \in \Nloopspace}{x(0) \in U}= \Bigset{x \in \Cinfty(\bbS^1,U)}{\begin{array}{c}
    x(t) = (\widetilde{\vartheta} + t, \widetilde{z} \,'\,, 0) \text{ for some}   \\
      (\widetilde{\vartheta} \,, \widetilde{z} \,'\,,0) \in \bbS^1 \times \R^{d-1} \times \R^{2n-d} 
\end{array}}
 \,\, .
\end{equation*}
We now define the \textit{shift operators}
\begin{align*}
   &\shift{\pm} : \, C^0 (\bbS^1,\bbS^1\times \R^{2n-1}) \rightarrow  C^0 (\bbS^1,\bbS^1\times \R^{2n-1}) \\
   &(\shift{\pm}(x))(t) := x(t) \pm (t,0) \quad \text{ for } t \in \bbS^1 \,\, .
\end{align*}
Clearly $\shift{-}$ is the inverse to $\shift{+} $ and $\Bar{x}_0 := \shift{-}(x_0)$ is the constant loop sitting at $(0,0) \in \bbS^1 \times \R^{2n-1} \,$. In view of the above, moreover
\begin{align}
\label{eq: shift^- (Nloopspace)}
\shift{-}\Big( \Bigset{x \in \Nloopspace}{x(0)\in U} \Big) = \set{x\in \Cinfty(\bbS^1,U)}{x \text{ constant and } z''(x) \equiv 0} \,\, .
\end{align}

\noindent Most statements below only make sense for loops in a sufficiently small neighborhood $\Ucal(x_0) \subseteq \Cinfty(\bbS^1,W) $ of $x_0$. When making such a statement, we tacitly assume our loops to be sufficiently close to $x_0 \,$. To begin with, we choose $\Ucal(x_0)$ small enough so that every $x \in \Ucal(x_0)$ has image $x(\bbS^1)$ contained in the tubular neighborhood $U \subseteq W$ of $x_0 \,$. In the following we will choose successively smaller neighborhoods of $x_0$ but, to not clutter notation, we will denote them all by $\Ucal(x_0)$. The final and smallest choice of $\Ucal(x_0)$ will then have the desired property stated in Theorem \ref{mthm: Floer - exponential decay}.

\begin{remark}
\label{rmk: C^0 close to constant loop in S^1 x R^(2n-d)}
 Every loop $x \in C^0(\bbS^1,\bbS^1\times \R^{2n-1})$ sufficiently $C^0$-close to $\xzerobar = \shift{-}(x_0) \equiv (0,0)$ can be lifted to a loop $C^0(\bbS^1,\R^{2n})$ close to the constant loop at the origin in $\R^{2n}$. We choose $\Ucal(x_0)$ small enough so that, for every $x \in \Ucal(x_0)$, the shift $\shift{-}(x)$ can be lifted to a loop in $C^0(\bbS^1,\R^{2n})$ close to the origin. Abusing notation, we will denote a loop in $C^0(\bbS^1,\bbS^1 \times \R^{2n-1})$ and its lift by the same symbol.
\end{remark}

\noindent For every cylinder $u : [s_0,\infty) \rightarrow \loopspace$ with image in $\Ucal(x_0)$, we define its shift $\Bar{u} : [s_0,\infty) \rightarrow \loopspace $ by $\Bar{u}_s := \shift{-}(u_s) \comma s \geq s_0 \,$. Observe that $z(\Bar{u}) = z(u)$ and moreover $\partial_s \Bar{u} = \partial_s u$ and $\partial_t \Bar{u} = \partial_t u - \partial_{\vartheta} \,$. Since $\partial_\vartheta - X_H = 0$ on $\{z''=0\}\,$, by Hadamard's lemma (Lemma \ref{app lem: Hadamard's lemma}) we can write
\[
\partial_{\vartheta} - X_H(\vartheta,z) = S(\vartheta,z) \cdot z'' \qquad \Forall (\vartheta,z) = (\vartheta,z',z'') \in \bbS^1 \times \R^{d-1} \times \R^{2n-d}
\]
with a smooth matrix-valued function $S : \bbS^1 \times \R^{2n-1} \rightarrow \R^{2n \times (2n-d)} \,$. Then $u$ is a solution of the Floer equation \eqref{eq: Floer eq - Floer illustration} iff 
\begin{equation}
\label{eq: Floer eq - shift(u) hybrid version}
\partial_s \Bar{u} + J_t(u) \, \big( \partial_t \Bar{u} + S(u) \cdot z''(\Bar{u}) \big) = 0\,\, .
\end{equation}
We write
\[
H^k := H^k(\bbS^1,\R^{2n}) \comma k \geq 0  \quad \text{and} \quad L^2 := L^2(\bbS^1,\R^{2n}) = H^0
\]
for the Sobolev spaces. Abbreviating $J_0(t) := J_t (x_0(t))$ and $g_t := \omega_{x_0(t)}(\,\cdot\,,J_0(t) \,\cdot\,) \,$, we endow $H^k$ with the inner product induced by $(g_t)_t \,$, that is
\begin{align}
\label{eq: L^2 and H^k inner product}
    \langle X_1, X_2 \rangle_{L^2} := \int_0^1 g_t (X_1(t), X_2(t)) \, dt  \quad \text{and} \quad 
    \langle X_1, X_2 \rangle_{H^k} := \sum_{j=0}^{k} \langle \partial_t^{j} X_1 , \, \partial_t^{j} X_2 \rangle_{L^2} \,\, .
\end{align}
These inner products induce norms that are equivalent to the standard norms on $H^k \,$, induced by the Euclidean inner product on $\R^{2n}$ instead of $(g_t)_t \,$. Notice that $\langle \,\cdot\comma \cdot\,\rangle_{L^2} = (g_J)_{x_0} \,$, the metric $g_J$ being introduced in \eqref{eq: g_J metric loopspace}.
\\

\noindent For every loop $x \in H^k(\bbS^1,\bbS^1 \times \R^{2n-1})$ we define the bounded linear operator
\begin{align}
    \Asf_x := \Asf_x^{(k)} : H^k \rightarrow H^{k-1} \comma \quad \Asf_x \, X := - J_t(\shift{+}(x)) \, \big( \partial_t X + S(\shift{+}(x)) \cdot z''(X) \big) \,\, ,
    \label{eq: operator A_x}
\end{align}
where $z'' : \R^{2n} \rightarrow \R^{2n-d}$ denotes the projection to the last coordinates. Clearly $\Asf_x^{(\ell)} = \Asf_x^{(k)}|_{H^\ell}$ for $k \leq \ell$ and $x \in H^\ell(\bbS^1,\bbS^1 \times \R^{2n-1}) \,$, which is why we drop the superscript $k$ whenever the precise domain of $\Asf_x$ is irrelevant or clear from the context. Consistent with the previous notation, we also abbreviate $S_0(t) := S(x_0(t)) \,$, so that the operator $\Asfxzerobar$ at $\Bar{x}_0 \equiv (0,0)$ can be written as
\begin{align}
\label{eq: Asf_xzerobar}
   ( \Asfxzerobar \, X)(t) = - J_0(t) \, \Big( \partial_t X(t) + S_0(t) \cdot z'' ( X(t)) \Big) \,\, .
\end{align}

\begin{lemma}
\label{lem: x mapsto A_x smooth}
Denote by $\mathcal{L}(H^k,H^{k-1})$ the space of bounded linear operators.
For every $1 \leq k \leq \ell$ the map
\begin{align*}
    \Asf : H^\ell(\bbS^1,\bbS^1\times\R^{2n-1}) \rightarrow \mathcal{L}(H^k , H^{k-1}) \comma \quad \Asf(x) := \Asf_x
\end{align*}
is smooth.
\end{lemma}

\begin{proof}
    The inclusion $H^\ell(\bbS^1,\bbS^1 \times \R^{2n-1}) \hookrightarrow H^k(\bbS^1,\bbS^1\times \R^{2n-1}) $ is smooth, so it suffices to prove the case $k= \ell \,$. The map
    \begin{align*}
   \beta &: \, \Big(H^k(\bbS^1,\R^{2n \times 2n}) \oplus H^k(\bbS^1,\R^{2n \times (2n-d)}) \Big) \times H^k \rightarrow H^{k-1} \comma \\
    &\beta \left( (R_1,R_2) \comma X \right) := R_1 \cdot \partial_t X + R_2 \cdot z''(X)
    \end{align*}
    is a continuous bilinear map. Moreover, one readily checks that
    \begin{align*}
    &\Theta : H^k(\bbS^1,\bbS^1 \times \R^{2n-1}) \rightarrow H^k(\bbS^1,\R^{2n \times 2n}) \times H^k(\bbS^1,\R^{2n \times (2n-d)}) \comma \\
    &\Theta(x) := (-J_t(x) \comma - J_t(x) \cdot S(x) )
    \end{align*}
    and 
    \[
    \shift{+} : H^k(\bbS^1,\bbS^1\times \R^{2n-1}) \rightarrow H^k(\bbS^1,\bbS^1 \times \R^{2n-1})
    \]
    are smooth. The lemma follows since $\Asf_x = \beta( \Theta \circ \shift{+}(x)\comma \cdot\,) \,$.
\end{proof}

\noindent The operators $\Asf_x$ allow us to write the Floer equation very succinctly. Given a cylinder $u : [s_0,\infty) \rightarrow \loopspace$ with image in $\Ucal(x_0)$ and shift $\Bar{u} = \shift{-}(u) \,$, the cylinder $u$ is a solution to the Floer equation \eqref{eq: Floer eq - Floer illustration} iff \eqref{eq: Floer eq - shift(u) hybrid version} holds iff $\partial_s \Bar{u}(s,\,\cdot\,) = \Asf_{\Bar{u}_s} \, \Bar{u}_s$ for all $ s \geq s_0$ iff
\begin{equation}
\label{eq: Floer eq - operator A}
 \frac{d \Bar{u}_s}{ds} \Big|_s = \Asf_{\Bar{u}_s} \, \Bar{u}_s \qquad \Forall s \geq s_0 \,\, ,
\end{equation}
where on the left-hand side the derivative of $s \mapsto u_s$ in $H^k \,$, for any choice of $k$, is taken. Notice that, although $\Bar{u}_s$ is an element of $H^k(\bbS^1,\bbS^1 \times \R^{2n-1}) \,$, we can still plug it (or rather its lift) into $\Asf_{\Bar{u}_s}$ (see Remark \ref{rmk: C^0 close to constant loop in S^1 x R^(2n-d)}). We now collect some decisive facts about the operators $\Asf_x \,$.

{
\let\realItem\item 
\makeatletter
\NewDocumentCommand\myItem{ o }{%
   \IfNoValueTF{#1}%
      {\realItem}
      {\realItem[#1]\def\@currentlabel{#1}}
}
\makeatother

\setlist[enumerate]{
    before=\let\item\myItem,       
    label=\textnormal{(\arabic*)}, 
    widest=(2')                    
}

\begin{enumerate}
    \item[(Fact 1)]\label{it: operators A_x - Asfxzerobar = Hessian} In the fixed coordinates $(\vartheta,z)$, the operator $\Asfxzerobar = \Asfxzerobar^{(1)}$ equals the covariant Hessian $\nabla^2 \actionH (x_0)$ at $x_0 \,$.\footnote{
    \ie the linearization of $\nabla^{g_J} \actionH$ at its zero $x_0$
    } Therefore $\Asfxzerobar$ is an unbounded selfadjoint operator on $L^2$ for the inner product $(g_J)_{x_0}=\langle \,\cdot\comma \cdot\,\rangle_{L^2} \,$.
    \item[(Fact 2)]\label{it: operators A_x - kernel} Inspecting \eqref{eq: operator A_x}, for every loop $x$ the kernel of $\Asf_x$ contains all constant loops $X$ with $z'' \circ X \equiv 0 \,$. Moreover, by the Morse-Bott assumption \ref{it: MB}, these are all of the elements in the kernel of $\Asfxzerobar \,$, that is
    \begin{align*}
        \ker \Asfxzerobar = \set{X \in \Cinfty(\bbS^1,\R^{2n})}{\partial_t X \equiv 0 \text{ and } z''(X) \equiv 0 } \,\, .
    \end{align*}
    In particular the kernel of $\Asfxzerobar = \Asfxzerobar^{(k)}$ is $d$-dimensional and independent of $k$.
    \item[(Fact 3)]\label{it: operators A_x - decomposition} $\Asfxzerobar$ is an elliptic linear partial differential operator of order $1$ that is moreover selfadjoint by \ref{it: operators A_x - Asfxzerobar = Hessian}. Hence, for every $k \geq 0 \,$, we have an $L^2$-orthogonal splitting of $H^k$-closed subspaces\footnote{
    This is the Hodge decomposition of $\Asfxzerobar \,$, see \cite[Cor.~10.4.10]{Nicolaescu_Geometry_of_Manifolds}.
    }
    \begin{equation}
    \label{eq: Hodge decomposition}
    H^k = \ker(\Asfxzerobar) \oplus \Asfxzerobar(H^{k+1}) \,\, .
    \end{equation}
    Since $\ker(\Asfxzerobar)$ consists of constant loops only, the above splitting also is $H^k$-orthogonal.
\end{enumerate}
Although standard, as service to the reader we prove \ref{it: operators A_x - Asfxzerobar = Hessian} as Lemma \ref{app lem: Asf_xzerobar = covariant Hessian} and the assertion about $\ker(\Asfxzerobar)$ in \ref{it: operators A_x - kernel} as Lemma \ref{app lem: kernel Asf_xzerobar} in the appendix.\\
Denote by $\Psf = \Psf^{(k)} : H^k \rightarrow \ker(\Asfxzerobar)$ and $\Qsf = \Qsf^{(k)} : H^k\rightarrow \Asfxzerobar(H^{k+1})$ the projections with respect to the splitting \eqref{eq: Hodge decomposition}. They are orthogonal projections, so
    \begin{equation}
    \label{eq: Qsf orthogononal projection}
       \lVert X \rVert_{H^k}^2 = \lVert \Psf X \rVert_{H^k}^2 + \lVert \Qsf X \rVert_{H^k}^2 \geq \lVert \Qsf X \rVert_{H^k}^2 \qquad \Forall X \in H^k \,\, .
    \end{equation}
    It holds $\Qsf^{(\ell)} = \Qsf^{(k)}|_{H^\ell}$ and $\Psf^{(\ell)} = \Psf^{(k)}|_{H^\ell}$ for $k \leq \ell $ because $\ker(\Asf_{\xzerobar}^{(k)}) = \ker(\Asf_{\xzerobar}^{(\ell)})$ and $\Asfxzerobar(H^{\ell+1})\subseteq \Asfxzerobar(H^{k+1}) \,$.
    For this reason we often omit the superscript.
The following algebraic relations are direct consequences from the definitions and above facts.
\begin{enumerate}
    \item[(Fact 4)]\label{it: operators A_x - Q Asfxzerobar = Asfxzerobar} $\Qsf \, \Asfxzerobar = \Asfxzerobar$ or more precisely $\Qsf^{(k)} \, \Asfxzerobar^{(k+1)} = \Asfxzerobar^{(k+1)} \,$.
    \item[(Fact 5)]\label{it: operators A_x - A Q = A} $\Asf_x \, \Qsf = \Asf_x$ for every loop $x$. Indeed, $\Psf X \in \ker(\Asfxzerobar) \subseteq \ker(\Asf_x)$ by \ref{it: operators A_x - kernel}.
    \item[(Fact 6)]\label{it: operators A_x - partial_t Q = partial_t} $\partial_t \, \Qsf = \partial_t$ since $\ker(\Asfxzerobar)$ consists of constant loops only.
    \item[(Fact 7)]\label{it: operators A_x - Q Asfxzerobar invertible} The bounded operator 
    \[
    \Qsf \, \Asfxzerobar = \Asfxzerobar : \, (\Asfxzerobar(H^{k+2}) , \lVert \,\cdot\, \rVert_{H^{k+1}}) \rightarrow (\Asfxzerobar(H^{k+1}) , \lVert \,\cdot\, \rVert_{H^{k}})
    \]
    is bijective and hence invertible. Consequently, defining the constant $c(k) := \lVert (\Qsf \, \Asfxzerobar)^{-1} \rVert^{-2} \,$, we have
    \begin{align*}
        \lVert \Asfxzerobar \, \Qsf \, X \rVert_{H^k}^2 = \lVert \Qsf \, \Asfxzerobar \, \Qsf \, X \rVert_{H^k}^2 \geq c(k) \, \lVert \Qsf \, X \rVert_{H^{k+1}}^2 \qquad \Forall X \in H^{k+1} \,\, .
    \end{align*}
\end{enumerate}
}

\begin{definition}
\label{def: smooth cylinder}
A map $Z : [s_0,\infty) \rightarrow \Cinfty(\bbS^1,\R^{2n})$ is called \textit{smooth cylinder} if the induced map $Z : [s_0,\infty) \times \bbS^1 \rightarrow \R^{2n}$ is smooth.
\end{definition}

\noindent The following lemma gives an alternative characterization of smooth cylinders. Its proof is straightforward and left to the reader.

\begin{lemma}
\label{lem: smooth cylinder}
Let $Z : [s_0,\infty) \rightarrow \Cinfty(\bbS^1,\R^{2n})$ be a map.
\begin{enumerate}[label=(\alph*)]
    \item\label{it: smooth cylinder - characterization}
   The following are equivalent.
    \begin{enumerate}[label=(\arabic*)]
        \item $Z$ is a smooth cylinder (in the sense of Definition \ref{def: smooth cylinder}).
        \item For every $k \geq 0$ the induced map $Z : [s_0,\infty) \rightarrow H^k$ is smooth (as map into a Banach space).
    \end{enumerate}
    \item\label{it: smooth cylinder - derivative}
    Suppose $Z$ is a smooth cylinder. For every $k \geq 0$ the derivative of $Z : [s_0,\infty) \rightarrow H^k$ at $s \geq s_0$ is $Z'(s) = \partial_s Z(s,\,\cdot\,) \in H^k \,$, where on the right-hand side we mean the partial derivative of $Z : [s_0,\infty) \times \bbS^1 \rightarrow \R^{2n} \,$.
\end{enumerate}
\end{lemma}

\noindent Motivated by part \ref{it: smooth cylinder - derivative} of the previous lemma, for a smooth cylinder $Z$ we will write $\partial_s Z (s)  \in \Cinfty(\bbS^1,\R^{2n})$ for the derivative of $Z : [s_0,\infty) \rightarrow H^k$ at $s$, for any choice of $k \geq 0 \,$. (This is independent of $k$.) Notice that $\partial_s Z$ is again a smooth cylinder.
\\

\noindent The next proposition is essentially a shifted version of Theorem \ref{mthm: Floer - exponential decay}. Justified by Remark \ref{rmk: C^0 close to constant loop in S^1 x R^(2n-d)}, we consider $\xzerobar$ as constant loop at the origin in $\R^{2n} \,$.

\begin{proposition}
\label{prop: exponential decay - localized version}
There exist a constant $B >0 \,$, a $\Cinfty$-open neighborhood $\Ucal(\Bar{x}_0) \subseteq \Cinfty(\bbS^1,\R^{2n})$ of the constant loop $\Bar{x}_0 = \shift{-}(x_0) \equiv 0$ at the origin and a $\Cinfty$-continuous map $\Xi : \Ucal(\Bar{x}_0) \rightarrow [0,\infty)$ with
\[
\Xi(x) = 0 \quad \text{for every constant loop } x \text { with } z''(x) \equiv 0
\]
having the following significance:\\
For every $s_0 \in \R$ and every smooth cylinder $Z : [s_0,\infty) \rightarrow \Cinfty(\bbS^1,\R^{2n})$ with 
\begin{enumerate}[label=(\roman*)]
    \item\label{it: exponential decay - properties Z image in nbhd} image contained in $\Ucal(\Bar{x}_0) \,$,
    \item\label{it: exponential decay - properties Z abstract Floer eq} satisfying the abstract Floer equation
    \begin{equation}
    \label{eq: abstract Floer eq}
    \partial_s Z (s) = \Asf_{Z(s)} \, Z(s) \qquad \Forall s \geq s_0 \,\, ,
    \end{equation}
    \item\label{it: exponential decay - properties Z limit} $ \lim_{s \rightarrow \infty} Z(s) =x$ in the $\Cinfty(\bbS^1,\R^{2n})$-topology, where $x \in \Ucal(\Bar{x}_0)$ is some constant loop with $z''(x) \equiv 0 \,$,
\end{enumerate}
it holds
\[
\lVert \Qsf \, Z(s) \rVert_{H^2} \leq \Xi(Z(s_0)) \, e^{- B(s-s_0)} \qquad \Forall s \geq s_0 \,\, .
\]
\end{proposition}

\noindent Before proving this proposition, which will occupy the rest of the section, we deduce Theorem \ref{mthm: Floer - exponential decay} from it.

\begin{proof}[Proof of Theorem \ref{mthm: Floer - exponential decay}]
    Choose the constant $B > 0$, the $\Cinfty$-open neighborhood $\Ucal(\Bar{x}_0)$ of $\Bar{x}_0 = \shift{-}(x_0)$ and the continuous map $\Xi : \Ucal(\Bar{x}_0) \rightarrow [0,\infty)$ as in Proposition \ref{prop: exponential decay - localized version}. Since $\Asf : \Ucalxzerobar \rightarrow \mathcal{L}(H^2,H^1)$ is continuous (Lemma \ref{lem: x mapsto A_x smooth}), making $\Ucal(\Bar{x}_0)$ smaller if necessary, we can assume that the operator norm of $\Asf_x$ is bounded uniformly for all $x \in \Ucal(\Bar{x}_0) \,$. Using also the continuity of the Sobolev embedding $H^1 \hookrightarrow C^0(\bbS^1,\R^{2n}) \,$, for a smooth cylinder $Z$, satisfying \ref{it: exponential decay - properties Z image in nbhd}\textbf{--}\ref{it: exponential decay - properties Z limit} in Proposition \ref{prop: exponential decay - localized version}, we estimate\footnote{
    This estimate can similarly be found in the proof of \cite[Prop.~32]{FauckThesis}.
    }
    \begin{align*}
        \lVert \partial_s Z (s) \rVert_{C^0} &\leq c_1 \, \lVert \partial_s Z (s) \rVert_{H^1} = c_1 \, \lVert \Asf_{Z(s)} \, Z(s) \rVert_{H^1} \overset{\text{\ref{it: operators A_x - A Q = A}}}{=} c_1 \, \lVert \Asf_{Z(s)} \Qsf \, Z(s) \rVert_{H^1} \\
        &\leq c_1 \, \lVert \Asf_{Z(s)} \rVert_{\mathrm{op}} \, \lVert \Qsf \, Z(s)  \rVert_{H^2} \\
        &\leq c_2 \, \Xi(Z(s_0)) \, e^{-B(s-s_0)} \,\, ,
    \end{align*}
    with constants $c_1,c_2 >0$ independent of $Z$. Now shift $\Ucal(\Bar{x}_0)$ to the neighborhood $\Ucal(x_0) := \shift{+}(\Ucal(\Bar{x}_0))$ of $x_0$ and define the continuous function $\widehat{\Xi} := (c_2 \, \Xi) \circ \shift{-}$ on $\Ucal(x_0) \,$, which vanishes on $\Ucal(x_0) \cap \Nloopspace \,$. Given a Floer cylinder $u : [s_0,\infty) \rightarrow \Ucal(x_0)$ as in Theorem \ref{mthm: Floer - exponential decay}, its shift $Z := \Bar{u} = \shift{-}(u)$ satisfies the abstract Floer equation \eqref{eq: abstract Floer eq} due to \eqref{eq: Floer eq - operator A} and moreover converges asymptotically to a constant loop $Z(\infty)$ with $z''(Z(\infty)) = 0 $ due to \eqref{eq: shift^- (Nloopspace)}. By the above considerations 
    \[
    \lVert \partial_s u(s,\,\cdot\,) \rVert_{C^0} = \lVert \partial_s Z(s) \rVert_{C^0} \leq c_2 \, \Xi(Z(s_0)) \, e^{-B(s-s_0)} = \widehat{\Xi}(u_{s_0}) \, e^{-B(s-s_0)}
    \]
    and this finishes the proof.
\end{proof}

\noindent It remains to prove Proposition \ref{prop: exponential decay - localized version} and we will accomplish this in the next two subsections in a very technical manner.

\subsection{Exponential Decay Estimate for the $L^2$-Norm}
\label{subsec: Exponential Decay - L^2-norm}

\noindent In this section we prove an exponential decay estimate of the form described in Proposition \ref{prop: exponential decay - localized version} but for the $L^2$-norm instead of the $H^2$-norm. This is not quite strong enough but a solid first step. In the subsequent section we will then copy these arguments for higher-order vectors. The proof of the $L^2$-norm estimate below is an adaption of \cite[Prop.~A.1]{Oancea_Bourgeois_MorseBott}. We first recapitulate a particularly easy case of a maximum principle.

\begin{lemma}[Maximum Principle]
\label{lem: maximum principle}
Given $f  \in C^2([s_0,\infty), \,\R)$ and a constant $\delta >0 $ with $f''(s) \geq \delta^2 \, f(s)$ for all $s \geq s_0 \,$. Suppose there exists a sequence $\Bar{s}_n \geq s_0$ tending to infinity with $f(\Bar{s}_n) \overset{n}{\rightarrow} 0 \,$. Then 
\[
f(s) \leq f(s_0) \, e^{-\delta (s-s_0)} \quad \Forall s \geq s_0 \,\, .
\]
\end{lemma}

\begin{proof}
    Define $g(s) := f(s_0) \, e^{-\delta (s-s_0)} \,$. Then $(f-g)'' \geq \delta^2 (f-g) \,$, which shows that $f-g$ does not have a strictly positive local maximum on $(s_0,\infty) \,$. Using moreover $(f-g)(s_0) =0$ and $(f-g)(\Bar{s}_n) \overset{n}{\rightarrow} 0 \,$, we conclude that $f - g \leq 0 \,$.
\end{proof}

\noindent Recall that $\Bar{x}_0$ is the constant loop at the origin in $\R^{2n}$. In the proposition below, let $c(0) > 0$ be the constant from \ref{it: operators A_x - Q Asfxzerobar invertible}.

\begin{proposition}
\label{prop: exponential decay - L^2 estimate}
Let $0 < B < \sqrt{ c(0)}$ be arbitrary. There exists a $\Cinfty$-open neighborhood $\Ucal(\Bar{x}_0) \subseteq \Cinfty(\bbS^1,\R^{2n})$ of $\Bar{x}_0 $ so that for every $s_0 \in \R$ and every smooth cylinder $Z : [s_0,\infty) \rightarrow \Cinfty(\bbS^1,\R^{2n}) \,$, satisfying \ref{it: exponential decay - properties Z image in nbhd}\textbf{--}\ref{it: exponential decay - properties Z limit} in Proposition \ref{prop: exponential decay - localized version}, it holds
\[
\lVert \Qsf \, Z(s) \rVert_{L^2} \leq  \lVert \Qsf \, Z(s_0) \rVert_{L^2} \, e^{-B(s-s_0)} \qquad \Forall s \geq s_0 \,\, .
\]
\end{proposition}

\noindent The function $\Xi : \Ucal(\Bar{x}_0) \rightarrow [0,\infty)$ defined by $\Xi(x) := \lVert \Qsf \, x \rVert_{L^2}$ is $L^2$-continuous, so a fortiori $\Cinfty$-continuous, and moreover vanishes on $\ker \Asfxzerobar \,$. Since $\ker \Asfxzerobar$ consists of the constant loops $x$ with $z''(x) \equiv 0 \,$, the above inequality is nearly as desired in Proposition \ref{prop: exponential decay - localized version}, but unfortunately for the $L^2$-norm instead of the $H^2$-norm.

\begin{proof}[Proof of Proposition \ref{prop: exponential decay - L^2 estimate}]
    We define $\eps := \frac{1}{2} (c(0) - B^2) > 0$ and start by selecting the neighborhood $\Ucal(\Bar{x}_0) \,$. We will successively choose $\Ucal(\Bar{x}_0)$ smaller below but keep writing the same symbol.
    \begin{enumerate}[label=(\arabic*)]
        \item\label{it: exponential decay - L^2 estimate invertibility}  Recall from \ref{it: operators A_x - Q Asfxzerobar invertible} that $\Qsf \,\Asfxzerobar|_{\Asfxzerobar(H^2)}^{\Asfxzerobar(H^1)}$ is invertible. Now invertibility is an open condition and taking the inverse of an operator a continuous assignment. Because $\Asf$ is a continuous map from $\Cinfty(\bbS^1,\bbS^1 \times \R^{2n-1})$ into the space of bounded linear operators (Lemma \ref{lem: x mapsto A_x smooth}), we can choose a $\Cinfty$-open neighborhood $\Ucal(\Bar{x}_0)$ of $\Bar{x}_0$ so that for every $x \in \Ucal(\Bar{x}_0)$ the operator $\Qsf \, \Asf_x |_{\Asfxzerobar(H^2)}^{\Asfxzerobar(H^1)}$ is invertible with
        \begin{align*}
            0 < c(0) - \eps = \left\lVert \Big( \Qsf \, \Asfxzerobar \big|_{\Asfxzerobar(H^2)}^{\Asfxzerobar(H^1)} \Big)^{-1} \right\rVert^{-2} - \eps \leq  \left\lVert \Big( \Qsf \, \Asf_x \big|_{\Asfxzerobar(H^2)}^{\Asfxzerobar(H^1)} \Big)^{-1} \right\rVert^{-2} \,\, .
        \end{align*}
        By this choice of $\Ucalxzerobar$ and \eqref{eq: Qsf orthogononal projection}, for every $x \in \Ucalxzerobar$ it holds
        \begin{equation}
        \label{eq: exponential decay - L^2 estimate invertibility}
        \lVert \Asf_x \, \Qsf \, X \rVert_{L^2}^2 \geq \lVert \Qsf \,\Asf_x \, \Qsf \, X \rVert_{L^2}^2 \geq (c(0) - \eps ) \, \lVert \Qsf \, X \rVert_{H^1}^2 \qquad \Forall X \in H^1  \,\, .
        \end{equation}
        \item\label{it: exponential decay - L^2 estimate adjoint} Denote by $\Asf_x^* : H^1 \rightarrow L^2$ the adjoint of $\Asf_x$ with respect to the inner product $(g_J)_{x_0} = \langle \,\cdot\comma \cdot\,\rangle_{L^2} \,$.\footnote{
We emphasize that we define the domain of the adjoint to be $H^1$, see Definition \ref{app def: adjoint} in the appendix. Being a partial differential operator of order $1$, the operator $\Asf_x$ has an adjoint which lies in $\mathcal{L}(H^1,L^2) \,$, see Appendix \ref{app subsec: adjoints continuous}.
}
The map $\Ucal(\Bar{x}_0) \rightarrow \mathcal{L}(H^1,L^2) \comma x \mapsto \Asf_x^* \,$, is continuous (by Proposition \ref{app prop: adjoint continuous} in the appendix) and hence so is
        \begin{align*}
          \Ucal(\Bar{x}_0) \rightarrow [0,\infty) \comma \quad  x \mapsto \lVert \Asf_x \rVert_{\mathcal{L}(H^1,L^2)} \, \lVert \Asf^*_x - \Asf_x \rVert_{\mathcal{L}(H^1,L^2)} \,\, .
        \end{align*}
    Moreover, this map vanishes at $\Bar{x}_0$ because $\Asfxzerobar$ is selfadjoint, see \ref{it: operators A_x - Asfxzerobar = Hessian}. Thus we can choose $\Ucalxzerobar$ smaller so that
    \begin{equation}
    \label{eq: exponential decay - L^2-estimate adjoint}
    \lVert \Asf_x \rVert_{\mathcal{L}(H^1,L^2)} \, \lVert \Asf^*_x - \Asf_x \rVert_{\mathcal{L}(H^1,L^2)} \leq \eps  \qquad \Forall x \in \Ucalxzerobar \,\, .
    \end{equation}
    \item\label{it: exponential decay - L^2 estimate derivative A_Z Z} We choose $\Ucalxzerobar$ smaller, so that every smooth cylinder $Z : [s_0,\infty) \rightarrow \Cinfty(\bbS^1,\R^{2n})$ with image contained in $\Ucalxzerobar$ satisfies
    \begin{equation}
    \label{eq: exponential decay - L^2-estimate derivative A_Z Z}
        \partial_s Z = \Asf_Z \, Z \quad \Longrightarrow \quad \big\lVert (\partial_s \Asf_{Z(s)}) (s_1) \big\rVert_{\mathcal{L}(H^1,L^2)} < \eps \quad \Forall s_1 \geq s_0 \,\, .
    \end{equation}
    This is possible, as we now explain. Recall from Lemma \ref{lem: x mapsto A_x smooth} that the map $\Asf : H^1(\bbS^1,\bbS^1 \times \R^{2n-1}) \rightarrow \mathcal{L}(H^1,L^2)$ is smooth. We denote by 
    \[
    D_\bullet \Asf : H^1(\bbS^1,\bbS^1 \times \R^{2n-1}) \rightarrow \mathcal{L}(H^1,\mathcal{L}(H^1,L^2)) \comma \quad x \mapsto D_x \Asf \, ,
    \]
    its differential. Then the map
    \begin{align*}
        &\Theta : \, \Ucalxzerobar \rightarrow [0,\infty) \comma \quad \Theta(x) := \left\lVert D_x \Asf \, [\Asf_x \, x] \right\rVert = \Big\lVert D_x \Asf^{(1)} \, \underbrace{[\Asf_x^{(2)} \, x]}_{\in H^1} \Big\rVert_{\mathcal{L}(H^1,L^2)} 
    \end{align*}
    is continuous, where as usual $\Ucalxzerobar$ carries the $\Cinfty$-topology. Moreover $\Theta(\xzerobar) = 0$ since $0 \equiv \xzerobar \in \ker(\Asfxzerobar) \,$. Thus we can choose $\Ucalxzerobar$ smaller so that $0 \leq \Theta(x) < \eps$ for every $x \in \Ucalxzerobar \,$. This choice of $\Ucalxzerobar$ now satisfies \eqref{eq: exponential decay - L^2-estimate derivative A_Z Z}: Given a smooth cylinder $Z$ with image in $\Ucalxzerobar$ and $\partial_s Z = \Asf_Z \, Z \,$, we compute
    \[
    (\partial_s \Asf_Z) (s) = D_{Z(s)} \Asf \, [\partial_s Z(s)] = D_{Z(s)} \Asf \, [\Asf_{Z(s)} \, Z(s)]
    \]
    and hence $\lVert (\partial_s \Asf_Z)(s) \rVert = \Theta(Z(s)) < \eps \,$.
    \end{enumerate}
\noindent Having chosen $\Ucalxzerobar \,$, we now verify that it has the desired property. Let $Z : [s_0,\infty) \rightarrow \Cinfty(\bbS^1,\R^{2n})$ be an arbitrary smooth cylinder satisfying \ref{it: exponential decay - properties Z image in nbhd}\textbf{--}\ref{it: exponential decay - properties Z limit} in Proposition \ref{prop: exponential decay - localized version}. We define the smooth function
\[
f : [s_0,\infty) \rightarrow [0,\infty) \comma \quad f(s) := \frac{1}{2} \, \big\lVert \Qsf \, Z(s) \big\rVert_{L^2}^2 \,\, .
\]
Notice that $\lim_{s \rightarrow \infty} f(s) = 0$ since $\lim_{s \rightarrow \infty} Z(s) \in  \ker \Asfxzerobar = \ker \Qsf $ by \ref{it: exponential decay - properties Z limit} and \ref{it: operators A_x - kernel}. We aim to show
\begin{equation}
\label{eq: exponential decay - L^2 estimate maximum principle}
f''(s) \geq 4 B^2 \, f(s) \qquad \Forall s \geq s_0 \,\, .
\end{equation}
If this has been shown, the maximum principle (Lemma \ref{lem: maximum principle}) yields
\[
\lVert \Qsf \, Z(s) \rVert_{L^2}^2 \leq \lVert \Qsf \, Z(s_0) \rVert_{L^2}^2 \, e^{- 2B (s-s_0)} \qquad \Forall s \geq s_0 \,\, ,
\]
which is our ultimate goal. We now show \eqref{eq: exponential decay - L^2 estimate maximum principle}. To this end, we use that $\Qsf \, \partial_s = \partial_s \, \Qsf$ (because $\Qsf$ is a bounded linear operator), that $\Asf_Z \, \Qsf = \Asf_Z$ by \ref{it: operators A_x - A Q = A}, that $\langle \Qsf X_1, \, \Qsf \, X_2 \rangle_{L^2} = \langle \Qsf X_1, \,  X_2 \rangle_{L^2}$ for all $X_1,X_2$ (because $\Qsf$ is an orthogonal projection) and the abstract Floer equation
\[
\partial_s Z = \Asf_Z \, Z = \Asf_Z \, \Qsf \, Z  \,\, ,
\]
which holds due to assumption \ref{it: exponential decay - properties Z abstract Floer eq}. Let $s \geq s_0$ be arbitrary but fixed (below we drop $s$ in the notation). We compute
\begin{align}
    f''(s) &= \lVert \partial_s ( \Qsf \, Z ) (s) \rVert^2_{L^2} + \big\langle \Qsf \, Z (s) , \, \partial_s^2 (\Qsf \, Z ) (s)\big\rangle_{L^2} \notag\\
    &= \lVert  \Qsf \, \Asf_Z \, \Qsf \, Z \rVert^2_{L^2} + \big\langle \Qsf \, Z , \, \Qsf \, \partial_s (\Asf_Z \, \Qsf \, Z) \big\rangle_{L^2} \notag\\
    &= \lVert  \Qsf \, \Asf_Z \, \Qsf \, Z \rVert^2_{L^2} + \big\langle \Qsf \, Z , \, (\partial_s \Asf_Z) \, \Qsf \, Z + \Asf_Z \, \Qsf \, \partial_s Z) \big\rangle_{L^2} \notag\\
    &= \lVert  \Qsf \, \Asf_Z \, \Qsf \, Z \rVert^2_{L^2} + \big\langle \Qsf \, Z , \, (\partial_s \Asf_Z) \, \Qsf \, Z \big\rangle_{L^2} + \big\langle \Qsf \, Z , \, \Asf_Z^2 \, \Qsf \, Z) \big\rangle_{L^2} \notag\\
    &= \lVert  \Qsf \, \Asf_Z \, \Qsf \, Z \rVert^2_{L^2} + \big\langle \Qsf \, Z , \, (\partial_s \Asf_Z) \, \Qsf \, Z \big\rangle_{L^2} \notag\\
    &\hspace*{3mm}+ \big\langle (\Asf_Z^* - \Asf_Z) \, \Qsf \, Z , \, \Asf_Z \, \Qsf \, Z \big\rangle_{L^2} + \lVert \Asf_Z \, \Qsf \, Z \rVert_{L^2}^2 \,\, . \label{eq: exponential decay - L^2-estimate f''}
\end{align}
Since $Z$ has image contained in $\Ucalxzerobar$, the first and the fourth term on the right can be bounded below by \eqref{eq: exponential decay - L^2 estimate invertibility}, so that
\[
\lVert \Qsf \, \Asf_Z \, \Qsf \, Z \rVert^2_{L^2} +  \lVert \Asf_Z \, \Qsf \, Z \rVert_{L^2}^2 \geq 2 (c(0) - \eps) \, \lVert \Qsf \, Z \rVert_{H^1}^2 \,\, .
\]
On the other hand, using \eqref{eq: exponential decay - L^2-estimate adjoint} and \eqref{eq: exponential decay - L^2-estimate derivative A_Z Z}, the second and third term can be bounded above by
\begin{align*}
   \big|\big\langle \Qsf \, Z , \, (\partial_s \Asf_Z) \, \Qsf \, Z \big\rangle_{L^2} \big|  + \big|\big\langle (\Asf_Z^* - \Asf_Z) \, \Qsf \, Z , \, \Asf_Z \, \Qsf \, Z \big\rangle_{L^2} \big| \leq 2 \, \eps \, \lVert \Qsf \, Z \rVert^2_{H^1} \,\, .
\end{align*}
Thus 
\begin{align*}
   f''(s) \geq  (2 c(0) -  4 \eps) \, \lVert \Qsf \, Z \rVert_{H^1}^2 \geq  (2c(0) - 4 \eps) \, \lVert \Qsf \, Z \rVert_{L^2}^2 = 4 B^2 \, f(s) \,\, ,
\end{align*}
showing \eqref{eq: exponential decay - L^2 estimate maximum principle}.
\end{proof}

\begin{remark}
\label{rmk: mistake by Bourgeois-Oancea}
In the proof of \cite[Prop.~A.1]{Oancea_Bourgeois_MorseBott} it is claimed that the above estimates not only hold for the $L^2$-norm but for every $H^k$-norm. If this was true, replacing the $L^2$-norm with the $H^2$-norm in the above proof would directly prove Proposition \ref{prop: exponential decay - localized version}. However, we believe this to be a mistake. Even if working with the Euclidean $L^2$-inner product, the standard almost complex structure $J_{\mathrm{std}}$ on $\R^{2n}$ and a loop of symmetric asymptotic matrices $S_\infty(t) \,$, which is the setting in \cite{Oancea_Bourgeois_MorseBott}, the asymptotic operator $\Asf_\infty := J_{\mathrm{std}} \, \partial_t + S_\infty $ (which, up to sign, corresponds to our operator $\Asfxzerobar$) is $L^2$-selfadjoint but not $H^1$-selfadjoint. Indeed,
\begin{align*}
\langle \Asf_\infty X_1 , \, X_2 \rangle_{H^1} - \langle X_1 , \Asf_\infty \, X_2 \rangle_{H^1} &= \langle \dot S_\infty \, X_1, \, \partial_t X_2 \rangle_{L^2} - \langle  \partial_t X_1, \, \dot S_\infty \, X_2 \rangle_{L^2} \\
&= \langle - \ddot{S}_\infty \, X_1 - 2 \, \dot{S}_\infty \, \partial_t X_1 , \, X_2 \rangle_{L^2}
\end{align*}
in general does not vanish for all $X_1 , X_2 \in H^2 \,$. It is therefore unclear to us how the analog of \eqref{eq: exponential decay - L^2-estimate adjoint} for the $H^1$-adjoint should be accomplished.
\end{remark}

\subsection{Proof of Proposition \ref{prop: exponential decay - localized version}}
\label{subsec: Exponential Decay - Diagonal Operator}

\noindent We now prove Proposition \ref{prop: exponential decay - localized version} by imitating the proof of Proposition \ref{prop: exponential decay - L^2 estimate}, replacing $Z$ by the vector $(Z,\partial_s Z, \partial_s^2 Z)$ and the operators $\Asf_x$ by operators on $(L^2)^{\oplus 3} \,$. This trick seems to have its origins in the analysis of the asymptotic behavior of pseudoholomorphic curves by Hofer-Wysocki-Zehnder, \cf \cite{HWZ_IV}. It was pointed out and explained to me in great detail by Richard Siefring, for which I would like to express my sincere gratitude. 

Contrary to the original trick however, we have not fixed a single Floer cylinder but instead investigate all asymptotic cylinders whose loops lie close to a fixed ``asymptotic loop''. For this reason, the terms we want to bound should not depend on $s$-derivatives $\partial_s Z$ of a cylinder $Z : [s_0,\infty) \rightarrow \Cinfty(\bbS^1,\R^{2n})$ but only on the loop $Z(s) \,$. The general philosophy of the proof below therefore is to replace all $s$-derivatives by a loop-dependent expression via the abstract Floer equation \eqref{eq: abstract Floer eq}.
\\

\noindent As before, we denote by $D_\bullet \Asf$ the differential of the smooth map $\Asf = \Asf^{(1)} : H^1(\bbS^1,\bbS^1\times \R^{2n-1}) \rightarrow \mathcal{L}(H^1,L^2) \,$, by $D^2_\bullet \Asf$ its second derivative and so on. By the chain rule, for all $ k \leq \ell$ and $x \in H^\ell(\bbS^1,\bbS^1\times \R^{2n-1}) \,$, the diagram
\begin{equation*}
\begin{tikzcd}[row sep=tiny]
    H^\ell \ar[r, "D_x \Asf^{(\ell)}"] \ar[dd, hookrightarrow] & \mathcal{L}(H^\ell, H^{\ell - 1}) \ar[dr, start anchor=east, end anchor= north west]   &   \\ & &  \mathcal{L}(H^\ell, H^{k-1}) \\
    H^k \ar[r,"D_x \Asf^{(k)}"] & \mathcal{L}(H^k, H^{k-1}) \ar[ur, start anchor=east, end anchor=south west] & 
\end{tikzcd}    
\end{equation*}
commutes, where the unlabeled arrows are the obvious linear maps. An analogous statement holds for higher order derivatives.
For this reason we usually omit the superscript when working with the derivatives of the operators $\Asf_x \,$.
As last notational preparation, we introduce
\[
\Lhat^2 := L^2 \oplus L^2 \oplus L^2 \quad \text{ and } \Hhat^k := H^k \oplus H^k \oplus H^k 
\]
and endow them with the direct sum inner products respectively, where $L^2$ and $H^k$ carry the inner products from before (defined in \eqref{eq: L^2 and H^k inner product}).

\begin{proof}[Proof of Proposition \ref{prop: exponential decay - localized version}]
We consider the neighborhood $\Ucalxzerobar$ of $\xzerobar$ either as subset of $\Cinfty(\bbS^1,\R^{2n}) $ or of $\Cinfty(\bbS^1,\bbS^1\times\R^{2n-1}) \,$, whatever is more convenient, which is justified by Remark \ref{rmk: C^0 close to constant loop in S^1 x R^(2n-d)}.
We will make $\Ucalxzerobar$ smaller in the course of the proof but keep writing $\Ucalxzerobar \,$. 
The proof of the proposition is subdivided into two parts.

\textbf{Part I}: 
Let $Z : [s_0,\infty) \rightarrow \Cinfty(\bbS^1,\R^{2n})$ be a smooth cylinder with image contained in $\Ucal(\xzerobar)$ and satisfying the abstract Floer equation $\partial_s Z = \Asf_Z \, Z \,$. Deriving the abstract Floer equation twice, we obtain
\begin{align}
\label{eq: abstract Floer eq - higher derivatives}
\left(\begin{matrix}
    \partial_s Z \\ \partial_s^2 Z \\ \partial_s^3 Z
\end{matrix} \right)  = \left(\begin{matrix}
\Asf_Z & 0 & 0 \\
\partial_s \Asf_Z & \Asf_Z & 0 \\
\partial_s^2 \Asf_Z & 2 \, \partial_s \Asf_Z & \Asf_Z
\end{matrix}\right) \, \left( \begin{matrix}
    Z \\ \partial_s Z \\ \partial_s^2 Z
\end{matrix} \right) \quad .
\end{align}
Using the abstract Floer equation, we eliminate the $s$-derivatives:
\begin{align}
    \partial_s \Asf_Z &= D_{Z} \Asf \, [\partial_s Z] = D_Z \Asf \, [\Asf_Z \,Z] \label{eq: exponetial decay - partial_s A_z}\\
    \partial_s^2 \Asf_Z &= \partial_s \big( D_Z \Asf \, [\Asf_Z \,Z] \big) 
    \notag\\
    &= (\partial_s D_Z \Asf) \, [\Asf_Z \, Z] + D_Z \Asf \, [\partial_s ( \Asf_Z \, Z)]  \notag\\
    &= D^2_Z \Asf \, [\Asf_Z \, Z , \, \Asf_Z \, Z] + D_Z \Asf \, [D_Z \Asf \, [\Asf_Z \, Z] \, Z] + D_Z \Asf \, [\Asf_Z^2 \, Z] \,\, . \label{eq: exponential decay - partial_s^2 A_Z}
\end{align}
The last expression depends only on the value of $Z(s)$ in the loopspace. Introducing the loop-dependent operator $\mathsf{B}_x \in \mathcal{L}(H^1,L^2) \,$, for $x \in \Ucalxzerobar \,$, by
\begin{equation}
\label{eq: operator B_x}
  \mathsf{B}_x := D^2_x \Asf \, [\Asf_x \, x , \, \Asf_x \, x] + D_x \Asf \, [D_x \Asf \, [\Asf_x \, x] \, x] + D_x \Asf \, [\Asf_x^2 \, x]  \,\, ,
\end{equation}
we therefore have $\partial_s^2 \Asf_Z (s) = \mathsf{B}_{Z(s)} \,$. Similarly, we define the loop-dependent operator $\Asfhat_x \in \mathcal{L}(\Hhat^1,\Lhat^2) \,$, for every $x \in \Ucalxzerobar \,$, by
\begin{equation}
\label{eq: operator Ahat_x}
    \Asfhat_x := \left( \begin{matrix}
        \Asf_x & 0 & 0 \\
        D_x \Asf \, [\Asf_x \, x] & \Asf_x & 0 \\
        \mathsf{B}_x & 2 \, D_x \Asf \, [\Asf_x \, x] & \Asf_x
    \end{matrix} \right) \,\, .
\end{equation}
We have defined $\Asfhat_x$ in such a way that, for a cylinder $Z$ satisfying $\partial_s Z = \Asf_Z Z \,$, it holds
\begin{align}
\label{eq: Asfhat_Z = derivatives A_z}
       \Asfhat_{Z(s)} =  \left(\begin{matrix}
\Asf_{Z(s)} & 0 & 0 \\
\partial_s \Asf_Z (s) & \Asf_{Z(s)} & 0 \\
\partial_s^2 \Asf_Z (s) & 2 \, \partial_s \Asf_Z (s) & \Asf_{Z(s)}
\end{matrix}\right) \,\, .
\end{align}
Abbreviating
\[
\Zhat(s) := (Z(s), \,\partial_s Z (s), \,\partial_s^2 Z (s)) \in \Cinfty(\bbS^1,\R^{2n})^{\oplus 3} = \bigcap_{k \geq 0} \Hhat^k \,\, ,
\]
we can now rewrite \eqref{eq: abstract Floer eq - higher derivatives} very succinctly as
\begin{equation}
\label{eq: abstract Floer eq - diagonal version}
\partial_s \Zhat (s) = \Asfhat_{Z(s)} \, \Zhat(s) \qquad \Forall s \geq s_0 \,\, .
\end{equation}
Since $\xzerobar \in \ker(\Asfxzerobar) \,$, the operator $\Asfhat_{\xzerobar}$ at $\xzerobar$ is
\begin{equation*}
    \Asfhat_{\xzerobar} = \mathrm{diag}(\Asfxzerobar, \Asfxzerobar, \Asfxzerobar) = \left(\begin{matrix}
        \Asfxzerobar & 0 & 0 \\
        0 & \Asfxzerobar & 0 \\
        0 & 0 & \Asfxzerobar
    \end{matrix} \right) \,\, ,
\end{equation*}
that is $\Asfhat_{\xzerobar}$ is a diagonal operator of $\Asfxzerobar \,$. Therefore, the diagonal operators
\begin{align*}
   \Psfhat = \Psfhat^{(k)} := \mathrm{diag}(\Psf^{(k)},\Psf^{(k)},\Psf^{(k)})  \quad \text{and} \quad \Qsfhat = \Qsfhat^{(k)} := \mathrm{diag}(\Qsf^{(k)},\Qsf^{(k)}, \Qsf^{(k)}) 
\end{align*}
are the projections with respect to the $\Lhat^2$- and $\Hhat^k$-orthogonal splitting
\[
\Hhat^k = \ker (\Asfhat_{\xzerobar}) \oplus \Asfhat_{\xzerobar}(\Hhat^{k+1}) = \ker(\Asfxzerobar)^{\oplus 3} \oplus \Asfxzerobar(H^{k+1})^{\oplus 3} \,\, .
\]
Since $\Asfxzerobar$ is $\langle\,\cdot\comma\cdot\,\rangle_{L^2}$-selfadjoint by \ref{it: operators A_x - Asfxzerobar = Hessian}, the diagonal operator $\Asfhat_{\xzerobar}$ is selfadjoint for
\[
\langle\,\cdot\comma\cdot\,\rangle_{\Lhat^2} = \langle\,\cdot\comma\cdot\,\rangle_{L^2} \oplus \langle\,\cdot\comma\cdot\,\rangle_{L^2} \oplus \langle\,\cdot\comma\cdot\,\rangle_{L^2} \,\, .
\]
\begin{claim}
$\Asfhat_Z \, \Qsfhat = \Asfhat_Z$ for every smooth cylinder $Z : [s_0,\infty) \rightarrow \Cinfty(\bbS^1,\R^{2n})$ with image contained in $\Ucalxzerobar$ and satisfying $\partial_s Z = \Asf_Z \, Z \,$.
\end{claim}
\begin{proofClaim}
    It holds $ \Asf_Z \, \Qsf = \Asf_Z$ by \ref{it: operators A_x - A Q = A}. Derive this twice to obtain
    \begin{align*}
       (\partial_s \Asf_Z) \, \Qsf = \partial_s \Asf_Z \quad \text{and} \quad  (\partial_s^2 \Asf_Z) \, \Qsf = \partial_s^2 \Asf_Z
    \end{align*}
    and therefore
    \begin{equation*}
    \Asfhat_Z \, \Qsfhat  = \left(\begin{matrix}
\Asf_Z & 0 & 0 \\
\partial_s \Asf_Z & \Asf_Z & 0 \\
\partial_s^2 \Asf_Z & 2 \, \partial_s \Asf_Z & \Asf_Z
\end{matrix}\right) \, \left(\begin{matrix}
        \Qsf & 0 &  0\\
        0 & \Qsf & 0 \\
        0 & 0 & \Qsf
    \end{matrix} \right) =  \left(\begin{matrix}
\Asf_Z & 0 & 0 \\
\partial_s \Asf_Z & \Asf_Z & 0 \\
\partial_s^2 \Asf_Z & 2 \, \partial_s \Asf_Z & \Asf_Z
\end{matrix}\right) = \Asfhat_Z \,\, .
    \end{equation*}
\end{proofClaim}
\noindent For given cylinder $Z : [s_0,\infty) \rightarrow \Cinfty(\bbS^1,\R^{2n}) $ satisfying the abstract Floer equation, as before, set $\Zhat = (Z,\partial_s Z, \partial_s^2 Z)$ and define
\begin{equation*}
   f = f_Z : [s_0,\infty) \rightarrow [0,\infty) \comma \quad f(s) := \frac{1}{2} \, \lVert \Qsfhat \, \Zhat (s) \rVert_{\Lhat^2}^2 \,\, .
\end{equation*}
Due to \eqref{eq: abstract Floer eq - diagonal version}, using also that $\Asfhat_Z \, \Qsfhat = \Asfhat_Z$ by the claim and that $\Qsfhat$ is an $\Lhat$-orthogonal projection, the computation in \eqref{eq: exponential decay - L^2-estimate f''} now goes through verbatim. This yields
\begin{align}
    f_Z''(s) &= \lVert  \Qsfhat \, \Asfhat_Z \, \Qsfhat \, \Zhat \rVert^2_{\Lhat^2} + \big\langle \Qsfhat \, \Zhat , \, (\partial_s \Asfhat_Z) \, \Qsfhat \, \Zhat \big\rangle_{\Lhat^2} \notag\\
    &\hspace*{3mm}+ \big\langle (\Asfhat_Z^* - \Asfhat_Z) \, \Qsfhat \, \Zhat , \, \Asfhat_Z \, \Qsfhat \, \Zhat \big\rangle_{\Lhat^2} + \lVert \Asfhat_Z \, \Qsfhat \, \Zhat \rVert_{\Lhat^2}^2 \,\, . \label{eq: exponential decay - H^2-estimate f'' equality}
\end{align}
Here $\Asfhat_x^* \in \mathcal{L}(\Hhat^1,\Lhat^2)$ denotes the $\Lhat^2$-adjoint of $\Asfhat_x \,$, for every loop $x \in \Ucalxzerobar \,$. The adjoint of $\Asfhat_x$ exists and lies in $\mathcal{L}(\Hhat^1,\Lhat^2)$ because it is a matrix of partial differential operators of order $1$.\footnote{
To see this, stare at \eqref{eq: operator Ahat_x} and recall the proof of Lemma \ref{lem: x mapsto A_x smooth} to compute the derivatives $D_x \Asf \, [\Asf_x x] \,$.
}
Notice moreover that the map $s \mapsto \Asfhat_{Z(s)} \in \mathcal{L}(\Hhat^1,\Lhat^2)$ is smooth. This follows from its representation in \eqref{eq: Asfhat_Z = derivatives A_z} and from $s \mapsto \Asf_{Z(s)}$ being smooth by Lemma \ref{lem: x mapsto A_x smooth} and Lemma \ref{lem: smooth cylinder}.
\\

\noindent We are now in a position completely analogous to the one in the proof of Proposition \ref{prop: exponential decay - L^2 estimate} and we continue by ensuring the analogs of \eqref{eq: exponential decay - L^2 estimate invertibility}, \eqref{eq: exponential decay - L^2-estimate adjoint} and \eqref{eq: exponential decay - L^2-estimate derivative A_Z Z}.

\begin{claim}
  We can choose $\Ucalxzerobar$ smaller and constants $c > 0$ and $\eps \in (0,\frac{c}{2})$ so that the following hold.
  \begin{enumerate}[label=(\arabic*)]
      \item\label{it: exponential decay - H^2-estimate invertibility}
      For all $x \in \Ucalxzerobar$ it holds
      \begin{equation}
      \label{eq: exponential decay - H^2-estimate invertibility}
      \lVert \Asfhat_x \, \Qsfhat \, \widehat{X} \rVert_{\Lhat^2}^2 \geq \lVert \Qsfhat \, \Asfhat_x \, \Qsfhat \, \widehat{X} \rVert_{\Lhat^2}^2 \geq (c-\eps) \, \lVert \Qsfhat \, \widehat{X} \rVert_{\Hhat^1}^2 \qquad \Forall \widehat{X} \in \Hhat^1 \,\, .
      \end{equation}
      \item\label{it: exponential decay - H^2-estimate adjoint} For all $x \in \Ucalxzerobar$ it holds
       \begin{equation}
    \label{eq: exponential decay - H^2-estimate adjoint}
    \lVert \Asfhat_x \rVert_{\mathcal{L}(\Hhat^1,\Lhat^2)} \, \lVert \Asfhat^*_x - \Asfhat_x \rVert_{\mathcal{L}(\Hhat^1,\Lhat^2)} \leq \eps   \,\, .
    \end{equation}
    \item\label{it: exponential decay - H^2-estimate derivative Ahat_Z Zhat}
    Every smooth cylinder $Z : [s_0,\infty) \rightarrow \Cinfty(\bbS^1,\R^{2n})$ with image in $\Ucalxzerobar$ satisfies
    \begin{equation}
    \label{eq: exponential decay - H^2-estimate derivative Ahat_Z Zhat}
        \partial_s Z = \Asf_Z \, Z \quad \Longrightarrow \quad \big\lVert (\partial_s \Asfhat_{Z(s)}) (s_1) \big\rVert_{\mathcal{L}(\Hhat^1,\Lhat^2)} < \eps \quad \Forall s_1 \geq s_0 \,\, .
    \end{equation}
  \end{enumerate}
\end{claim}
\begin{proofClaim}
   We will successively choose $\Ucalxzerobar$ smaller.
   \\
   \noindent We start by ensuring \ref{it: exponential decay - H^2-estimate invertibility}. This follows as in the proof of \eqref{eq: exponential decay - L^2 estimate invertibility}. Just notice that 
   \[
   \Qsfhat \, \Asfhat_{\xzerobar} = \mathrm{diag}(\Qsf \, \Asfxzerobar) \, : \,\,\Asfhat_{\xzerobar}(\Hhat^2) \rightarrow \Asfhat_{\xzerobar}(\Hhat^1)
   \]
   is invertible. Now define $c^{-1/2}$ to be the operator norm of its inverse and choose $0 < \eps < \frac{c}{2}$ arbitrary.
    \\
   \noindent Next \ref{it: exponential decay - H^2-estimate adjoint} follows similar to \eqref{eq: exponential decay - L^2-estimate adjoint}. Indeed, $\Asfhat_{\xzerobar}$ is $\Lhat^2$-selfadjoint and the assignment $\Ucalxzerobar \rightarrow \mathcal{L}(\Hhat^1,\Lhat^2) \comma x \mapsto \Asfhat_x^* \,$, is continuous. To see the latter, it suffices to show that the adjoints of the matrix components of $\Asfhat_x\,$, given in \eqref{eq: operator Ahat_x}, depend continuously on $x$. To this end, observe that each component is a first-order partial differential operator with coefficient functions varying continuously in the $\Cinfty(\bbS^1)$-topology if $x$ varies in the $\Cinfty(\bbS^1,\R^{2n})$-topology. Now apply Proposition \ref{app prop: adjoint continuous} in the appendix.
    \\ 
   \noindent Lastly, we deal with \ref{it: exponential decay - H^2-estimate derivative Ahat_Z Zhat}. This is the cumbersome part because we have to eliminate $s$-derivatives again. For a smooth cylinder $Z$ satisfying the abstract Floer equation \eqref{eq: abstract Floer eq}, we have
   \begin{align*}
    \partial_s \Asfhat_Z = \partial_s \left(\begin{matrix}
 \Asf_Z & 0 & 0 \\
\partial_s \Asf_Z &  \Asf_Z & 0 \\
\partial_s^2 \Asf_Z & 2 \, \partial_s \Asf_Z & \Asf_Z
\end{matrix}\right)  
=  
\left(\begin{matrix}
D_Z \Asf \, [\Asf_Z \, Z] & 0 & 0 \\
\mathsf{B}_Z & D_Z \Asf \, [\Asf_Z \, Z] & 0 \\
\partial_s \mathsf{B}_Z & 2 \, \mathsf{B}_Z & D_Z \Asf \, [\Asf_Z \, Z]
\end{matrix}\right) 
   \end{align*}
by \eqref{eq: exponetial decay - partial_s A_z}, \eqref{eq: exponential decay - partial_s^2 A_Z}, \eqref{eq: operator B_x} and \eqref{eq: Asfhat_Z = derivatives A_z}.
Repeatedly invoking the abstract Floer equation to eliminate $s$-derivatives, one can see (see Appendix \ref{app subsec: operator C_x}) that $\partial_s \mathsf{B}_Z (s) = \mathsf{C}_{Z(s)} \,$, where $\mathsf{C} : \Ucalxzerobar \rightarrow \mathcal{L}(H^1,L^2) \comma x \mapsto \mathsf{C}_x \,$, is the continuous map defined by the following intimidating formula:
\begin{align}
    \mathsf{C}_x &:= D_x^3  \Asf \, [\Asf_x x , \, \Asf_x x , \, \Asf_x x] + 3 \, D_x^2  \Asf \, [\Asf_x x , \, \Asf_x^2 x] + 3 \, D_x^2  \Asf \, [\Asf_x x , \, D_x  \Asf \, [\Asf_x x] \, x] \notag\\
    &\hspace*{3mm}+ 2 \, D_x \Asf \, [D_x \Asf \, [\Asf_x x] \, \Asf_x x] +   D_x \Asf \, \left[ D_x \Asf \,\left[ D_x \Asf \, [\Asf_x x] \, x \right] \, x \right] +  D_x \Asf \, [D_x \Asf \, [\Asf_x^2 \, x] \, x]\notag\\
    &\hspace*{3mm}+  D_x \Asf \, \left[ D_x^2 \Asf \, [\Asf_x x , \, \Asf_x x] \, x 
\right] + D_x \Asf \, [\Asf_x \, D_x \Asf \, [\Asf_x  x] \, x ] + D_x \Asf \, [\Asf_x^3 \, x] \,\, .   \label{eq: operator C_x}
\end{align}
The only relevant property of $\mathsf{C} \,$, beside its continuity, is that its evaluation at $\xzerobar$ is $\mathsf{C}_{\xzerobar} = 0 \in \mathcal{L}(H^1,L^2)$ because every summand contains an application of a linear operator to the zero loop $\xzerobar\,$. For the same reason, also $\mathsf{B}_{\xzerobar} = 0 \,$. Consequently, the map
\begin{align*}
  \Theta :  \Ucalxzerobar \rightarrow [0,\infty) \comma \quad \Theta(x) := \left\lVert \left( \begin{matrix}
      D_x \Asf \, [\Asf_x \, x] &  0 & 0 \\
      \mathsf{B}_x &  D_x \Asf \, [\Asf_x \, x] & 0 \\
     \mathsf{C}_x  & 2 \, \mathsf{B}_x &  D_x \Asf \, [\Asf_x \, x]
  \end{matrix} \right) \right\rVert_{\mathcal{L}(\Hhat^1,\Lhat^2)}
\end{align*}
is continuous and vanishes at $\xzerobar \,$. We can therefore choose $\Ucalxzerobar$ smaller, so that $\Theta(x) < \eps$ for every $x \in \Ucalxzerobar \,$. If $Z$ is an arbitrary cylinder satisfying the abstract Floer equation \eqref{eq: abstract Floer eq} and having image in $\Ucalxzerobar \,$, then the above calculation shows
\[
\lVert (\partial_s \Asfhat_Z)(s) \rVert_{\mathcal{L}(\Hhat^1,\Lhat^2)} = \Theta(Z(s)) < \eps \,\, ,
\]
so that we have established \eqref{eq: exponential decay - H^2-estimate derivative Ahat_Z Zhat}.
\end{proofClaim}

\noindent Choose $\Ucalxzerobar$ and constants $c >0$ and $0 < \eps < c/2$ as in the previous claim and define
\[
B := \sqrt{c - 2 \eps} > 0 \,\, .
\]
Define moreover a continuous map $\mathsf{E} : \Ucalxzerobar \rightarrow \Lhat^2$ by
\[
\mathsf{E}(x) := \big(x , \, \Asf_x x , \, D_x \Asf \, [\Asf_x x] \, x + \Asf_x^2 \, x\big) \in \Cinfty(\bbS^1,\R^{2n})^{\oplus 3} \,\, .
\]
Observe that
\begin{equation}
\label{eq: Qhat E(x) = 0 for x constant z''(x)=0}
\Qsfhat \, \mathsf{E}(x) = 0 \quad \text{if } x \text{ is a constant loop with } z''(x) \equiv 0 
\end{equation}
because for every constant loop $x$ with $z''(x) \equiv 0$ we have $x \in \ker \Asf_x \cap \ker \Asfxzerobar$ by \ref{it: operators A_x - kernel} and hence $\mathsf{E}(x) = (x,0,0) \in \ker \Qsfhat \,$.

\begin{claim}
Let $Z : [s_0,\infty) \rightarrow \Cinfty(\bbS^1,\R^{2n})$ be a smooth cylinder with image in $\Ucalxzerobar \,$ and satisfying the abstract Floer equation \eqref{eq: abstract Floer eq}. Define $\Zhat = (Z,\partial_s Z , \partial_s^2 Z) \,$.
\begin{enumerate}[label=(\alph*)]
    \item\label{it: exponential decay estimate - Zhat = E(Z)}
    It holds $\Zhat(s) = \mathsf{E}(Z(s))$ for every $s \geq s_0 \,$.
    \item\label{it: exponential decay estimate - maximum principle Zhat} Suppose additionally that the limit $Z(\infty) := \lim_{s \rightarrow \infty} Z(s) \in \Cinfty(\bbS^1,\R^{2n})$ (in the $\Cinfty$-topology) exists and that
\begin{align*}
    Z(\infty) \text{ is a constant loop with } z'' \circ Z(\infty) \equiv 0 \,\, .
\end{align*}
Then 
\begin{equation}
\label{eq: exponential decay - Lhat^2 estimate Qsfhat Zhat}
    \lVert \Qsfhat \, \Zhat (s) \rVert_{\Lhat^2}^2 \leq  \lVert \Qsfhat \, \mathsf{E}(Z(s_0)) \rVert_{\Lhat^2}^2 \, e^{-2B(s-s_0)} \qquad \Forall s \geq s_0 \,\, .
\end{equation}
\end{enumerate}
\end{claim}

\begin{proofClaim}
    Part \ref{it: exponential decay estimate - Zhat = E(Z)} follows directly from \eqref{eq: abstract Floer eq - higher derivatives} and \eqref{eq: exponetial decay - partial_s A_z}. Now we show part \ref{it: exponential decay estimate - maximum principle Zhat}.
    Define the smooth function $f : [s_0,\infty) \rightarrow [0,\infty)$ by
 \[
 f(s) := \frac{1}{2} \, \lVert \Qsfhat \, \Zhat (s) \rVert_{\Lhat}^2 = \frac{1}{2} \, \lVert \Qsfhat \, \mathsf{E} (Z(s)) \rVert_{\Lhat}^2 \,\, .
 \]
    We have seen in \eqref{eq: exponential decay - H^2-estimate f'' equality} that
    \begin{align*}
    f''(s) &= \lVert  \Qsfhat \, \Asfhat_Z \, \Qsfhat \, \Zhat \rVert^2_{\Lhat^2} + \big\langle \Qsfhat \, \Zhat , \, (\partial_s \Asfhat_Z) \, \Qsfhat \, \Zhat \big\rangle_{\Lhat^2} \notag\\
    &\hspace*{3mm}+ \big\langle (\Asfhat_Z^* - \Asfhat_Z) \, \Qsfhat \, \Zhat , \, \Asfhat_Z \, \Qsfhat \, \Zhat \big\rangle_{\Lhat^2} + \lVert \Asfhat_Z \, \Qsfhat \, \Zhat \rVert_{\Lhat^2}^2 \,\, .
\end{align*}
    Completely analogous to the proof of Proposition \ref{prop: exponential decay - L^2 estimate}, we can estimate the first and fourth summand from below by \eqref{eq: exponential decay - H^2-estimate invertibility} and the second and third summand from above by \eqref{eq: exponential decay - H^2-estimate derivative Ahat_Z Zhat} and \eqref{eq: exponential decay - H^2-estimate adjoint}, yielding
    \[
    f''(s) \geq (2 c- 4 \eps) \, \lVert \Qsfhat \, \Zhat(s) \rVert_{\Hhat^1}^2 \geq (2c- 4 \eps) \, \lVert \Qsfhat \, \Zhat (s)\rVert_{\Lhat^2}^2 = 4 B^2 \, f(s) \,\, .
    \]
    The asymptotic condition on $Z$ and \eqref{eq: Qhat E(x) = 0 for x constant z''(x)=0} now imply 
    \[
    \lim_{s \rightarrow \infty} \Qsfhat \, \mathsf{E} (Z(s)) = \Qsfhat \, \mathsf{E} (Z(\infty)) = 0
    \]
    and therefore $f(s) \overset{s \rightarrow \infty}{\longrightarrow} 0 \,$. The maximum principle (Lemma \ref{lem: maximum principle}) now gives the claimed inequality \eqref{eq: exponential decay - Lhat^2 estimate Qsfhat Zhat}.
\end{proofClaim}

\textbf{Part II}: 
Define
\[
\Xi_1 : \Ucalxzerobar \rightarrow [0,\infty) \comma \quad  \Xi_1(x) := \lVert \Qsfhat \, \mathsf{E}(x) \rVert_{\Lhat^2} \,\, .
\]
Note that $\Xi_1$ is continuous in the $\Cinfty$-topology and that it vanishes on constant loops $x$ with $z''(x) \equiv 0$ by \eqref{eq: Qhat E(x) = 0 for x constant z''(x)=0}. For a smooth cylinder $Z : [s_0,\infty) \rightarrow \Cinfty(\bbS^1,\R^{2n})$ 
satisfying assumptions \ref{it: exponential decay - properties Z image in nbhd}\textbf{--}\ref{it: exponential decay - properties Z limit} in Proposition \ref{prop: exponential decay - localized version}, writing out \eqref{eq: exponential decay - Lhat^2 estimate Qsfhat Zhat} we have
\begin{align}
\label{eq: exponential decay - Lhat^2 estimate Qsfhat Zhat rewritten}
\lVert \Qsf \, Z (s) \rVert_{L^2}^2 + \lVert \partial_s \, \Qsf  Z (s) \rVert_{L^2}^2 + \lVert \partial_s^2 \, \Qsf  Z (s) \rVert_{L^2}^2 \leq \Xi_1(Z(s_0))^2 \, e^{-2B(s-s_0)} \quad \Forall s \geq s_0 \,\, .
\end{align}
 In the second part we now transform this estimate into an estimate for the $H^2$-norm of $Z(s) $ by repeatedly invoking the Floer equation.\footnote{
How this can be accomplished was again explained to me by Richard Siefring.
}
That is, for a given cylinder $Z$ as above, we must find exponential decay estimates for the $L^2$-norm of
\begin{equation}
\label{eq: partial_t Z = partial_t Q Z}
\partial_t ( \Qsf \, Z) = \partial_t Z \quad \text{and} \quad \partial_t^2 (\Qsf \, Z ) =  \partial_t^2 Z \,\, ,
\end{equation}
where we have made use of \ref{it: operators A_x - partial_t Q = partial_t}. For convenience, we work with the shift $u := \shift{+}(Z) \,$, so that $Z = \Bar{u} = \shift{-}(u) \,$. Of course, more precisely we should write $u(s,t) = \shift{+}(Z(s))(t) \,$. Notice that $u$ is a solution of the classical Floer equation \eqref{eq: Floer eq - Floer illustration} since $Z$ is a solution of the abstract Floer equation $\partial_s Z = \Asf_Z Z \,$.
\\

\noindent We have $\Psf \, \Asfxzerobar = 0$ by \ref{it: operators A_x - Q Asfxzerobar = Asfxzerobar} and moreover $z'' ( \Qsf  X) = z''(X)$ for every $X \in L^2 $ because $z'' (\Psf  X) \equiv 0$ as $\Psf  X \in \ker \Asfxzerobar \,$. Using this and the definition of the operators $\Asf_x$ in \eqref{eq: operator A_x} and \eqref{eq: Asf_xzerobar}, we compute
\begin{align}
    \partial_s \, \Psf  Z &= \Psf \, \partial_s Z = \Psf \, \Asf_Z \, Z = \Psf (\Asf_Z - \Asfxzerobar) \, Z \notag\\
    &= \Psf \Big( (J_0(t) - J_t(u) ) \, \partial_t \Bar{u} + (J_0(t) \, S_0(t) - J_t(u) \, S(u)) \cdot z''(\Qsf \, Z) \Big) \,\, . \label{eq: partial_s P Z}
\end{align}
Rearranging \eqref{eq: Floer eq - shift(u) hybrid version} yields
\begin{align}
    \partial_t \Bar{u} &= J_t(u) \, \partial_s \Bar{u} - S(u) \, z''(\Bar{u}) \notag\\
    &= J_t(\shift{+}(\Bar{u})) \, (\partial_s (\Qsf  Z) + \partial_s (\Psf Z)) - S(\shift{+}(\Bar{u})) \, z''(\Qsf Z) \,\, . \label{eq: partial_t Bar(u)}
\end{align}
Now insert \eqref{eq: partial_t Bar(u)} into \eqref{eq: partial_s P Z} and rearrange so that the terms $\Qsf Z \comma \partial_s \Qsf Z$ and $\partial_s \Psf Z$ are collected. This can be written as
\begin{align}
\label{eq: partial_s P Z = C^1_Z(QZ) + C^2_Z(partial_s QZ)}
\partial_s \, \Psf Z - \Psf\Big( (J_0(t) - J_t(\shift{+}(\Bar{u}))  ) \cdot J_t(\shift{+}(\Bar{u})) \cdot \partial_s (\Psf Z) \Big)  = \mathsf{C}^1_Z (\Qsf Z) + \mathsf{C}^2_Z (\partial_s \, \Qsf Z) \,\, ,
\end{align}
where the argument of $\Psf$ on the left-hand side is a loop, so $t$ is not fixed therein, and where $\mathsf{C}^1_x \comma \mathsf{C}^2_x \in \mathcal{L}(L^2,L^2)$ are loop-dependent operators such that the maps $\mathsf{C}^j: \, \Ucalxzerobar \rightarrow \mathcal{L}(L^2,L^2) \comma x \mapsto \mathsf{C}^j_x$ are continuous for $j=1,2 \,$. The precise formulae for $\mathsf{C}^1_x$ and $\mathsf{C}^2_x$ will be irrelevant subsequently but, to allow the reader to double-check, we nonetheless state them in \eqref{app eq: operator C^1_x} and \eqref{app eq: operator C^2_x} in the appendix.
By continuity, choosing $\Ucalxzerobar$ smaller if necessary, we may assume that the operator norms of $\mathsf{C}^1_x$ and $\mathsf{C}^2_x$ are bounded by some constant $c_1 > 0$ uniformly for all $x \in \Ucalxzerobar \,$.  
We can express the left-hand side of \eqref{eq: partial_s P Z = C^1_Z(QZ) + C^2_Z(partial_s QZ)} in terms of a loop-depending operator as well: For $x \in \Ucalxzerobar$ define $\mathsf{D}_x \in \mathcal{L}(L^2,L^2)$ by
\begin{align}
\label{eq: operator D_x}
  \mathsf{D}_x (X) := X - \mathsf{P}\Big( (J_0(t) - J_t(\shift{+}(x))  ) \, J_t(\shift{+}(x)) \cdot X \Big) \qquad \Forall X \in L^2 \,\, .
\end{align}
Then the left-hand side of \eqref{eq: partial_s P Z = C^1_Z(QZ) + C^2_Z(partial_s QZ)} is just $\mathsf{D}_{Z(s)} (\partial_s (\Psf  Z)(s)) \,$.
 Clearly 
 \[
 \mathsf{D}: \, \Ucalxzerobar \rightarrow \mathcal{L}(L^2,L^2) \comma \quad x \mapsto \mathsf{D}_x \,\, ,
 \]
 is continuous and moreover its evaluation at $\xzerobar$ is the identity, $\mathsf{D}_{\xzerobar} = \mathrm{id}_{L^2} \,$, because
\[
J_0(t) - J_t(\shift{+}(\xzerobar)) = J_0(t) - J_0(t) = 0 \,\, .
\]
It follows that, for any fixed choice of constant $c_2 > 1 \,$, we can assume, after making $\Ucalxzerobar$ smaller if necessary, that for all $x \in \Ucalxzerobar$ the operator $\mathsf{D}_x$ is invertible with
\begin{equation}
\label{eq: operator D_x invertibility}
\lVert X \rVert_{L^2} \leq c_2 \, \lVert \mathsf{D}_x  X \rVert_{L^2}  \qquad \Forall X \in L^2 \,\, .
\end{equation}
Since $Z$ has image in $\Ucalxzerobar \,$, for every $s \geq s_0$ (suppressing $s$ in the notation below)
\begin{align}
     \lVert \partial_s  \,  \Psf Z  \rVert_{L^2} &\leq c_2 \, \lVert \mathsf{D}_Z \, \partial_s \, \Psf Z \rVert_{L^2} \overset{\text{\eqref{eq: partial_s P Z = C^1_Z(QZ) + C^2_Z(partial_s QZ)}}}{=} c_2 \, \lVert \mathsf{C}_Z^1(\Qsf Z ) + \mathsf{C}^2_Z (\partial_s \, \Qsf Z ) \rVert_{L^2}  \notag\\
    &\leq c_2 \, \big(\lVert \mathsf{C}_Z^1(\Qsf Z) \rVert_{L^2} + \lVert \mathsf{C}_Z^2(\partial_s \, \Qsf Z) \rVert_{L^2} \big) \notag\\
    &\leq  c_1  c_2 \, \big( \lVert \Qsf Z \rVert_{L^2} + \lVert \partial_s \, \Qsf Z \rVert_{L^2} \big) \notag\\
    &\overset{\text{\eqref{eq: exponential decay - Lhat^2 estimate Qsfhat Zhat rewritten}}}{\leq} 2  c_1  c_2 \, \Xi_1(Z(s_0)) \, e^{- B(s-s_0)}  \,\, .\label{eq: partial_s P Z exponential decay}
\end{align}
We now have exponential decay estimates for $\Qsf Z$ and $\partial_s \, \Qsf Z$ by \eqref{eq: exponential decay - Lhat^2 estimate Qsfhat Zhat rewritten} and for $\partial_s \, \Psf Z$ by \eqref{eq: partial_s P Z exponential decay}.
Returning to \eqref{eq: partial_t Bar(u)}, this gives an exponential decay estimate for $\partial_t \Bar{u}  \,$. In detail, we can bound the Euclidean operator norms of the matrices $J_t(\shift{+}(x)(t))$ and $S(\shift{+}(x)(t))$ uniformly for every $t \in \bbS^1$ and every loop $x \in \Ucalxzerobar$ sufficiently close to the constant loop at the origin $\xzerobar \,$. Using that $\lVert \,\cdot \,\rVert_{L^2}$ is equivalent to the Euclidean $L^2$-norm, we estimate
\begin{align}
\label{eq: partial_t Bar(u) exponential decay}  
\lVert \partial_t \Bar{u} \rVert_{L^2} \overset{\text{\eqref{eq: partial_t Bar(u)}}}{\leq} c_3 \, \big( \lVert \Qsf Z \rVert_{L^2} + \lVert \partial_s \, \Qsf Z \rVert_{L^2} + \lVert \Psf Z \rVert_{L^2} \big) \leq c_4 \, \Xi_1(Z(s_0)) \, e^{-B(s-s_0)} \,\, ,
\end{align}
with constant $c_3 >0$ just depending on the equivalence of $L^2$-norms and the uniform bound on the operator norm of the matrices $ J_t(\shift{+}(x))$ and $S(\shift{+}(x)) $ and with $c_4 = c_4(c_1,c_2,c_3) >0$ a combinatorial constant.
\\

\noindent We emphasize that the constants $c_1,\ldots , c_4 >0$ and the size of the neighborhood $\Ucalxzerobar$ have been chosen independently of the cylinder $Z$. We could have defined the maps $\mathsf{C}^1 \comma \mathsf{C}^2$ and $\mathsf{D}$ before fixing the cylinder $Z$ and make the appropriate choices. (For the sake of exposition we decided
against this.) So \eqref{eq: partial_t Bar(u) exponential decay} is an exponential decay estimate for $\partial_t Z = \partial_t \Bar{u} $ of the desired form.
\\

\noindent The exponential decay estimate for $\partial_t^2 Z = \partial_t^2 \Bar{u}$ now works completely analogously. Derive \eqref{eq: partial_t Bar(u)} in $s$-direction to obtain
\begin{align}
\partial_s \partial_t \Bar{u}  &=  J_t(u) \, \big( \partial_s^2(\Qsf Z) + \partial_s^2(\Psf Z) \big) - S(u) \, z''(\partial_s \, \Qsf Z) \notag\\
&\hspace*{3mm}+ (\partial_s J_t(u)) \, \big( \partial_s (\Qsf Z) + \partial_s (\Psf Z) \big) - (\partial_s S(u)) \, z''(\Qsf Z) \,\, .  \label{eq: partial_s partial_t Bar(u)}
\end{align}
As before, we can use the abstract Floer equation to replace the $s$-derivatives $\partial_s (J_t(u))$ and $\partial_s (S(u))$ by an expression depending only on the loop $Z(s)\,$. We carry this out exemplarily for $\partial_s (S(u))$ in Subsection \ref{app subsec: Eliminate s-derivatives via Floer eq} in the appendix.
In particular, after choosing $\Ucalxzerobar$ smaller if necessary and using that $Z(s) \in \Ucalxzerobar \,$, we can assume that the matrix norms of $\partial_s (J_t(u))$ and $\partial_s(S(u))$ are uniformly bounded by a constant independent of the cylinder $Z$. 
\\
Similarly, deriving \eqref{eq: partial_s P Z} in $s$-direction and applying the Floer equation to replace $s$-derivatives of $J_t(u)$ respectively $S(u) \,$, will give on the right-hand side an expression in which every summand contains one of $\Qsf Z \comma \partial_s \Qsf Z \comma \partial_t \Bar{u}$ or $\partial_s \partial_t \Bar{u} $ and otherwise depends just on the loop $Z(s)$. Substituting $\partial_s \partial_t \Bar{u}$ in this formula with the right-hand side in \eqref{eq: partial_s partial_t Bar(u)} and collecting the terms $\partial_s^2 \Psf Z \,$, we can write
\begin{align*}
    \mathsf{D}_Z (\partial_s^2 \Psf Z) &= \partial_s^2 \Psf Z - \Psf\Big( (J_0(t) - J_t(u)  ) \, J_t(u) \cdot (\partial_s^2 \Psf Z) \Big) \\
    &= \mathsf{C}^3_Z(\Qsf Z) + \mathsf{C}^4_Z(\partial_s \, \Qsf Z) +  \mathsf{C}^5_Z(\partial_s^2 \, \Qsf Z) + \mathsf{C}^6_Z (\partial_s \, \Psf Z) + \mathsf{C}^7_Z (\partial_t \Bar{u})
\end{align*}
with certain loop-dependent operators $\mathsf{C}^3_x , \ldots , \mathsf{C}^7_x \in \mathcal{L}(L^2,L^2) \,$, so that the maps $\mathsf{C}^j : \Ucalxzerobar \rightarrow \mathcal{L}(L^2,L^2) \comma x \mapsto \mathsf{C}^j_x \,$, are continuous respectively. The operator $\mathsf{D}_x$ is the one defined in \eqref{eq: operator D_x} above.
Keeping \eqref{eq: operator D_x invertibility} in mind and the already established exponential decay estimates in \eqref{eq: exponential decay - Lhat^2 estimate Qsfhat Zhat rewritten}, \eqref{eq: partial_s P Z exponential decay} and \eqref{eq: partial_t Bar(u) exponential decay}, we can proceed similar to \eqref{eq: partial_s P Z exponential decay} and estimate
\begin{align*}
    \lVert \partial_s^2 \, \Psf Z \rVert_{L^2} &\leq c_5 \, \big( \lVert  \Qsf Z \rVert_{L^2} + \lVert  \partial_s \, \Qsf Z \rVert_{L^2} +  \lVert  \partial_s^2 \, \Qsf Z \rVert_{L^2} + \lVert  \partial_s \, \Psf Z \rVert_{L^2} + \lVert  \partial_t \Bar{u} \rVert_{L^2} \big) \\
    &\leq c_6 \, \Xi_1(Z(s_0)) \, e^{-B(s-s_0)} \,\, ,
\end{align*}
with constants $c_5, c_6 >0$ independent of the cylinder $Z$. We therefore have an exponential decay estimate for $\partial_s^2 \, \Psf Z \,$. In view of \eqref{eq: partial_s partial_t Bar(u)} and because the matrix norms of $\partial_s(J_t(u))$ and $\partial_s(S(u))$ have been bounded by a constant independent of the cylinder $Z$, we now also have an exponential decay estimate
\[
\lVert \partial_s \partial_t \Bar{u} \rVert_{L^2} \leq c_7 \, \Xi_1(Z(s_0)) \, e^{-B(s-s_0)} \qquad \Forall s \geq s_0 \,\, .
\]
Finally, we derive \eqref{eq: partial_t Bar(u)} in $t$-direction
\begin{align*}
    \partial_t^2 \Bar{u} &= J_t(u) \, \partial_s \partial_t \Bar{u} - S(u) \, z''(\partial_t  \Bar{u}) + \partial_t ( J_t(u)) \, \partial_s \Bar{u} - \partial_t ( S(u)) \, z''(\Bar{u}) \\
    &= J_t(\shift{+}(\Bar{u})) \, \partial_s \partial_t \Bar{u} - S(\shift{+}(\Bar{u})) \, z''(\partial_t \Bar{u}) \\
    &\hspace*{3mm}+  \partial_t ( J_t \circ \shift{+}(\Bar{u})) \,\, \big(\partial_s(\Qsf Z) + \partial_s(\Psf Z) \big) - \partial_t ( S \circ \shift{+}(\Bar{u})) \,\, z''(\Qsf Z) \,\, .
\end{align*}
The matrix norms of $J_t ( \shift{+}(x))  $ and $  S ( \shift{+}(x)) $ and its $t$-derivatives can be uniformly bounded for every $t \in \bbS^1$ and every loop $x \in \Ucalxzerobar \,$, by making $\Ucalxzerobar$ smaller if necessary. Every of the above summands contains one of 
\[
\partial_s \partial_t \Bar{u} \comma \,\, \partial_t \Bar{u} \comma \,\, \Qsf Z \comma \,\, \partial_s(\Qsf Z) \comma \,\, \partial_s(\Psf Z) 
\]
and, up to a constant independent of $Z$, we can bound the $L^2$-norm of each of them by $\Xi_1(Z(s_0)) \, e^{-B(s-s_0)} \,$. Hence 
\begin{equation}
\label{eq: partial_t^2 Bar(u) exponential decay}
\lVert \partial_t^2 \Bar{u} \rVert_{L^2} \leq  c_8 \, \Xi_1(Z(s_0)) \, e^{-B(s-s_0)} \,\, ,
\end{equation}
which is the desired exponential decay estimate for $\partial_t^2 Z = \partial_t^2 \Bar{u} \,$. All in all, combining \eqref{eq: exponential decay - Lhat^2 estimate Qsfhat Zhat rewritten}, \eqref{eq: partial_t Bar(u) exponential decay}, \eqref{eq: partial_t^2 Bar(u) exponential decay} and \eqref{eq: partial_t Z = partial_t Q Z}, for a suitable combinatorial constant $c_9 = c_9(c_1,\ldots , c_8)>0$ independent of the cylinder $Z$, we have
\begin{align*}
    \lVert \Qsf \, Z(s) \rVert_{H^2} \leq c_9 \, \Xi_1(Z(s_0)) \, e^{-B(s-s_0)} \qquad \Forall s \geq s_0 \,\, .
\end{align*}
Thus the continuous function $\Xi := c_9 \, \Xi_1 : \Ucalxzerobar \rightarrow [0,\infty) \,$, which vanishes at constant loops $x$ with $z''(x) \equiv 0$ since $\Xi_1$ does, has the desired property claimed in Proposition \ref{prop: exponential decay - localized version}. This finishes the second part and the proof of the proposition.
\end{proof}

\noindent We have successfully proven Proposition \ref{prop: exponential decay - localized version} and therefore also Theorem \ref{mthm: Floer - exponential decay}.

\begin{remark}
\label{rmk: diagonal operators larger}
\begin{enumerate}[label=(\roman*)]
    \item Considering the larger vector $(Z,\partial_s Z , \ldots , \partial_s^k Z) \,$, one analogously gets an exponential decay estimate for the $H^k$-norm of $\Qsf \, Z$ and thus, as in the proof of Theorem \ref{mthm: Floer - exponential decay}, also exponential decay estimates for $\partial_s u , \ldots , \, \partial_s \partial_t^{k-2} u \,$.
    Using the Floer equation, from this one obtains for every $k,\ell \geq 0$ an exponential decay estimate for $\partial_s^{k + 1} \partial_t^{\ell} u \,$. Here the higher-order partial derivatives of $u$ are understood in the coordinates of Lemma \ref{lem: tubular nbhd around limit loop}.
    \item The exponential coefficient $B > 0$ only depends on the operator $\Asfxzerobar \,$.
\end{enumerate}
\end{remark}

\appendix

\section{Supplements to the Proof of Theorem \ref{mthm: Floer - exponential decay}}
\label{app sec: supplements proof exponential decay}

\subsection{$\Asfxzerobar$ is the covariant Hessian}
\label{app subsec: Asfxzerobar = covariant Hessian}

We show what we claimed in \ref{it: operators A_x - Asfxzerobar = Hessian}, namely that the operator $\Asfxzerobar$ is the covariant Hessian of $\actionH$ at the loop $x_0 \,$. A similar argument (for the Rabinowitz action functional) appears in \cite[Lemma 30]{FauckThesis}.
First we recall Hadamard's lemma.

\begin{lemma}[Hadamard's lemma]
\label{app lem: Hadamard's lemma}
Let $f : \R^m \rightarrow \R^k$ be a smooth function. For $i=1,\ldots , m$ define $g_i : \R^m \rightarrow \R^k$ by
\[
g_i(y) := \int_0^1 \frac{\partial f}{\partial y_i} (r y) \, dr \,\, .
\]
Then $f(y) = f(0) + \sum_{i=1}^m y_i \, g_i(y) \,$.
\end{lemma}
\begin{remark}
\label{app rmk: Hadamard's lemma}
\begin{enumerate}[label=(\roman*)]
\item\label{app it: Hadamard's lemma derivative} Observe that the differential at the origin is $d_0 f \, [v] = \sum_{i=1}^m v_i \, g_i(0) \,$.
\item\label{app it: Hadamard's lemma parametrized} Let $M$ be a smooth manifold and $f : M \times \R^m \rightarrow \R^k$ be a smooth function with $f(p,0)=0$ for every $p \in M \,$. Applying Hadamard's lemma fiberwise, we can write $f(p,y) = \sum_{i=1}^m y_i \, g_i(p,y)$ with smooth functions $g_i : M \times \R^m \rightarrow \R^k \,$. The differential of $f$ at $(p,0)$ is $
d_{(p,0)} f \, [u,v] = \sum_{i=1}^m v_i \, g_i(p,0) \,$.
\end{enumerate}
\end{remark}

\begin{lemma}
\label{app lem: Asf_xzerobar = covariant Hessian}
In the fixed coordinates $(\vartheta,z)$, the covariant Hessian at the loop $x_0$
\[
\nabla^2 \actionH(x_0) : \,\, H^1(\bbS^1,x_0^*TW) \rightarrow L^2(\bbS^1,x_0^*TW) 
\]
equals the operator $\Asfxzerobar$ in \eqref{eq: Asf_xzerobar}.
\end{lemma}

\begin{proof}
The $g_J$-gradient $ \nabla^{g_J} \actionH$ is a section of the Banach bundle
\begin{equation*}
\bigsqcup_{x \in H^1(\bbS^1,W)} L^2(\bbS^1,x^*TW) \longrightarrow H^1(\bbS^1,W) \,\, .
\end{equation*}
The coordinates $(\vartheta,z) : U \overset{\sim}{\rightarrow} \bbS^1 \times \R^{2n-1} $ induce a local trivialization of this bundle over $H^1(\bbS^1,U) \,$.
The covariant derivative $\nabla^2 \actionH (x_0)$ is equal to the vertical differential of $\nabla^{g_J} \actionH$ at its zero $x_0 \,$, which, in the trivialization determined by $(\vartheta,z) \,$, is just the ordinary derivative $D_{x_0} \nabla^{g_J} \actionH$ at $x_0$ of the principal part of $\nabla^{g_J} \actionH $ in this trivialization. We show that $D_{x_0} \nabla^{g_J} \actionH$ is equal to $\Asfxzerobar \in \mathcal{L}(H^1,L^2) \,$.
To this end, let $(-\eps,\eps) \rightarrow H^1(\bbS^1,U) \comma a \mapsto x_a \,$, be an arbitrary path of loops through $x_0$ with derivative $X := \frac{\partial x_a}{\partial a}(0,\,\cdot\,)$ at $a=0 \,$. Then
\begin{align*}
    D_{x_0} \nabla^{g_J} \actionH \, [X] &= \frac{d}{da}\Big|_{a=0} \,\nabla^{g_J} \actionH(x_a) =  \frac{d}{da}\Big|_{a=0} \Big( - J_t(x_a) \, (\partial_t x_a - X_H(x_a)) \Big) \\
    &= - J_t(x_0) \, (\partial_t X - d_{x_0} X_H \, [X] ) \,\, ,
\end{align*}
where $d X_H$ denotes the differential of the vector-valued function $X_H : U \rightarrow \R^{2n} \,$. For the last equality we used that $\partial_t x_0 - X_H(x_0) = 0 \,$, so that no derivative of $J$ appears. Now recall that, since $\partial_{\vartheta} - X_H$ vanishes on $\{z''=0\} \,$, we could apply Hadamard's lemma in the $z''$-coordinate and write
\[
\partial_\vartheta - X_H(\vartheta,z',z'') = S(\vartheta,z,z'') \cdot z'' \,\, .
\]
By Remark \ref{app rmk: Hadamard's lemma} \ref{app it: Hadamard's lemma parametrized}, the differential of $\partial_\vartheta - X_H$ at the point $x_0(t)$ is 
\begin{equation}
\label{app eq: d X_H = S_0 z''}
-d_{x_0(t)} X_H = S(x_0(t)) \cdot z'' = S_0 (t) \cdot z'' \,\, ,
\end{equation}
where $z'' : \R^{2n} \rightarrow \R^{2n-d}$ designates the projection to the last components. Hence
\[
D_{x_0} \nabla^{g_J} \actionH \, [X] = - J_0(t) \, (\partial_t X + S_0(t) \cdot z''(X)) = \Asfxzerobar \, X \,\, .
\]
\end{proof}

\begin{lemma}
\label{app lem: kernel Asf_xzerobar}
Under assumption \ref{it: MB}, the kernel of $\Asfxzerobar$ consists precisely of the constant loops $X$ with $z''(X) \equiv 0 \,$.
\end{lemma}

\begin{proof}
That the constant loops $X$ with $z''(X) \equiv 0$ lie in the kernel is clear. Conversely, suppose $X$ is in the kernel, so
\[
0 = - J_0(t) \, (\partial_t X + S_0(t) \cdot z''(X)) \,\, .
\]
This implies, using that $J_0(t)$ is invertible and \eqref{app eq: d X_H = S_0 z''}, that $\partial_t X (t) = d_{x_0 (t)} X_H \cdot X(t) \,$. In other words, $X$ is a solution of the linearized differential equation along $x_0 \,$, so that
\begin{equation}
\label{app eq: linearized solution X for ker Asf_xzerobar}
X(t) = d_{x_0(0)} \phi_{X_H}^{t} \cdot X(0)  \,\, .
\end{equation}
From this and assumption \ref{it: MB}, we infer $X(0)=X(1) \in T_{x_0(0)} N \,$. Now $X$ is constant since $d_{x_0(0)} \phi_{X_H}^t|_{TN} = \mathrm{id}$ and by \eqref{app eq: linearized solution X for ker Asf_xzerobar} again.
\end{proof}

\subsection{Adjoints depends continuously on base loop}
\label{app subsec: adjoints continuous}

In this subsection we show that the adjoints $\Asf_{x}^*$ and $\Asfhat^{*}_x$ depend continuously on the loop $x$. This will follow from the more general Proposition \ref{app prop: adjoint continuous} below.
\\

\noindent Usually, the domain of the adjoint of an unbounded operator is defined to be the maximal possible domain. We deviate from this and for this reason specify what we mean by the adjoint.

\begin{definition}[Adjoint]
\label{app def: adjoint}
Given a Hilbert space $(H, \langle\,\cdot\comma \cdot\,\rangle_H)$ and a dense linear subspace $D \subseteq H \,$. Let $A : D \rightarrow H$ be an unbounded linear operator. There exists at most one unbounded linear operator $A^* : D \rightarrow H$ satisfying
\[
\langle A x , y \rangle_H  = \langle x , A^* y \rangle_H \qquad \Forall x,y \in D \,\, .
\]
If existent, this operator $A^*$ is called the \textit{$\langle \,\cdot\comma \cdot\,\rangle_H$-adjoint of $A$}.
\end{definition}

\begin{remark}
\label{app rmk: adjoint} 
\begin{enumerate}[label=(\roman*)]
    \item\label{app it: adjoint criterion} $A : D \rightarrow H$ has an adjoint if and only if for every $y \in D$ there exists a constant $c_y \geq 0$ with $|\langle Ax,y \rangle_H| \leq c_y \, |x|_H$ for all $x \in D \,$.
    \item\label{app it: adjoint isometry} Suppose $A : D \rightarrow H$ has an adjoint $A^* \,$. Let $\phi : H \rightarrow H'$ be an isometry between the Hilbert spaces $H$ and $H'$ and set $D' := \phi(D) \subseteq H' \,$. Then $\phi \, A \, \phi^{-1} : D' \rightarrow H'$ has the $H'$-adjoint $(\phi \, A \, \phi^{-1})^* = \phi \, A^* \, \phi^{-1} \,$. 
\end{enumerate}

\end{remark}

\noindent Consider the trivial bundle $E := \bbS^1 \times \R^m \rightarrow \bbS^1$ over the circle. Every first-order linear partial differential operator (short: PDO) $A : \Cinfty(E) \rightarrow \Cinfty(E)$ can be written uniquely as
\begin{equation}
\label{app eq: PDO on trivial bundle}
A X = R(t) \cdot \partial_t X + S(t) \cdot X \qquad \Forall X \in \Cinfty(\bbS^1,\R^m) \,\, ,
\end{equation}
where $R , S\in \Cinfty(\bbS^1,\R^{m\times m})$ are matrix-valued functions. Now let $g= (g_t)_{t \in \bbS^1}$ be an arbitrary smooth bundle metric on $E$. It defines an $L^2$-inner product on the $L^2$-sections of $E$ as usual via
\[
\langle X_1, X_2 \rangle_{L^2}^{g} := \int_0^1 g_t(X_1(t), X_2(t)) \, dt \,\, .
\]
We abbreviate $H^1 = H^1(\bbS^1,\R^m)$ and $L^2 = L^2(\bbS^1,\R^m) \,$. Because the bundle base $\bbS^1$ is compact, the norm induced by the above $L^2$-inner product is equivalent to the standard (Euclidean) $L^2$-norm. The PDO $A$ induces a map $A \in \mathcal{L}(H^1,L^2) \,$, which we also consider as unbounded linear operator on $L^2 $ with dense domain $H^1 \subseteq L^2 \,$.
We denote by $A^*$ its $\langle \,\cdot\comma \cdot\,\rangle_{L^2}^g$-adjoint. This exists and coincides with the formal adjoint of a PDO, described in \cite[\S10.1.3]{Nicolaescu_Geometry_of_Manifolds}. So $A^*$ is itself a first-order PDO and therefore $A^* \in \mathcal{L}(H^1,L^2) \,$. 

\begin{proposition}
\label{app prop: adjoint continuous}
For every $R, S \in \Cinfty(\bbS^1,\R^{m \times m})$ define the PDO $A_{(R,S)} : \Cinfty (E) \rightarrow \Cinfty(E)$ by \eqref{app eq: PDO on trivial bundle}. Then the map
\[
\Cinfty(\bbS^1,\R^{m \times m}) \times \Cinfty(\bbS^1,\R^{m \times m}) \rightarrow \mathcal{L}(H^1,L^2) \comma \quad (R,S) \mapsto (A_{(R,S)})^*
\]
is continuous, where each factor in the domain carries the $\Cinfty$-topology.
\end{proposition}

\begin{remark}
\label{app rmk: adjoint continuous}
As the proof will show, the above map is even continuous if we endow the domain with the $(C^1 \times C^0)$-topology
\end{remark}

\begin{proof}[Proof of Proposition \ref{app prop: adjoint continuous}] The proof has two steps.

\textit{Step 1:} We first consider the special case in which $g$ is the Euclidean bundle metric, that is $g_t = \langle \,\cdot\comma \cdot\,\rangle_{\mathrm{euc}}$ for every $t \in \bbS^1 \,$. Integration by parts immediately shows that, in this case, the adjoint is given by
\[
(A_{(R,S)})^* X = - R^{\mathrm{T}} \cdot \partial_t X + ( S^{\mathrm{T}} - \partial_t R^{\mathrm{T}}) \cdot X \,\, ,
\]
where $R(t)^{\mathrm{T}} , S(t)^{\mathrm{T}}$ are the matrix transposes respectively.
This shows that, if $(R,S)$ varies in the $(C^1 \times C^0)$-topology, the coefficient functions of $(A_{(R,S)})^*$ vary continuously in the $(C^1 \times C^0)$-topology, so that $(A_{(R,S)})^*$ varies continuously in $\mathcal{L}(H^1,L^2) \,$.

\textit{Step 2:} Now we deal with the general case. Choose a global $g$-orthonormal frame $e_1,\ldots , e_m \in \Cinfty(\bbS^1,\R^{m})$ of $E=\bbS^1 \times \R^{m} \,$, which exists by Gram-Schmidt. The frame induces a global orthogonal trivialization $\Phi : (E, g) \overset{\sim}{\rightarrow} (E, \langle \,\cdot\comma\cdot\,\rangle_{\mathrm{euc}}) \,$. The pushforward via $\Phi$ gives an isometry of Hilbert spaces
\[
\Phi_* : \,\, (L^2, \, \langle\,\cdot\comma \cdot\, \rangle_{L^2}^g ) \overset{\sim}{\longrightarrow} (L^2, \, \langle\,\cdot\comma\cdot\,\rangle_{L^2}^{\mathrm{euc}}) \comma  \quad (\Phi_*(X)) (t) := \Phi_t \cdot X(t) \,\, ,
\]
which preserves regularity, $\Phi_*(H^1) = H^1 $ and $\Phi_*|_{H^1} \in \mathcal{L}(H^1,H^1) \,$.
Using Step 1 and Remark \ref{app rmk: adjoint} \ref{app it: adjoint isometry}, it therefore suffices to show that the coefficient functions of the PDO $\Phi_* A_{(R,S)} \, (\Phi^{-1})_*$ vary continuously in the $(C^1\times C^0)$-topology if $(R,S)$ varies in the $(C^1 \times C^0)$-topology. And this is indeed the case, as a straightforward calculation shows that $\Phi_* \, A_{(R,S)} \, (\Phi^{-1})_* = A_{(\widetilde{R},\widetilde{S})} \,$, where
\begin{align*}
    \widetilde{R}(t) := \Phi_t \cdot R(t) \cdot \Phi_t^{-1}
     \quad \text{and} \quad 
    \widetilde{S}(t) :=  \Phi_t \cdot S(t) \cdot \Phi_t^{-1} + \Phi_t \cdot R(t) \cdot \partial_t (\Phi_t^{-1})(t)  \,\, .
\end{align*}
\end{proof}

\subsection{Equality $\partial_s \mathsf{B}_Z = \mathsf{C}_Z$}
\label{app subsec: operator C_x}

We show that $\partial_s \mathsf{B}_{Z}(s) = \mathsf{C}_{Z(s)}$ for every $s \,$, where $\mathsf{B}_x$ and $\mathsf{C}_x$ were defined in \eqref{eq: operator B_x} and \eqref{eq: operator C_x} respectively. We will only use the abstract Floer equation $\partial_s Z = \Asf_Z \, Z $ and the usual rules for differentiation.
\begin{align*}
    \partial_s \mathsf{B}_Z = \partial_s \Big(\underbrace{ D^2_Z \Asf \, [\Asf_Z Z , \, \Asf_Z Z]}_{=: \mathrm{I}(s)}  \Big) + \partial_s \Big( \underbrace{ D_Z \Asf \, [D_Z \Asf \, [\Asf_Z \, Z] \, Z] }_{=: \mathrm{II}(s)} \Big)+ \partial_s \Big( \underbrace{D_Z \Asf \, [\Asf_Z^2 \, Z] }_{=: \mathrm{III}(s)} \Big) \,\, .
\end{align*}
For auxiliary purposes, we first compute separately
\begin{align}
\partial_s (\Asf_Z \, Z) &= D_Z \Asf \, [\Asf_Z \, Z] \cdot Z + \Asf_Z^2 \, Z  \label{app eq: partial_s (Asf_Z Z)} \\
\partial_s (\Asf_Z^2 \, Z) &=  \Asf_Z  \big( D_Z \Asf \, [\Asf_Z \, Z] \cdot Z + \Asf_Z^2 \, Z \big) + (\partial_s \Asf_Z) \, \Asf_Z Z \notag\\
&\hspace*{-2mm}\overset{\text{\eqref{eq: exponetial decay - partial_s A_z}}}{=} \Asf_Z  \big( D_Z \Asf \, [\Asf_Z \, Z] \cdot Z + \Asf_Z^2 \, Z \big) + D_Z \Asf \, [\Asf_Z Z] \cdot \Asf_Z Z  \,\, . 
\label{app eq: partial_s (Asf^2_Z Z)}
\end{align}
We now derive $\mathrm{I}$ with the help of \eqref{app eq: partial_s (Asf_Z Z)}:
\begin{align*}
    \partial_s \mathrm{I} &= D_Z^3 \Asf \, [\Asf_Z Z , \, \Asf_Z Z , \, \Asf_Z Z] + 2 \, D^2_Z \Asf \, [ \partial_s(\Asf_Z Z) , \, \Asf_Z Z] \\
    &= D_Z^3 \Asf \, [\Asf_Z Z , \, \Asf_Z Z , \, \Asf_Z Z] + 2 \, D_Z^2 \Asf \, [D_Z \Asf \, [\Asf_Z Z] \cdot Z  , \, \Asf_Z Z] + 2 \, D_Z^2 \Asf \, [A_Z^2  Z , \, \Asf_Z Z] \,\, .
\end{align*}
Using \eqref{app eq: partial_s (Asf_Z Z)} again, we next derive $\mathrm{II} $
\begin{align*}
    \partial_s \mathrm{II} &= (\partial_s D_Z \Asf) \, [D_Z \Asf \, [\Asf_Z Z] \cdot Z] + D_Z \Asf \, \Big[ D_Z \Asf \, [\Asf_Z Z]  \cdot \partial_s Z \Big] + D_Z \Asf \, \Big[ \partial_s \Big(D_Z \Asf \, [\Asf_Z Z]  \Big) \cdot Z \Big]   \\
    &= D_Z^2 \Asf \, \Big[\Asf_Z Z , \,  D_Z \Asf \, [\Asf_Z Z] \cdot Z \Big] +   D_Z \Asf \, \Big[ D_Z \Asf \, [\Asf_Z Z]  \cdot \Asf_Z Z \Big] \\
    &\hspace*{4mm}+ D_Z \Asf \, \Big[ D_Z^2 \Asf \, [\Asf_Z Z, \, \Asf_Z Z] \cdot Z\Big] +   D_Z \Asf \, \Big[ D_Z \Asf \, \Big[ D_Z \Asf \, [\Asf_Z Z] \cdot Z + \Asf_Z^2  Z \Big]   \cdot Z\Big] 
\end{align*}
and finally, using \eqref{app eq: partial_s (Asf^2_Z Z)}, also compute
\begin{align*}
    \partial_s \mathrm{III} &= (\partial_s D_Z \Asf) \, [\Asf_Z^2 \, Z] + D_Z \Asf \, [\partial_s (\Asf_Z^2 \, Z)] \\
    &= D_Z^2 \Asf \, [\Asf_Z Z , \, \Asf_Z^2  Z] + D_Z \Asf \, \Big[ \Asf_Z  \big( D_Z \Asf \, [\Asf_Z Z] \cdot Z \big) + \Asf_Z^3 Z + D_Z \Asf \, [\Asf_Z Z] \cdot \Asf_Z Z \Big] \,\, .
\end{align*}
Thus $\mathsf{C}_{Z(s)} = \partial_s \mathrm{I}(s) + \partial_s \mathrm{II}(s) + \partial_s \mathrm{III}(s) $ in view of the definition of $\mathsf{C}$ in \eqref{eq: operator C_x}.

\subsection{The operators $\mathsf{C}^1_x$ and $\mathsf{C}^2_x$}
\label{app subsec: operators C^1 and C^2}

\noindent We give explicit formulae for the loop-depending operators $\mathsf{C}^1_x \comma \mathsf{C}^2_x \in \mathcal{L}(L^2,L^2)$ appearing in \eqref{eq: partial_s P Z = C^1_Z(QZ) + C^2_Z(partial_s QZ)}. For $x \in \Ucalxzerobar$ and $X \in L^2$ set
\begin{align}
\mathsf{C}_x^1 (X) &:= \Psf \Big( J_0(t) \cdot \big(S_0(t) - S(\shift{+}(x)) \big) \cdot z''( X) \Big)
\label{app eq: operator C^1_x} \\
\mathsf{C}_x^2 (X) &:= \Psf \Big( \big(J_0(t) - J_t(\shift{+}(x))\big) \cdot J_t(\shift{+}(x)) \cdot X \Big) \,\, .
\label{app eq: operator C^2_x}
\end{align}
Here the arguments of $\Psf$ are loops in $L^2$ themselves.
Clearly, for $j=1,2 \,$, the map 
\[
\Ucalxzerobar \rightarrow \mathcal{L}(L^2,L^2) \comma \quad x \mapsto \mathsf{C}^j_x \,\, ,
\] 
is continuous, where as usual $\Ucalxzerobar \subseteq \Cinfty(\bbS^1,\R^{2n})$ carries the $\Cinfty$-topology.

\subsection{Eliminating $s$-derivatives via the Floer equation}
\label{app subsec: Eliminate s-derivatives via Floer eq}

We show how to transform $\partial_s (S(u))$ into a term just depending on the loop $\shift{-}(u_s) = Z(s)$ by using the abstract Floer equation \eqref{eq: abstract Floer eq}. The $s$-derivative in $\partial_s (J_t(u))$ can be eliminated analogously. Recall that $\partial_s u =  \partial_s \Bar{u} = \partial_s Z \,$. We compute
\begin{align*}
    \partial_s (S \circ u) \, (s,t) &= d S_{u(s,t)} \cdot \partial_s u (s,t)= (dS)_{\shift{+}(Z(s))(t)} \cdot (\partial_s Z (s))(t) \\
    &= (dS)_{\shift{+}(Z(s))(t)} \cdot (\Asf_{Z(s)}  Z(s))(t)
\end{align*}
so that
\[
\partial_s (S \circ u) (s,\,\cdot\,) = (dS)_{\shift{+}(Z(s))} \cdot (\Asf_{Z(s)}  Z(s)) \,\, .
\]
The loop on the right-hand side is the evaluation at $Z(s)$ under the continuous map
\begin{align*}
   \Cinfty(\bbS^1,\R^{2n}) \supseteq \Ucalxzerobar \rightarrow \Cinfty(\bbS^1, \R^{2n \times (2n-d)}) \comma \quad x \mapsto (dS)_{\shift{+}(x)} \cdot (\Asf_{x} \, x) \,\, .
\end{align*}
By continuity, we can choose $\Ucalxzerobar$ smaller, so that we have a uniform $C^0(\bbS^1,\R^{2n \times (2n-d)})$-norm bound for $(dS)_{\shift{+}(x)} \cdot (\Asf_x x)$ for all $x \in \Ucalxzerobar \,$.

\printbibliography

\end{document}